%% file: paper.tex
\documentclass[russian,british]{amsart}
\usepackage[utf8]{inputenc}         
\usepackage[T1,T2A]{fontenc}        
\usepackage{babel}
\usepackage{amssymb}                
\usepackage{mathtools}              
\usepackage{thmtools,thm-restate}   
\usepackage[inline]{enumitem}       
\usepackage{graphicx}               
\usepackage{cite}                   
\usepackage[unicode]{hyperref}      
\usepackage[english,capitalise,nameinlink,noabbrev]{cleveref}
\usepackage{tikz-cd}                
\usetikzlibrary{babel}              
\usepackage{chngcntr}               
\usepackage{printlen}               
\usepackage{xparse}                 
\usepackage{mleftright}             
\usepackage[symbol]{footmisc}       
\usepackage{ifthen}                 %
\usepackage{fullpage}               

\makeatletter
\AtBeginDocument{\check@mathfonts}
\makeatother

\hypersetup{
	colorlinks=true,
	linkcolor=blue,
	citecolor=green,
	urlcolor=black
}

\newcounter{common}[section]
\newcounter{cdiagram}           
\makeatletter
	\let\c@equation=\c@common
	\let\c@subsection=\c@common
	\let\c@cdiagram=\c@common
	\let\c@paragraph\undefined
\makeatother
\counterwithout{subsection}{section}
\counterwithout{subsubsection}{subsection}

\newenvironment{cdiagram}{
	\equation \addtocounter{equation}{-1} \refstepcounter{cdiagram} \tikzcd
}{
	\endtikzcd \endequation
}
\newenvironment{cdiagram*}{
	\[ \tikzcd
}{
	\endtikzcd \]
}

\declaretheoremstyle[
	spaceabove=\medskipamount, spacebelow=\medskipamount, headfont=\normalfont\bfseries,
	bodyfont=\itshape, notefont=\mdseries, notebraces={(}{)},
	headformat=margin, postheadspace={ }, headpunct={.}, qed={\huge $\lrcorner$}
]{mythm}
\declaretheorem[sibling=common, style=mythm]{lemma,theorem,proposition,corollary}
\declaretheoremstyle[
	spaceabove=\medskipamount, spacebelow=\medskipamount, headfont=\normalfont\bfseries,
	bodyfont=\normalfont, notefont=\mdseries, notebraces={(}{)},
	headformat=margin, postheadspace={ }, headpunct={.}, qed={\huge $\lrcorner$}
]{mydef}
\declaretheorem[sibling=common, style=mydef]{definition,example}
\declaretheoremstyle[
	spaceabove=\medskipamount, spacebelow=\medskipamount, title={}, headfont=\normalfont,
	bodyfont=\normalfont, notefont=\mdseries, notebraces={(}{)},
  headformat=margin, postheadspace=0pt, headpunct={}, qed={\huge $\lrcorner$}
]{mypar}
\let\paragraph\undefined
\declaretheorem[sibling=common, style=mypar]{paragraph}

\newcommand*{\eqlap}[1]{\llap{(#1)}}
\newtagform{margin}[\eqlap]{\hspace{-4pt}}{}
\usetagform{margin}

\creflabelformat{equation}{#2#1#3}
\crefname{subsubsection}{Subsubsection}{Subsubsections}
\crefname{paragraph}{Paragraph}{Paragraphs}
\crefname{cdiagram}{Diagram}{Diagrams}

\def\p[#1]_#2{
	\setbox0=\hbox{$\scriptstyle{#2}$}
	\setbox2=\hbox{$\displaystyle{#1}$}
	\setbox4=\hbox{${}'\mathsurround=0pt$}
	\dimen0=.5\wd0 \advance\dimen0 by-.5\wd2
	\ifdim\dimen0>0pt
	\ifdim\dimen0>\wd4 \kern\wd4 \else\kern\dimen0\fi\fi
	\mathop{{#1}'}_{\kern-\wd4 #2}
}
\def\sump_#1{\p[\sum]_{#1}}
\def\prodp_#1{\p[\prod]_{#1}}
\def\minp_#1{\p[\min]_{#1}}
\def\maxp_#1{\p[\max]_{#1}}
\newcommand*{\NN}{\mathbb{N}}
\newcommand*{\NNO}{\mathbb{N}_0}
\newcommand*{\ZZ}{\mathbb{Z}}

\newcommand*{\CC}{\mathbb{C}}
\newcommand*{\CCi}{\mathbb{C}_\infty}
\newcommand*{\FF}{\mathbb{F}}
\newcommand*{\PP}{\mathbb{P}}
\newcommand*{\LL}{\mathcal{L}}
\renewcommand{\AA}{\mathbb{A}}
\newcommand*{\cJ}{\mathcal{J}}
\newcommand*{\finadele}{\AA_F^{fin}}
\newcommand*{\invertadele}{\bigl(\finadele\bigr)^{\!\times}}
\newcommand*{\fp}{\mathfrak{p}}
\DeclareMathOperator{\WeakMF}{\mathbf{Weak}}
\DeclareMathOperator{\StrongMF}{\mathbf{Strong}}
\DeclareMathOperator{\CuspMF}{\mathbf{Cusp}}
\newcommand*{\cw}{\mathcal{W}}
\newcommand*{\cm}{\mathcal{M}}
\newcommand*{\cs}{\mathcal{S}}
\newcommand*{\LLi}[2]{\protect\overleftarrow{\LL_{#1}^{#2}}}
\newcommand*{\LLNRi}{\LLi{N}{r}}
\newcommand*{\longto}{\relbar\joinrel\to}
\newcommand*{\isoto}{\xrightarrow{\mathmakebox[1.4ex]{\sim}}}
\newcommand*{\longisoto}{\xrightarrow{\mathmakebox[2.5ex]{\sim}}}
\newcommand*{\into}{\hookrightarrow}
\newcommand*{\longinto}{\lhook\joinrel\longrightarrow}
\newcommand*{\onto}{\twoheadrightarrow}
\newcommand*{\longonto}{\relbar\joinrel\onto}
\newcommand*{\ionto}{\lhook\joinrel\twoheadrightarrow}
\newcommand*{\longionto}{\lhook\joinrel\longonto}
\newcommand*{\xinto}[2][]{\xhookrightarrow[#1]{#2}}
\newcommand*{\xonto}[2][]{\xrightarrow[#1]{#2}\mathrel{\mkern-14mu}\rightarrow}
\newcommand*{\xionto}[2][]{\xhookrightarrow[#1]{#2}\mathrel{\mkern-22mu}\rightarrow\mathrel{\mkern6mu}}

\DeclareMathOperator{\End}{End}

\DeclareMathOperator{\D}{D}
\DeclareMathOperator{\Cl}{Cl}
\let\Im=\relax \DeclareMathOperator{\Im}{Im}
\DeclareMathOperator{\Sur}{Sur}
\DeclareMathOperator{\Inj}{Inj}
\DeclareMathOperator{\Bij}{Bij}
\DeclareMathOperator{\Free}{Free}
\DeclareMathOperator{\GLL}{GL}
\newcommand*{\GL}[2][]{\GLL_{#1}\parens{#2}}

\DeclareMathOperator{\bd}{\partial}

\DeclareMathOperator{\dd}{\operatorname{d}}

\newcommand*{\dbd}[2][]{\frac{\dd^{#1}}{\dd\!#2^{#1}}}
\mleftright	
\newcommand*{\set}[1]{\left\{#1\right\}}
\newcommand*{\setst}[2]{\left\{#1\:\middle|\:#2\right\}}

\newcommand*{\norm}[1]{\left\lVert 1\right\rVert}
\newcommand*{\abs}[1]{\left\lvert#1\right\rvert}

\newcommand*{\parens}[1]{\left(#1\right)}

\newcommand*{\brackets}[1]{\left[#1\right]}

\newcommand*{\size}{\#}
\newcommand*{\sizep}[1]{\size\parens{#1}}

\newcommand*{\lquotient}[2]{\left. #1 \middle\backslash #2 \right.}
\newcommand*{\rquotient}[2]{\left. #1 \middle/ #2 \right.}
\newcommand*{\lrquotient}[3]{\left. #1 \middle\backslash #2 \middle/ #3 \right.}
\newcommand*{\cin}[2]{\bigl[{#1}^{#2}\bigr]}

\newcommand*{\littleo}[1]{o\parens{#1}}
\newcommand*{\pcoord}[1]{%
  \begingroup\lccode`~=`: \lowercase{\endgroup
  \edef~}{\mathbin{\mathchar\the\mathcode`:}\nobreak}%
  \left(
  \begingroup
  \mathcode`:=\string"8000
  #1%
  \endgroup
  \right)
}
\newcommand*{\lrsptext}[1]{\quad\text{#1}\quad}
\newcommand*{\lsptext}[1]{\quad\text{#1}\ }
\newcommand*{\rsptext}[1]{\ \text{#1}\quad}
\newcommand*{\sptext}[1]{\ \text{#1}\ }
\newcommand*{\ie}{i.e.\ }

\newcommand*{\defeq}{\coloneqq}
\newcommand*{\eqdef}{\eqqcolon}

\begin{document}

\title[Lattice Drinfeld modular forms]{Drinfeld modular forms of higher rank from a lattice-oriented point of view}
\author{Liam Baker}
\address{
  Mathematics Division\\
  Department of Mathematical Sciences\\
  University of Stellenbosch\\
  Stellenbosch\\
  7599\\
  South Africa}
\email{liambaker@sun.ac.za}

\date{\today}


\begin{abstract}	
  A space $\LL_N^r$ of Drinfeld modules of rank $r \geq 1$ with level structure, or equivalently lattices of rank $r$ with level structure, is introduced, and its irreducible components and group actions on it are investigated.
  A metric is defined on this space, its completion $\LLNRi$ is established and the aforementioned group actions are extended to the completion.
  A decomposition of the completion into multiple smaller spaces $\LL_N^s$ is proven.
  Drinfeld modular forms are defined as homogeneous holomorphic functions on $\LL_N^r$ which are continuous on the completion $\LLNRi$, and the group actions above are extended to actions on the spaces of modular forms.
  Finally, the modular forms defined here are compared with those of Basson, Breuer, and Pink and with those of Gekeler, and it is shown that the cusp forms (those which are zero on the boundary) coincide.
\end{abstract}
 
\maketitle

\tableofcontents

\include{text/0__intro}

\include{text/1__DModules_lattices}

\include{text/2__lattice_modules}

\include{text/3__lattices_2}

\include{text/4__modular_forms}


\include{text/6__conclusion}

\bibliographystyle{amsalpha}
\bibliography{phdthesis}


\end{document}

%% file: text/0__intro.tex
\addtocounter{section}{-1}
\section{Introduction} \label{sec:intro}

\subsection{Literature overview}

In the beginning, Drinfeld defined what he called \emph{elliptic modules} to prove a special case of the Langlands conjecture for $\GLL_2$ over function fields \cite{drinfeld1974english}.
These modules are now called \emph{Drinfeld modules}, and are similar to elliptic curves, although they have arbitrarily high rank $r \in \NN$.
In particular, Drinfeld constructed a moduli space of Drinfeld modules of rank $r$ with level structure both as an algebraic variety and analytically as a double quotient of an $r-1$-dimensional space $\Omega^r$, which is a rigid analytic space over a field $\CCi$ of positive characteristic.

There is, however, a natural definition of a \emph{Drinfeld modular form} on $\Omega^r$ with values in $\CCi$ as given by Goss in \cite{goss1980eisenstein}; these can be defined algebraically à la Katz \cite{katz1973p} and analytically in analogy with classical modular forms, with $\Omega^r$ playing the role of the complex upper half plane.

In the case of rank $2$, these modular forms are functions of one variable and are in closest analogy with classical modular forms, which only exist in rank $2$.
The bulk of the study of Drinfeld modular forms has thus focused on this case; for surveys of the developments in this area, see \cite{gekeler2006dmc,cornelissen1997survey,gekeler1999survey}.

In arbitrary rank, the next development was due to Kapranov \cite{kapranov1988english} who constructed a compactification of the moduli variety of Drinfeld $\FF_q[T]$-modules with level structure, which he used to prove finite dimensionality of the space of Drinfeld modular forms of any particular weight, as in \cite{goss1992integrals}.

More recently, Basson, Breuer, and Pink wrote a series of papers \cite{BBPI,BBPII,BBPIII} culminating in the memoir \cite{BBP}, establishing a theory of modular forms of arbitrary rank, building on the papers \cite{Pink2013compactification,breuer2009drinfeld,basson2017product,basson2017certain} and Basson's PhD paper \cite{basson2014coefficients}, followed by Pink's \cite{Pink2019Areciprocal}.
In parallel, Gekeler has developed a theory of modular forms of arbitrary rank for the case of the simplest base ring $\FF_q[T]$ in the series \cite{gekelerHigherI,gekelerHigherII,gekelerHigherIII,gekelerHigherIV,gekelerHigherV,gekelerHigherVI,gekelerHigherVII,gekeler2026expansions}, where the connection with the Bruhat-Tits building is greatly used.

For a more detailed discussion of the history of Drinfeld modular forms of higher rank, see \cite[Section 7]{basson2017certain}.

\subsection{Motivation}

In the classical case of modular forms, the simplest case is that of Eisenstein series. For a lattice $\Lambda = \ZZ \omega_1 +\ZZ \omega_2 \subset \CC$ and integer $k > 2$, the $k$th Eisenstein series is defined by
\[ E^k(\Lambda) = \sump_{\lambda \in \Lambda} \lambda^{-k} = \sump_{m,n \in \ZZ} \frac{1}{(m\omega_1+n\omega_2)^k} \,.\footnote{Here and in what follows, primed sum $\sump_{\lambda \in \Lambda}$ and product $\prodp_{\lambda \in \Lambda}$ operators denote a sum or product over all nonzero elements of the index set $\Lambda$.} \]
This series is most conceptually simply viewed as a homogeneous function of lattices, which is holomorphic in a suitable sense.
However, it is most often normalised in the literature to $\omega_2 = 1$ using the homogeneity property, resulting in a function of one complex variable $\omega_1$, which can then be studied using the well developed theory of complex functions.
The homogeneity condition then transforms into a restriction of the behaviour of this function of one variable under the action of $\GL[2]{\ZZ}$.

In the Drinfeld `upper half plane', we have lattices of arbitrarily high rank $r$:
\[ \Lambda = A\omega_1 +A\omega_2 +\dotsb +A\omega_r\,. \]
Eisenstein series and other modular forms can be defined similarly as functions of lattices, but if we normalise to make the last component $\omega_r = 1$ as before, we have a function of $r-1$ variables, which is not as easily dealt with as the case $r = 2$.
Other work that has been done on Drinfeld modules of higher rank (such as \cite{BBPI,BBPII,BBPIII,BBP} and \cite{gekelerHigherI,gekelerHigherII,gekelerHigherIII,gekelerHigherIV,gekelerHigherV,gekelerHigherVI,gekelerHigherVII}) has been done from this perspective of functions of $r-1$ variables, whereas this paper investigates modular forms as functions on the space of lattices.

As it is not as easily seen that the space of lattices can be given rigid analytic structure (so that one can reasonable speak of a holomorphic or analytic function on such a space), we first link this space to earlier work to carry over rigid analyticity proven there into our setting.

This viewpoint yields some unexpected rewards:
\begin{itemize}
  \item We define our modular forms of higher rank in a relatively `low-tech' way, \ie largely avoiding the use of modern algebraic geometry and instead using tools such as metric spaces.
  This may help those who wish to enter and make progress in this field without experience in algebraic geometry.
  \item We also find actions of the group $\cJ(A)$ of fractional ideals of $A$ and the general linear group $\GL[r]{A/N}$ for $N$ an ideal of $A$ on the spaces of modular forms of rank $r$.
  The latter action subsumes that of $\GL[r]{A}$, which specialises to the subgroup of $\GL[r]{A/N}$ with determinant in the base field $\FF_q$, and that of $\parens{A/N}^\times$, being the Galois group of the field $F(\zeta_N)$ of $F$ with $N$-division points of the Carlitz module adjoined.
\end{itemize}

\subsection{Main results}

Our main results show that the modular forms defined in this paper are roughly equivalent to those defined by other authors:
firstly, those of Basson, Breuer, and Pink, in that there is a bijection between our cusp forms and tuples of BBP cusp forms:
\begin{restatable*}{theorem}{cuspFormsToCompsBijective} \label{thm:cuspFormsToComps_Bijective}
  If $H$ is a set of representatives in $\GL[r]{\finadele}$ for the double quotient
  \[ \lrquotient{\GL[r]{F}}{\GL[r]{\finadele}}{K(N)}, \]
  then $f \mapsto (f_g)_{g \in H}$ is a bijection between the spaces $\CuspMF_N^r$ and $\prod_{g \in H} \cs_*(\Gamma_g)$.
\end{restatable*}
Here $\finadele$ is the ring of finite adeles of $F$, $K(N)$ is the principal congruence subgroup of $\GL[r]{A/N}$ for $N$ an ideal of $A$, $\CuspMF_N^r$ is our space of cusp forms (defined later), and the $\cs_*(\Gamma_g)$ are spaces of BBP cusp forms.

Secondly, we have a similar result relating our modular forms and Gekeler's:


\begin{restatable*}{theorem}{toGekelerLevelN}
    Let $(Y_c)_{c \in \Cl(F)}$ be a tuple of $A$-lattices in $F^r$ with $\pi(Y_c) = c$ for each $c \in \Cl(F)$, let $k$ be a nonnegative integer, let $N$ be an ideal of $A$, and for each $A$-lattice $Y \subset F^r$ let $\Bij(Y)$ be the set of $A$-linear bijections $N^{-1}Y/Y \ionto (N^{-1}/A)^r$.
    Then there is an injection
    \begin{align*}
        \StrongMF_N^{k,r} &\longinto \prod_{c \in \Cl(F)} \prod_{t \in \operatorname{Bij}(Y_c)} \mathbf{Mod}_{k}(\GL{Y_c}(N)) \\
        f : \LLNRi \to \CCi &\longmapsto \left(i \mapsto f(\Xi_{Y_c}(i), t \circ i^{-1})\right)_{c \in \Cl(F), t \in \operatorname{Bij}(Y_c)}
    \end{align*}
    from our modular forms of weight $k$ for $K(N)$ to a tuple of Gekeler modular forms of weight $k$ and level $N$.
    Moreover, when restricted to cusp forms, this is a bijection.
    
    These induce an injection and a bijection of the algebras
    \begin{align*}
        \StrongMF_N^r &\longinto \prod_{c \in \Cl(F)} \prod_{t \in \operatorname{Bij}(Y_c)} \mathbf{Mod}(\GL{Y_c}(N)) \lsptext{and} \\
        \CuspMF_N^r &\longionto \prod_{c \in \Cl(F)} \prod_{t \in \operatorname{Bij}(Y_c)} \mathbf{Mod}^{cusp}(\GL{Y_c}(N))
    \end{align*}
    of our modular forms to Gekeler modular forms respectively.
\end{restatable*}
Here, $\StrongMF_N^r$ and $\CuspMF_N^r$ are our spaces of modular forms and cusp forms of level $N$ respectively, $\Cl(F)$ is the ideal class group of $F$, $\pi$ is the projection map $I_1b_1+\dotsb I_rb_r \mapsto [I_1\dotsm I_r]$ from the space of projective $A$-lattices to $\Cl(F)$, and $\mathbf{Mod}(\Gamma)$ and $\mathbf{Mod}^{cusp}(\Gamma)$ denote Gekeler's space of modular forms and cusp forms for a subgroup $\Gamma < \GL[r]{F}$.

In addition, we define actions of $\GL[r]{A/N}$ and the group $\cJ(A)$ of $A$-fractional ideals of $F$ on the spaces of lattices with and without level structure, which we then carry over to actions on the spaces of modular forms.

\subsection{Outline of the paper}

In \cref{sec:DModules_lattices}, we present an abridged introduction to Drinfeld modules and lattices (with and without level structure), presenting only the results necessary in later sections. We also introduce the exponential function associated to a lattice, and mention some of its analytic properties.
Finally we present the well-known equivalence between lattices and Drinfeld modules, both with their level structures.

In \cref{sec:lattices_Dmodules} we first present the realisation of the \emph{moduli space} of lattices of rank $r$ with level structure as a double quotient involving the Drinfeld period domain $\Omega^r$ and the ring $\finadele$ of finite adeles originally proven by Drinfeld; we use this to derive a similar realisation for the space of lattices with level structure \emph{itself}.
We then investigate the decomposition of these spaces into irreducible components as detailed by Hubschmid, extending some results about the identification of these components, especially the `identity component'.
In the final two subsections of this section we investigate the actions of the general linear group $\GL[r]{\hat{A}}$ of profinite integers and the group $\invertadele$ of invertible adeles on our spaces, which we specialise to their quotient actions of $\GL[r]{A/N}$, for $N$ an ideal of $A$, and $\cJ(A)$, the group of $A$-fractional ideals of $F$.

In \cref{sec:lattices_ii}, we first introduce metrics on the spaces of lattices with and without level structure and investigate their completions, as well as those of their irreducible components.
Then we characterise these completions as unions of similar spaces of smaller rank, and also extend the group actions of $\GL[r]{A/N}$ and $\cJ(A)$ defined earlier to these completions.

In \cref{sec:modular_forms}, we first define weak and strong modular forms and cusp forms as homogeneous functions on the spaces of lattices with and without level structure, as well as the graded algebras consisting of these functions, and carry over the aforementioned group actions to actions on these spaces of modular forms.
We then list some examples of modular forms and detail their behaviour under the above group actions.
Finally we investigate the relation between our modular forms and those defined by Basson, Breuer, and Pink, showing our modular forms to be a large subset of theirs, and similarly to those defined by Gekeler.


Finally, in \cref{sec:conclusion} we present some closing remarks, including possible extensions to this work.

%% file: text/1__DModules_lattices.tex
\section{Drinfeld modules and lattices} \label{sec:DModules_lattices}

Let $F$ be a fixed global function field with prime characteristic $p$ and field of constants $\FF_q$ where $q$ is a power of $p$.
Let $\infty$ be a fixed place of $F$ with degree $\delta$, let $\pi$ be a fixed uniformising parameter for $F$ at $\infty$, let $\deg$ be the degree function on $F$ determined by $\deg\pi = -\delta$, and let $\abs{\cdot}$ be the absolute value defined by $\abs{x} = q^{\deg{x}}$ on $F$.
Finally, let $F_\infty$ denote the metric completion of $F$ with respect to $\abs{\cdot}$, and let $\CCi$ denote the metric completion of an algebraic closure of $F_\infty$.

\input{text/1_1_Dmodules.tex}

\input{text/1_2_lattices.tex}

\input{text/1_3_lm_equiv.tex}

%% file: text/1_1_Dmodules.tex
\subsection{Drinfeld Modules}

\begin{definition} \label{def:End_Fq}
  We denote by $\End_{\FF_q}\parens{\CCi}$ the ring of $\FF_q$-linear endomorphisms of $\CCi$, with addition defined pointwise and multiplication defined as function composition.
  As a special element, we consider the \emph{Frobenius endomorphism}:
  \[ \tau \in \End_{\FF_q}\parens{\CCi},\quad X \mapsto X^q, \]
  and we also consider the subring of $\End_{\FF_q}\parens{\CCi}$ generated over $\CCi$ by $\tau$, which we denote as $\CCi\set{\tau}$, and a superring of that, the ring of formal power series in $\tau$ with coefficients in $\CCi$, which we denote by $\CCi\set{\set{\tau}}$.
  
  Each $f \in \CCi\set{\set{\tau}}$ can be uniquely written in the form $f = \sum_i l_i \tau^i$ (or equivalently $f(X) = \sum_i l_i X^{q^i}$); we then define $\D(f) = l_0$ (\ie the `constant term' of $f$) and $l(f) = l_{\deg f}$ for $f \in \CCi\set{\tau}$ (\ie the `leading coefficient' of $f$).
  This $D : \CCi\set{\set{\tau}} \to \CCi$ is a ring homomorphism.
\end{definition}

\begin{definition} \label{def:D_module}
  A Drinfeld module of integer rank $r \geq 0$ is a ring homomorphism
  \[ \phi : A \to \CCi\set{\tau},\quad a \mapsto \phi_a \]
  such that for each $a \in A$, both
  \begin{itemize}
  \item $\deg_\tau \phi_a = r \cdot \deg a$  (here $\deg_\tau \phi_a$ denotes the degree of $\phi_a$ in $\tau$), and
  \item $\D\parens{\phi_a} = a$.
  \end{itemize}

  For two Drinfeld modules $\phi$ and $\varphi$ of rank $r$, a morphism $u : \phi \to \varphi$ is an element of $\CCi\set{\tau}$ such that $u \phi_a = \varphi_a u$ for all $a \in A$; \ie the diagram
  \begin{cdiagram} \label{eq:Dmod_morph}
    \CCi
      \arrow{r}{\phi_a}
      \arrow{d}{u} &
    \CCi
      \arrow{d}{u} \\
    \CCi
      \arrow{r}{\varphi_a} &
    \CCi
  \end{cdiagram}
  commutes for each $a \in A$.
  A category of Drinfeld modules of rank $r$ is thus formed in the natural way.
\end{definition}

The only Drinfeld module of rank $r = 0$ is the `trivial' $\phi_a(X) = aX$.

Since $\tau$ is $\FF_q$-linear, so is each $\phi_a$ for a Drinfeld module $\phi$.
Hence each $f \in \FF_q^\times$ is an automorphism of any Drinfeld module $\phi$, considered as an element of $\CCi\set{\tau}$.
However, some Drinfeld modules have nontrivial automorphisms, so Drinfeld modules are augmented with \emph{level structure} to remove these nontrivial automorphisms.

\begin{definition} \label{def:N_division_points}
  Let $\phi$ be a Drinfeld module. For an element $a \in A$, we define the $a$-division points $\phi[a] = \ker \phi_a$, and for an ideal $N$ we define $\phi[N] = \bigcap_{a \in N} \phi[a]$.
\end{definition}

\begin{proposition} \label{prop:Dmodule_division_module}
  $\phi[N]$ has the structure of an $A/N$-module, defined by
  \[ a \cdot z = \phi_a(z) \lrsptext{for} a \in A \lrsptext{and} z \in \phi[N] \]
  and is isomorphic to $\parens{N^{-1}/A}^r$.
\end{proposition}


\begin{definition} \label{def:level_N_structure}
  For an ideal $N$ of $A$, a \emph{level $N$ structure} for a Drinfeld module $\phi$ of rank $r$ is an $A/N$-module isomorphism
  \[ \beta : \parens{N^{-1}/A}^r \longisoto \phi\brackets{N}. \qedhere \]
\end{definition}

%% file: text/1_2_lattices.tex
\subsection{Lattices}

\begin{definition} \label{def:lattice}
  An $A$-submodule $\Lambda \subset \CCi$ is called a lattice if and only if
  \begin{enumerate}
    \item $\Lambda$ is finitely generated as an $A$-module, and
    \item $\Lambda$ is strongly discrete as a subset of $\CCi$ (\ie any finite ball in $\CCi$ has finite intersection with $\Lambda$.)
  \end{enumerate}

  The \emph{rank} of $\Lambda$ is its rank as a finitely generated torsion-free (or equivalently finitely generated projective) submodule of $\CCi$.
  
  The set of lattices of rank $r$ is denoted $\LL^r$, the set of lattices of rank $\leq r$ is denoted $\LL^{\leq r}$, and the set of \emph{all lattices} is denoted $\LL$.
\end{definition}

\begin{definition} \label{def:prelattice}
  A prelattice is a strongly discrete $\FF_q$-sub-vector space of $\CCi$.
\end{definition}

Since $A$ is an $\FF_q$-vector space, any lattice is a prelattice.

\begin{paragraph} \label{para:latticeAmodule}
  Since $A$ is a Dedekind domain, we have from \cite[Section~4.3]{goss2012basic} that if $\Lambda$ is a lattice of rank $r \geq 1$, then there is an $A$-module isomorphism $\Lambda \simeq A^{r-1} \oplus I$ where $I$ is a nonzero ideal of $A$.
  In other words, there are $\omega_1, \omega_2, \dotsc, \omega_r \in \CCi$ such that $\Lambda = A\omega_1 +A\omega_2 +\dotsb +A\omega_{r-1} +I\omega_r$ and the $\omega_i$ are $F_\infty$-linearly independent, since $F$ is the fraction field of $A$ and $\Lambda$ is strongly discrete.
\end{paragraph}

\begin{definition} \label{def:lattice_morph}
  A morphism $c : \Lambda_1 \to \Lambda_2$ between two lattices \emph{of the same rank} is an element $c \in \CCi$ such that $c\Lambda_1 \subseteq \Lambda_2$.
  The category of lattices is then defined in the natural way.
\end{definition}

\begin{paragraph} \label{para:lattice_morph_inv}
  Note that for a morphism $c : \Lambda_1 \to \Lambda_2$ to have an inverse morphism $c'$ in this category, we must have that $c c' = 1$ and $c\Lambda_1 \subseteq \Lambda_2$ and $c'\Lambda_2 \subseteq \Lambda_1$ (or equivalently $\Lambda_2 \subseteq c\Lambda_1$); thus $\Lambda_2 = c\Lambda_1$.
\end{paragraph}

Every lattice has $\FF_q^\times$ as a set of trivial automorphisms, but some special lattices have nontrivial automorphisms; for instance, if $f \in \FF_{q^2}-\FF_q \subset \CCi$, then $A +fA$ is a rank 2 lattice with $f$ as an automorphism, since $f$ satisfies a quadratic equation with coefficients in $\FF_q \subset A$.
So, similarly to Drinfeld modules, they are augmented with \emph{level structure} to remove these nontrivial automorphisms.

\begin{paragraph} \label{para:lattice_quot}
Note that if $c : \Lambda_1 \to \Lambda_2$ is a nonzero morphism of lattices, then since $c\Lambda_1$ and $\Lambda_2$ have the same rank, $\Lambda_2/c\Lambda_1$ is a finite $A$-module.
Also, for any nonzero $a \in A$ and lattice $\Lambda$, $a$ is a morphism from $\Lambda$ to itself.
\end{paragraph}

Moreover, if $\Lambda$ has rank $r$, then by \cite[p.~67]{goss2012basic} we have that
\[ \Lambda/a\Lambda \simeq a^{-1}\Lambda/\Lambda \simeq \bigoplus_{i = 1}^r A/(a) \]
is a finite $A/(a)$-module, and more generally if $N$ is a nonzero ideal of $A$ then
\[ \Lambda/N\Lambda \simeq N^{-1}\Lambda/\Lambda \simeq \bigoplus_{i = 1}^r A/N \]
is a free finite $A/N$-module.
We can thus make the following

\begin{definition} \label{def:lattice_level}
  For a nonzero proper ideal of $A$, a \emph{level $N$ structure} for a lattice $\Lambda$ of rank $r$ is an $A/N$-module isomorphism
  \[ \alpha : \parens{N^{-1}/A}^r \longisoto N^{-1}\Lambda/\Lambda. \qedhere \]
\end{definition}

\begin{proposition} \label{prop:lattice_levelN_GLrAN}
  Every rank $r$ lattice has exactly $\size{\GL[r]{A/N}}$ level $N$ structures.
\end{proposition}

\begin{definition} \label{def:llattice_morph}
  A morphism $c : (\Lambda_1,\alpha_1) \to (\Lambda_2,\alpha_2)$ between two lattices of rank $r$ with associated level $N$ structures is an element $c \in \CCi$ such that both $c\Lambda_1 \subseteq \Lambda_2$ \emph{and} the following diagram commutes:
  \begin{cdiagram} \label{diag:llattice_morph}
    N^{-1}\Lambda_1/\Lambda_1
      \arrow[r, hook, two heads, "\times c"] &
    N^{-1}c\Lambda_1/c\Lambda_1
      \arrow[d, "\subseteq"] \\
    \parens{N^{-1}/A}^r
      \arrow[u, hook, two heads, "\alpha_1"]
      \arrow[r, hook, two heads, "\alpha_2"'] &
    N^{-1}\Lambda_2/\Lambda_2 \qedhere
  \end{cdiagram}
\end{definition}

Note that since the sets in the above diagram are finite, if we have such a morphism then the map on the right, induced from the inclusion of $c\Lambda_1$ in $\Lambda_2$, should also be a bijection.
For this to be the case, we must have that if $\lambda \in c\Lambda_1$ is not in $cN\Lambda_1$, then it is not in $N\Lambda_2$; \ie $N\Lambda_2 \cap c\Lambda_1 \subseteq cN\Lambda_1$.
The reverse inclusion is apparent, so we must in fact have equality.

\begin{proposition} \label{prop:llattice_isomorph}
  If $c : \parens{\Lambda_1,\alpha_1} \isoto \parens{\Lambda_2,\alpha_2}$ is an isomorphism of lattices with level $N$ structure, then $c\Lambda_1 = \Lambda_2$ and $c\alpha_1 = \alpha_2$, with $c$ considered as an element of $\CCi$.
\end{proposition}

\begin{corollary} \label{coro:llattice_automorph}
  If $c : \parens{\Lambda,\alpha} \isoto \parens{\Lambda,\alpha}$ is an automorphism of lattices with level $N$ structure, then $c$ is the identity morphism.
\end{corollary}
  

There is a special function in $\End_{\FF_q}\parens{\CCi}$ associated to each prelattice (\ie strongly discrete $\FF_q$-sub-vector space of $\CCi$; thus also to each lattice), as follows:
\begin{definition} \label{def:e_Lambda}
  For a prelattice $\Lambda$, the \emph{exponential function} is
  \[ e_\Lambda(z) = z \cdot \prodp_{\lambda \in \Lambda} \parens{1-\frac{z}{\lambda}}\,. \]
  This product converges, as $\abs{\lambda} \to \infty$ if $\Lambda$ is infinite since $\Lambda$ is strongly discrete.
\end{definition}

We collect the following properties of $e_\Lambda$, which we refrain from proving; for proofs, see \cite[Section~4.3]{goss2012basic}, \cite[Section~2.2]{gekeler2006dmc}, and \cite[Chapter~2]{BBPI}.
\begin{proposition}~ \label{prop:e_Lambda_props}  
  \begin{enumerate}
    \item $e_\Lambda$ is an entire function on $\CCi$ with simple zeroes at and only at $\Lambda$.
    \item $\dbd{z} e_\Lambda(z) = 1$.
    \item $\displaystyle \frac{1}{e_{\Lambda}(z)} = \sum_{\lambda \in \Lambda} \frac{1}{z-\lambda}$.
    \item $e_\Lambda$ is $\FF_q$-linear.
    \item For $c \in \CCi$, $e_{c\Lambda}(cz) = c\cdot e_\Lambda(z)$.
    \item $e_\Lambda$ has a power series of the following form, convergent for all $z \in \CCi$:
    \[ e_\Lambda(z) = \sum_{i \in \NNO} e_{\Lambda,i} z^{q^i}. \]
    \item If $\Lambda_1 \subseteq \Lambda_2$, then $e_{\Lambda_1}(\Lambda_2)$ is also a prelattice, and
    \[ e_{\Lambda_2}(z) = e_{e_{\Lambda_1}(\Lambda_2)}(e_{\Lambda_1}(z)). \qedhere \]
  \end{enumerate}
\end{proposition}

The following property of $e_\Lambda$ concerns its behaviour `at infinity':
\begin{proposition} \label{prop:e_Lambda_infinity}
  For variable $z \in \CCi$ and a prelattice $\Lambda$,
  \[ \dd(z,\Lambda) \to \infty \iff \abs{e_\Lambda(z)} \to \infty\,.\footnote{Here and elsewhere, $\displaystyle \dd(x,S) = \inf_{s\in S}\abs{x-s}$ denotes the distance from a point $x$ to the set $S$.} \qedhere \]
\end{proposition}

\begin{corollary} \label{coro:e_Lambda_infinity}
  Let $\Lambda_1 \subseteq \Lambda_2$ be prelattices.
  Then for a variable $z \in \CCi$,
  \[ \dd(z,\Lambda_2) \to \infty \iff \dd(e_{\Lambda_1}(z),e_{\Lambda_1}(\Lambda_2)) \to \infty. \qedhere \]
\end{corollary}

The following property instead concerns the behaviour of $e_\Lambda$ close to zero:
\begin{proposition} \label{prop:e_Lambda_leqR}
  Let $\Lambda$ be a prelattice with $\minp_{\lambda \in \Lambda} \abs{\lambda} = R$.
  Then for $\abs{z} < R$ we have that $\abs{e_\Lambda(z)-z} \leq \abs{z}^q R^{1-q}$.
\end{proposition}

The following properties are valid for $\Lambda$ a lattice (\ie in addition to being a prelattice, also being an $A$-module of finite rank):
\begin{proposition} ~ \label{prop:e_Lambda_latticeProps}
  \begin{enumerate}
    \item For $a \in A$, $\displaystyle \quad e_\Lambda(az) = a \cdot e_\Lambda(z) \cdot \prodp_{\lambda \in a^{-1}\Lambda/\Lambda} \parens{1-\frac{e_\Lambda(z)}{e_\Lambda(\lambda)}}$.
    \item For an ideal $N$ of $A$, $\displaystyle \quad e_{N^{-1}\Lambda}(z) = e_\Lambda(z) \cdot \prodp_{\lambda \in N^{-1}\Lambda/\Lambda} \parens{1-\frac{e_\Lambda(z)}{e_\Lambda(\lambda)}}$. \qedhere
  \end{enumerate}
\end{proposition}

\begin{paragraph}
  Since $e_\Lambda$, being entire, is surjective, and
  \[ e_\Lambda(a) = e_\Lambda(b) \iff e_\Lambda(a-b) = 0 \iff a-b \in \Lambda, \]
  we see that $e_\Lambda$ is a bijection between $\CCi/\Lambda$ and $\CCi$.
  We will thus use the same notation $e_\Lambda(\lambda)$ both for $\lambda \in \CCi$ and for $\lambda \in \CCi/\Lambda$.
\end{paragraph}

%% file: text/1_3_lm_equiv.tex
\subsection{Equivalence between lattices and Drinfeld modules}

In the previous two subsections, we introduced Drinfeld modules and lattices.
As it turns out, these two concepts are intimately connected.
Firstly, from each lattice we can construct an associated Drinfeld module:

\begin{proposition} \label{prop:Dmod_from_lattice}
  If $\Lambda$ is a lattice of rank $r$, then $\phi^\Lambda$ defined as follows is a Drinfeld module:
  \[ \phi^\Lambda_a (X) = a \cdot X \prodp_{\lambda \in a^{-1}\Lambda/\Lambda} \parens{1-\frac{X}{e_\Lambda(\lambda)}} = a \cdot X \prodp_{y \in e_\Lambda\parens{a^{-1}\Lambda}} \parens{1-\frac{X}{y}}. \qedhere \]
\end{proposition}


A fundamental result is that, in fact, \emph{every} Drinfeld module arises from a lattice in this way.
We state the following result; for proof, see \cite[Theorem 4.6.9]{goss2012basic}:
\begin{theorem} \label{thm:Dmod_lattice_equiv}
  Let $\phi$ be a Drinfeld module of rank $r$ over $\CCi$.
  Then there is a lattice $\Lambda = \Lambda_\phi$ of rank $r$ such that $\phi = \phi^\Lambda$.
  Moreover, the association $\phi \mapsto \Lambda_\phi$ gives rise to an equivalence of categories between Drinfeld modules of rank $r$ and lattices of rank $r$.
\end{theorem}

There is also an equivalence in the definitions of level structure for Drinfeld modules and lattices, as shown in the following propositions:
\begin{proposition} \label{prop:Dmodule_division_lattice}
  If a Drinfeld module $\phi$ corresponds to a lattice $\Lambda$, then
  \begin{align*}
    \phi[a] &= \setst{e_\Lambda(\lambda)}{\lambda \in a^{-1}\Lambda/\Lambda} = e_\Lambda(a^{-1}\Lambda) \lsptext{and} \\
    \phi[N] &= \setst{e_\Lambda(\lambda)}{\lambda \in N^{-1}\Lambda/\Lambda} = e_\Lambda(N^{-1}\Lambda). \qedhere
  \end{align*}
\end{proposition}
\begin{proof}
  Follows from \cref{prop:Dmod_from_lattice}.
\end{proof}

In \cref{prop:Dmod_from_lattice}, we see the Drinfeld module polynomial $\phi^\Lambda_a(X)$ associated to a lattice $\Lambda$ factorised as a product over $a^{-1}\Lambda/\Lambda$.
In the same way, we can define Drinfeld module-associated polynomials for each ideal $N \subseteq A$:
\begin{definition} \label{def:IdealDModule}
  If $\Lambda$ is a lattice of rank $r$ and $N$ an ideal of $A$, we define the polynomial
  \[ \phi^\Lambda_N(X) = X \prodp_{\lambda \in N^{-1}\Lambda/\Lambda} \parens{1-\frac{X}{e_\Lambda(\lambda)}} = X \prodp_{y \in e_\Lambda(N^{-1}\Lambda)} \parens{1-\frac{X}{y}}. \qedhere \]
\end{definition}

\begin{proposition} \label{prop:IdealDModuleEffect}
  For $z \in \CCi$, $\phi^\Lambda_N(e_\Lambda(z)) = e_{N^{-1}\Lambda}(z)$.
\end{proposition}
\begin{proof}
  Apply \cref{prop:e_Lambda_latticeProps}.
\end{proof}

\begin{paragraph}
  Note that by \cref{prop:Dmod_from_lattice} and \cref{def:IdealDModule}, the series of polynomials $\phi^\Lambda_N(X)$ and $\phi^\Lambda_a(X)$ are related by $\phi^\Lambda_a(X) = a \cdot \phi^\Lambda_{(a)}(X)$.
\end{paragraph}

Note that the above equivalence between lattices and Drinfeld modules also extends to an equivalence between lattices with level structure and Drinfeld modules with level structure, as follows:
\begin{proposition} \label{prop:Dmod_lattice_level_equiv}
  If $(\Lambda,\alpha)$ is a lattice of rank $r$ with level $N$ structure, then $(\phi^\Lambda, e_\Lambda \circ \alpha)$ is a Drinfeld module of rank $r$ with level $N$ structure.
\end{proposition}

Thus $\GL[r]{A/N}$ acts from the right on the set of level $N$ structures of a given Drinfeld module in a similar way as in \cref{prop:lattice_levelN_GLrAN}.

%% file: text/2__lattice_modules.tex
\section{The space of Lattices/Drinfeld modules, and actions thereon} \label{sec:lattices_Dmodules}

\input{text/2_1_lm_moduli.tex}

\input{text/2_2_lm_irred.tex}

\input{text/2_3_lm_GLrAN.tex}

\input{text/2_4_lm_invadele.tex}

%% file: text/2_1_lm_moduli.tex
\subsection{The Drinfeld moduli space}

\begin{paragraph} \label{para:DblQuotIso}
  Let
  \[ K(N) = \ker\parens{\GL[r]{\hat{A}} \onto \GL[r]{A/N}} \]
  denote the principal congruence subgroup of level $N$, where
  \[ \hat{A} = \varprojlim_{J \in \cJ_{\geq 0}} A/J \]
  is the profinite completion of $A$, $\cJ_{\geq 0}$ is the set of ideals of $A$, and $N$ is a proper ideal of $A$.

  We consider the moduli space of isomorphism classes of Drinfeld modules with level $N$ structure, or equivalently of isomorphism classes of lattices with level $N$ structure.
  By \cite[Section~6]{drinfeld1974english} and \cite[p.~5]{Pink2013compactification}, there is an algebraic variety $M_{A,K(N)}^r$ defined over $F$ which acts as the aforementioned moduli space and an isomorphism
  \begin{equation} \label{eq:moduli_iso}
    \lquotient{\GL[r]{F}}{\parens{\Omega^r \times \rquotient{\GL[r]{\finadele}}{K(N)}}} \longisoto M_{A,K(N)}^r\parens{\CCi}
  \end{equation}
  which sends the equivalence class of $\parens{\omega,g} \in \Omega^r \times \GL[r]{\finadele}$ to the isomorphism class of Drinfeld modules associated to the lattice $\Lambda = \omega\parens{F^r \cap g\hat{A}^r}$\footnote{Here we consider $\omega \in \Omega^r$ as a map $\omega : \PP^r(F) \to \CCi$ in the obvious way.} and the level structure $\alpha$ which makes the following diagram commute:
  \begin{cdiagram*} \label{diag:wg_lattice_level}
    \parens{N^{-1}/A}^r
      \arrow[rr, hook, two heads, "\alpha"]
      \arrow[d, hook, two heads, "\subset"']  &  &
    N^{-1}\Lambda/\Lambda \\
    N^{-1}\hat{A}^r/\hat{A}^r
      \arrow[r, hook, two heads, "g"']  &
    N^{-1}g\hat{A}^r/g\hat{A}^r  &
    N^{-1}\parens{F^r \cap g\hat{A}^r}/\parens{F^r \cap g\hat{A}^r}
      \arrow[l, hook, two heads, "\subset"]
      \arrow[u, hook, two heads, "\omega"']
  \end{cdiagram*}

  Here, the left and right action of $f \in \GL[r]{F}$ and $k \in K(N)$ respectively on $(\omega,g) \in \Omega^r \times \GL[r]{\finadele}$ is as follows:
  \[ f (\omega,g) k = (\omega f^{-1}, f g k). \]

  Also, $\hat{A}^r$ and $F^r$ are considered as sets of column vectors.
\end{paragraph}

\begin{paragraph}
  In \cite{drinfeld1974english}, Drinfeld requires the level $N$ to lie in two distinct maximal ideals of $A$ for his more general setting. Since we consider here Drinfeld modules and lattices over $\CCi$, we only require $N$ to lie in one maximal ideal, as in \cite[3]{Pink2013compactification}. Hence we only require $N \neq A$.
\end{paragraph}

\begin{definition} \label{def:rigidAnalyticDiscontinuousAction}
  A group $\Gamma$ acts \emph{discontinuously} on a separable rigid analytic space $Y$ if there is an index set $I$, an action of $\Gamma$ on $I$, and an admissible covering $(Y_i)_{i \in I}$ of $Y$ such that the following conditions are satisfied:
  \begin{enumerate}
    \item $\gamma Y_i = Y_{\gamma i}$ for $i \in I$, $\gamma \in \Gamma$
    \item $\Gamma_i \defeq \setst{\gamma \in \Gamma}{\gamma i = i}$ is finite for each $i \in I$.
    \item If $\gamma \notin \Gamma_i$ then $Y_i \cap Y_{\gamma i} = \emptyset$. Moreover, if $i,j \in I$ then $Y_j \cap Y_{\gamma i} = \emptyset$ for all but finitely many $\gamma \in \Gamma$.
    \item For each $i \in I$, the covering $(Y_{\gamma i})_{\gamma \in \Gamma}$ of $\bigcup_{\gamma \in \Gamma} Y_{\gamma i}$ is admissible. \qedhere
  \end{enumerate}
\end{definition}

We cite the following proposition from \cite[582]{drinfeld1974english}:
\begin{proposition} \label{prop:rigidAnalyticQuotient}
  If a group $\Gamma$ acts discontinuously on a separable rigid analytic space $Y$, then $\lquotient{\Gamma}{Y}$ can be made into a separable rigid analytic space in such a way that the projection $\pi_{\Gamma,Y} : Y \onto \lquotient{\Gamma}{Y}$ is a morphism of rigid analytic spaces.
\end{proposition}

\begin{paragraph} \label{para:rigidAnalyticQuotient}
  Drinfeld showed in \cite{drinfeld1974english}, as did Schneider and Stuhler in \cite[\S 1]{SchneiderStuhler1991cohomology}, that the space $\Omega^r$ can be endowed with the structure of a separable rigid analytic space.
  Moreover, Drinfeld showed that any subgroup of $\GL[r]{F}$ commensurable with $\GL[r]{A}$ acts discontinuously on $\Omega^r$; thus by \cref{prop:rigidAnalyticQuotient} the quotient $\lquotient{\GL[r]{F}}{\Omega^r}$ can be given a derived rigid analytic structure.
  Thus the above double quotient can be given a rigid analytic structure, with the quotient $\rquotient{\GL[r]{\finadele}}{K(N)}$ given the discrete topology.
\end{paragraph}

\begin{paragraph}
  The space $\LL_N^r$ of lattices with level $N$ structure has an action of $\CCi^\times$, given by scaling the lattice and the level structure, which is free by \cref{coro:llattice_automorph}; thus each fibre of the quotient $\LL_N^r \onto \rquotient{\LL_N^r}{\CCi^\times}$ is isomorphic to $\CCi^\times$.
  Moreover, by \cref{prop:llattice_isomorph} the space of isomorphism classes of lattices with level $N$ structure is the quotient $\rquotient{\LL_N^r}{\CCi^\times}$.
\end{paragraph}

We can extend the isomorphism in \cref{eq:moduli_iso} to an isomorphism between the set $\LL_N^r$ and a related double quotient; but we will first define a rigid analytic structure on the space $\Psi^r$ which will take the place of $\Omega^r$:

\begin{definition} \label{def:Psi,H}
  \newcommand*{\CH}{\mathcal{H}}
  We let $\CH^r$ be the set of all hyperplanes in $\CCi^r$ which can be defined with coefficients in $F_\infty$.
  Then we define the space
  \[ \Psi^r = \CCi^r - \bigcup_{H \in \CH^r} H \]
  of all points in $\CCi^r$ which do not lie on any $F_\infty$-rational hyperplane.
\end{definition}

There is an obvious analogy between this space $\Psi^r$ and the traditional $\Omega^r$, the latter being formed by deleting all $F_\infty$-rational hyperplanes from $\PP_r(\CCi)$.
In fact, we use this analogy to define the rigid analytical structure on $\Psi^r$:

\begin{proposition} \label{prop:PsiIsoRigid}
  There is a bijection
  \begin{align*}
    \kappa : \Omega^r \times \CCi^\times &\ionto \Psi^r \\
    \kappa \parens{\pcoord{\omega_1 : \omega_2 : \dotsc : \omega_r}, \psi_r} &\defeq \parens{\omega_1, \omega_2, \dotsc, \omega_r} \cdot \frac{\psi_r}{\omega_r} \\
    \kappa^{-1} \parens{\psi_1, \psi_2, \dotsc, \psi_r} &= \parens{\pcoord{\psi_1 : \psi_2 : \dotsc : \psi_r}, \psi_r} \qedhere
  \end{align*}
\end{proposition}
This bijection is equivalent to normalising $\Omega^r$ so that the last component $\omega_r = 1$.

\begin{paragraph}
  We thus define the rigid analytic structure on $\Psi^r$ as the structure of the product $\Omega^r \times \CCi^\times$, each of these being rigid analytic spaces.
\end{paragraph}


There is a left and right action of $f \in \GL[r]{F}$ and $k \in K(N)$ respectively on $(\psi,g) \in \Psi^r \times \GL[r]{\finadele}$, which extends that given in \cref{para:DblQuotIso}, as follows:
\[ f(\psi,g)k \defeq (\psi f^{-1}, fgk). \]

\begin{proposition} \label{prop:fkActionPsiOmega}
  The above action of $f \in \GL[r]{F}$ and $k \in K(N)$ respectively on $(\psi,g) \in \Psi^r \times \GL[r]{\finadele}$ translates to $\Omega^r \times \CCi^\times \simeq \Psi^r$ via \cref{prop:PsiIsoRigid} as follows:
  \[ f \bigl((\omega,\psi_r),g\bigr) k = \parens{\parens{\omega f^{-1}, \frac{(\omega f^{-1})_r}{\omega_r} \psi_r}, fgk}. \qedhere \]
\end{proposition}
Here $(\omega f^{-1})_r$ denotes the last entry of $\omega f^{-1}$, and the fraction $\frac{(\omega f^{-1})_r}{\omega_r}$ is independent of the representative for $\omega$ chosen in $\CCi^r$.
\begin{proof}
  Let $\psi = \kappa(\omega,\psi_r)$ be a representative for $\omega$ in $\CCi^r$. Then
  \begin{align*}
    f \bigl((\omega,\psi_r),g\bigr) k &\overset{\mathmakebox[15pt]{\kappa}}{ = } f (\psi,g) k = (\psi f^{-1}, f g k) \\
    &\overset{\kappa^{-1}}{ = } \bigl(\kappa^{-1}(\psi f^{-1}), f g k\bigr) \\
    &\overset{\mathmakebox[15pt]{}}{ = } \bigl( (\psi f^{-1}, (\psi f^{-1})_r), f g k \bigr) = \parens{\parens{\psi f^{-1}, \frac{(\psi f^{-1})_r}{\psi_r} \psi_r}, f g k} \\
    &\overset{\mathmakebox[15pt]{}}{ = } \parens{\parens{\psi f^{-1}, \frac{(\omega f^{-1})_r}{\omega_r} \psi_r}, f g k} \qedhere
  \end{align*}
\end{proof}

Similarly to \cref{para:DblQuotIso} we then have a double quotient bijection for $\LL_N^r$:
\begin{theorem} \label{thm:PsiDoubleQuotientIso}
  There is a bijection
  \[ \Theta : \lquotient{\GL[r]{F}}{\parens{\Psi^r \times \rquotient{\GL[r]{\finadele}}{K(N)}}} \longisoto \LL_N^r \]
  which sends the equivalence class of a pair $(\psi,g) \in \Psi^r \times \GL[r]{\finadele}$ to the lattice $\Lambda = \psi\parens{F^r \cap g \hat{A}^r}$\footnote{Here we consider $\psi \in \Psi^r \subset \CCi^r$ as a map $F^r \to \CCi$ in the obvious way.} and the level structure $\alpha$ which makes the following diagram commute:
  \begin{cdiagram} \label{diag:psig_lattice_level}
    \parens{N^{-1}/A}^r
      \arrow[rr, hook, two heads, "\alpha"]
      \arrow[d, hook, two heads, "\subset"']  &  &
    N^{-1}\Lambda/\Lambda \\
    N^{-1}\hat{A}^r/\hat{A}^r
      \arrow[r, hook, two heads, "g"']  &
    N^{-1}g\hat{A}^r/g\hat{A}^r  &
    N^{-1}\parens{F^r \cap g\hat{A}^r}/\parens{F^r \cap g\hat{A}^r}
      \arrow[l, hook, two heads, "\subset"]
      \arrow[u, hook, two heads, "\psi"']
  \end{cdiagram}
  
  \vspace{-19pt} \qedhere
\end{theorem}
\begin{proof}
  For $f \in \GL[r]{F}$ and $k \in K(N)$,
  \[ \parens{\psi f^{-1}} \parens{F^r \cap (fgk) \hat{A}^r} = \psi\parens{f^{-1}F^r \cap f^{-1}fg \hat{A}^r} = \psi\parens{F^r \cap g\hat{A}^r}, \]
  so acting by $\GL[r]{F}$ and $K(N)$ leaves the lattice $\Lambda = \psi\parens{F^r \cap g\hat{A}^r}$ unchanged.
  Following the above commutative diagram, we see that the actions of $\GL[r]{F}$ and $K(N)$ also leave the level structure $\alpha$ unchanged, since $k$ changes nothing modulo $N$ and the addition of $f^{-1}$ on the right-hand map and $f$ on the bottom left map cancel.
  Hence the above map is well defined.

  If we consider the actions of $\CCi^\times$ on $\Psi^r$ and $\LL_N^r$ by scaling, their quotients are $\Omega^r$ and $M^r_{A,K(N)}(\CCi)$ respectively, each fibre being isomorphic to $\CCi^\times$.
  Moreover, this scaling commutes with the group actions of $\GL[r]{F}$ and $\GL[r]{\finadele}$ in the above double quotient.
  Hence the bijection
  \begin{align*}
      \lquotient{\GL[r]{F}}{\parens{\Omega^r \times \rquotient{\GL[r]{\finadele}}{K(N)}}} &\longisoto M_{A,K(N)}^r\parens{\CCi}
      \shortintertext{extends to a bijection}
      \lquotient{\GL[r]{F}}{\parens{\Psi^r \times \rquotient{\GL[r]{\finadele}}{K(N)}}} &\longisoto \LL_N^r. \qedhere
  \end{align*}
\end{proof}

\begin{paragraph}
  Similarly to Drinfeld in \cite{drinfeld1974english} and Schneider and Stuhler in \cite{SchneiderStuhler1991cohomology}, we can show that $\Psi^r$ can be given rigid analytic structure, and by extension the same is true for
  \[ \lquotient{\GL[r]{F}}{\parens{\Psi^r \times \rquotient{\GL[r]{\finadele}}{K(N)}}}, \]
  and thus also for $\LL_N^r$, making the bijection in \cref{thm:PsiDoubleQuotientIso} a rigid analytic isomorphism.
\end{paragraph}

From the above, we see that the set $\LL_N^r$ of all lattices of rank $r$ with level $N$ structure can be given rigid analytic structure.
From this we can induce rigid analytic structure on the set $\LL^r$ of lattices of rank $r$ \emph{without} level structure, as follows:
\begin{proposition} \label{prop:LLR_rigidAnalytic}
  $\LL^r$ can be given rigid analytic structure induced from that of $\LL_N^r$.
\end{proposition}
\begin{proof}
  Consider the left action of $\GL[r]{A/N}$ (considered as automorphisms of $\parens{N^{-1}/A}^r$) on $\LL_N^r$ defined by $\gamma(\Lambda,\alpha) = (\Lambda, \alpha \circ \gamma^{-1})$ for $\gamma \in \GL[r]{A/N}$.
  This action is free since $\alpha$ and $\gamma$ are bijections, and is transitive on the second component of $(\Lambda,\alpha)$ while leaving the first unchanged; hence the quotient $\lquotient{\GL[r]{A/N}}{\LL_N^r}$ is bijective with $\LL^r$.
  Now, using \cref{prop:rigidAnalyticQuotient}, since $\GL[r]{A/N}$ is finite it acts discontinuously on $\LL_N^r$ and so its quotient $\LL^r$ has an induced rigid analytic structure.
\end{proof}

%% file: text/2_2_lm_irred.tex
\subsection[Irreducible components of the moduli space]{Irreducible components of $M_{A,K(N)}^r(\CCi)$ and $\LL_N^r$}

The rigid analytic space $M_{A,K(N)}^r(\CCi)$ decomposes into irreducible components as in \cite[Proposition 2.1.3]{hubschmid2013andre}, given below.
We call the corresponding partition of $\LL_N^r$, induced from its quotient map onto $M_{A,K(N)}^r(\CCi)$, the irreducible components of $\LL_N^r$.

\begin{proposition} \label{prop:moduli_components}
  Let $H$ be a set of representatives in $\GL[r]{\finadele}$ for the double quotient
  \[ \lrquotient{\GL[r]{F}}{\GL[r]{\finadele}}{K(N)}, \]
  and set $\Gamma_g = g K(N) g^{-1} \cap \GL[r]{F}$ for $g \in H$.
  Then the map
  \begin{align*}
    \bigsqcup_{g \in H} \lquotient{\Gamma_g}{\Omega^r} &\longrightarrow \lquotient{\GL[r]{F}}{\parens{\Omega^r \times \GL[r]{\finadele}/K(N)}} \\
    [\omega]_g &\longmapsto [(\omega,g)]
  \end{align*}
  is a rigid analytic isomorphism which maps for each $g \in H$ the quotient space $\lquotient{\Gamma_g}{\Omega^r}$ to an irreducible component of $M_{A,K(N)}^r(\CCi)$.
\end{proposition}
Each $\Gamma_g$ is an arithmetic subgroup of $\GL[r]{F}$.

Here is the corresponding result for $\LL_N^r$:
\begin{proposition} \label{prop:LLNR_components}
  For $H$ and $\Gamma_g$ as in \cref{prop:moduli_components}, the map
  \begin{align*}
    \bigsqcup_{g \in H}\ \lquotient{\Gamma_g}{\Psi^r} &\longrightarrow \lquotient{\GL[r]{F}}{\parens{\Psi^r \times \GL[r]{\finadele}/K(N)}} \\
    [\psi]_g &\longmapsto [(\psi,g)]
  \end{align*}
  is a rigid analytic isomorphism which maps for each $g \in H$ the space $\lquotient{\Gamma_g}{\Psi^r}$ to an irreducible component of $\LL_N^r$.
  Here the action of $f \in \Gamma_g \subseteq \GL[r]{F}$ on $\psi \in \Psi^r$ is as in \cref{thm:PsiDoubleQuotientIso}, \ie $f \cdot \psi = \psi f^{-1}$.
\end{proposition}
\begin{proof}
  Consider the action of $\CCi^\times$ on $\Psi^r$, with quotient $\Omega^r$, each fibre of which is isomorphic to $\CCi^\times$; since this action of $\CCi^\times$ commutes with the actions of $\GL[r]{F}$ and $\GL[r]{\finadele}$, the bijection in \cref{prop:moduli_components} extends to the bijection given above.
\end{proof}

We include the following result from \cite[Definition 3.4.1, Proposition 3.4.2]{hubschmid2013andre}:
\begin{proposition} \label{prop:detSeperatesComps}
  The map $\det$ from $M_{A,K(N)}^r(\CCi)$ given by
  \begin{align*}
    \lquotient{\GL[r]{F}}{\parens{\Omega^r \times \GL[r]{\finadele}/K(N)}} &\longrightarrow \lrquotient{F^\times}{\parens{\finadele}^\times}{\det{K(N)}} \\
    [(\omega,g)] &\longmapsto [\det{g}]
  \end{align*}
  is surjective and the fibres are the irreducible components of $M_{A,K(N)}^r(\CCi)$.
\end{proposition}

\begin{corollary} \label{coro:DetDblQuot}
  The map $\det$ above induces a bijection
  \begin{align*}
    \lrquotient{\GL[r]{F}}{\GL[r]{\finadele}}{K(N)} &\longionto \lrquotient{F^\times}{\parens{\finadele}^\times}{\det{K(N)}} \\
    [g] &\longmapsto [\det{g}]. \qedhere
  \end{align*}
\end{corollary}

We know the number of irreducible components by \cite[Corollary 3.4.5]{hubschmid2013andre}:
\begin{proposition} \label{prop:comps_HubschmidNumber}
  The irreducible components of $\LL_N^r$ have number
  \[ \sizep{\lrquotient{F^\times}{\parens{\finadele}^\times}{\det K(N)}} = \size{\Cl(F)} \cdot \sizep{\hat{A}^\times/(\FF_q^\times \cdot \det K(N))}. \qedhere \]
\end{proposition}

We can take the above count further as shown in \cref{prop:Comps_ClF_AN*_Bijection}, but first a

\begin{lemma} \label{lem:detKNStructure}
  $\displaystyle \quad \det{K(N)} = \hat{A}^\times \cap (1 + N\hat{A}) = \prod_{\fp \nmid N} A_\fp^\times \cdot \prod_{\fp \mid N} 1+(\fp A_\fp)^{v_\fp(N)}$.
\end{lemma}
Here the product is over prime ideals $\fp$.
\begin{proof}
  Each $x \in \det{K(N)}$ has $x \equiv_N 1$, and conversely if $x \in \hat{A}^\times \cap (1+N\hat{A})$ then $x' \in \GL[r]{\hat{A}}$, which viewed as an $r \times r$ matrix has $x$ in the first entry, $1$ along the rest of the diagonal and $0$ elsewhere, has $\det{x'} = x$ and is in $K(N)$.
  This proves the first equality.

  For the second equality, note that
  \begin{align*}
    x = (x_\fp)_\fp \equiv_N 1 &\iff x-1 \in N\hat{A} & & \\
    &\iff x-1 \in \fp^{v_\fp(N)} \hat{A} & & \lsptext{for each} \fp \mid N \\
    &\iff x_\fp-1 \in (\fp A_\fp)^{v_\fp(N)} & & \lsptext{for each} \fp \mid N \,. \qedhere
  \end{align*}
\end{proof}

\begin{proposition} \label{prop:Comps_ClF_AN*_Bijection}
  The injective map
  \begin{align*}
      i_N : \rquotient{\parens{A/N}^{\!\times}}{\FF_q^\times} &\longinto \lrquotient{F^\times}{\parens{\finadele}^{\!\times}}{\det K(N)} \\
        [x, x \in A] &\longmapsto [(x)_{\fp \mid N} \cup (1)_{\fp \nmid N}]
  \end{align*}
  and the surjective map
  \begin{align*}
       \pi_N : \lrquotient{F^\times}{\parens{\finadele}^{\!\times}}{\det K(N)} &\longonto \Cl(F) \\
        [x] &\longmapsto \brackets{\prod_\fp \fp^{v_\fp(x)}}_{\Cl(F)}
  \end{align*}
  together form a short exact sequence of abelian groups.
\end{proposition}
\begin{proof}
  First we show that $i_N$ and $\pi_N$ are well defined and are injective and surjective respectively.
  Note that $v_\fp(k) = 0$ for all prime $\fp$ and $k \in \det{K(N)}$.
  
  $i_N$: Let $x, y \in A$ with $x+N, y+N \in \parens{A/N}^\times$, noting that $v_\fp(x) = v_\fp(y) = 0$ for all $\fp \mid N$.
  Then for $f \in \FF_q^\times$,
  \begin{align*}
    & x \equiv f y \pmod{N} \\
    \iff & x/fy \equiv 1 \pmod{N} \\
    \iff & x/fy \equiv 1 \pmod{\fp^{v_\fp(N)}} \lsptext{for all} \fp \mid N \\
    \iff & \rquotient{((x)_{\fp \mid N} \cup (1)_{\fp \nmid N})}{f((y)_{\fp \mid N} \cup (1)_{\fp \nmid N})} \in \det{K(N)};
  \end{align*}
  thus $i_N$ is well-defined and injective.
  
  $\pi_N$: Let $x = \parens{\finadele}^{\!\times}$, $f \in F^\times$, and $k \in \det{K(N)}$.
  Then $\prod_\fp \fp^{v_\fp(f)} = (f)$ is principal, so that $\brackets{\prod_\fp \fp^{v_\fp(x)}} = \brackets{\prod_\fp \fp^{v_\fp(f x k)}}$.
  Also, if $I$ is a fractional ideal then $v_\fp(I) \neq 0$ for only finitely many $\fp$; choosing uniformisers $u_\fp \in F_\fp$ for all such, we have that $\pi_N$ maps $\brackets{\parens{u_\fp^{v_\fp(I)}}_{v_\fp(I) \neq 0} \cup (1)_{v_\fp(I) = 0}}$ to $[I]$.

  Now to show that $\Im{i_N} = \ker{\pi_N}$, let $[x] \in \ker{\pi_N}$, where $x = (x_\fp)_\fp$.
  Then $\prod_\fp \fp^{v_\fp(x)} = (f)$ is principal, for some $f \in F^\times$.
  Thus $v_\fp(x/f) = 0$ for all $\fp$, so that $x/f \in \hat{A}^\times$ and so is invertible under the projection $\hat{A} \onto A/N$.
  So there is a $t \in \parens{A/N}^\times$ such that $x/f \equiv t \pmod{N}$, or equivalently $x/f \equiv (t)_{\fp \mid N} \cup (1)_{\fp \nmid N} \pmod{N}$, so we have that $k = \rquotient{(x/f)}{((t)_{\fp \mid N} \cup (1)_{\fp \nmid N})} \in \det{K(N)}$.
  Thus $[x] = [f \backslash x / k] \in \Im{i_N}$.
  Conversely, consider $x = (t)_{\fp \mid N} \cup (1)_{\fp \nmid N} \in \hat{A}^\times$ for $t \in A$ which is invertible modulo $N$, so that $[x] \in \Im{i_N}$.
  Then $v_\fp(x) = 0$ for all $\fp$, so that $\pi_N(x) = 0$, \ie $x \in \ker{\pi_N}$.
\end{proof}

So for each irreducible component $C$ of $\LL_N^r$, there is a corresponding class group element $\pi_N(C) \in \Cl(F)$, and those for which $\pi_N(C) = 1$ can be written as $C = i_N(x)$ for some $x \in \rquotient{\parens{A/N}^\times\!}{\FF_q^\times}$.
Note here the abuse of notation: we will use $\pi_N$ as a function from $\lrquotient{F^\times}{\parens{\finadele}^{\!\times}}{\det K(N)}$, from the double quotients $\lquotient{\GL[r]{F}}{\parens{\Omega^r \times \GL[r]{\finadele}/K(N)}}$ and $\lquotient{\GL[r]{F}}{\parens{\Psi^r \times \GL[r]{\finadele}/K(N)}}$, \\and from the set of irreducible components.
Similarly, $i_N$ could have as codomain any of the above spaces, with context dictating which is intended.

\begin{definition} \label{def:InitialComponent}
  The \emph{identity component} of $M_{A,K(N)}^r(\CCi)$ is the fibre of the identity element in the surjection of \cref{prop:detSeperatesComps}.
  We will use the notation $1_N^r$ for the corresponding identity component of $\LL_N^r$.
\end{definition}

So far we have looked at identifying the different components from the point of view of $\omega \in \Omega^r$ and $g \in \GL[r]{\finadele}$.
The following series of results carries through this identification to the point of view of a lattice $\Lambda$ with level structure $\alpha$.

\begin{proposition} \label{prop:latticeClassCharac}
  For fractional ideals $I_1, \dotsc, I_r$ of $F$ and $\psi = (\psi_1, \dotsc, \psi_r) \in \Psi^r$, and the resulting lattice $\Lambda = I_1\psi_1 +\dotsb +I_r\psi_r$ together with any associated level structure $\alpha$, we have that $\pi_N(\Lambda,\alpha) = [I_1 I_2 \dotsm I_r]_{\Cl(F)} \in \Cl(F)$.
\end{proposition}
\begin{proof}
  We will proceed by finding $\psi'$ and $g$ such that $\Theta([(\psi',g)]) = (\Lambda,\alpha)$, with $\Theta$ being the isomorphism
  \[ \Theta : \lquotient{\GL[r]{F}}{\parens{\Psi^r \times \rquotient{\GL[r]{\finadele}}{K(N)}}} \longisoto \LL_N^r \]
  from \cref{thm:PsiDoubleQuotientIso}.
  We choose $\psi' = \psi$, and proceed to constructing $g$.
  Each $I_i$ can be uniquely factorised as a product $\prod_\fp \fp^{a_{\fp,i}}$ for prime ideals $\fp$ and integers $a_{\fp,i}$ almost all zero, so choosing uniformisers $u_\fp$ for each localisation $A_\fp$ with $a_{\fp,i}$ not all zero we have that $g_i \defeq \parens{u_\fp^{a_{\fp,i}}}_\fp$ satisfies $g_i \hat{A} = I_i \hat{A}$.
  Let $g'$ be the matrix with the $g_i$ on the diagonal and zeroes elsewhere.
  We have that
  \begin{align*}
    g'\hat{A}^r &= \parens{g_1\hat{A}, \dotsc, g_r\hat{A}}^T = \parens{I_1\hat{A}, \dotsc, I_r\hat{A}}^T,
    \shortintertext{so that}
    \psi\parens{F^r \cap g'\hat{A}^r} &= \psi(I_1,\dotsc, I_r) = I_1\psi_1 +\dotsb +I_r\psi_r = \Lambda
  \end{align*}
  as desired.
  Now this $g'$ induces a level structure $\alpha' : \parens{N^{-1}/A}^r \ionto N^{-1}\Lambda/\Lambda$ as in \cref{thm:PsiDoubleQuotientIso}, which may not be the desired $\alpha$.
  However, since $\GL[r]{\hat{A}}$ surjects onto $\GL[r]{A/N}$, there is a lift $\gamma \in \GL[r]{\hat{A}}$ of $\alpha'^{-1} \circ \alpha \in \GL[r]{A/N}$.
  Then defining $g = g' \circ \gamma$, we have that $g\hat{A}^r = g'\hat{A}^r$, so that $\Lambda = \psi\parens{F^r \cap g\hat{A}^r}$, and by following \cref{diag:psig_lattice_level} that $g$ induces the level structure $\alpha$.

  Now $\det{\gamma} \in \hat{A}^\times$, so $\det{g} = \det{g'} \det{\gamma} \in g_1 \dotsm g_r \hat{A}^\times$.
  Thus
  \begin{align*}
    \pi_N(\Lambda,\alpha) &= \pi_N(\det{g}) \\
    &= \brackets{\prod_\fp \fp^{v_\fp(g_1 \dotsm g_r)}} = \brackets{\prod_\fp \fp^{v_\fp(g_1)} \dotsm \prod_\fp \fp^{v_\fp(g_r)}} \\
    &= \brackets{I_1 \dotsm I_r}. \qedhere
  \end{align*}
\end{proof}

\begin{corollary} \label{coro:latticeClass1Charac}
  $\pi_N(\Lambda,\alpha) = [A]$ if and only if $\Lambda = A\psi_1 +\dotsb A\psi_r = \psi A^r$ for some row vector $\psi = (\psi_1, \dotsc, \psi_r) \in \Psi^r$.
\end{corollary}
\begin{proof}
  For the forward direction, by \cref{para:latticeAmodule} there are an ideal $I$ of $A$ and a tuple $\psi = (\psi_1, \dotsc, \psi_r) \in \Psi^r$ such that $\Lambda = A\psi_1 +\dotsb +A\psi_{r-1} +I\psi_r$.
  Now by \cref{prop:latticeClassCharac} we have $[I] = [A]$ in $\Cl(F)$, \ie $I = (n)$ is principal for some $n \in A$.
  Hence replacing $\psi'_r$ by $n\psi_r$, we have that $\Lambda = A\psi_1 +\dotsb +A\psi_r = \psi A^r$.

  The converse is a direct application of \cref{prop:latticeClassCharac}.
\end{proof}

\begin{paragraph} \label{para:piN_noN}
  Note that the value of $\pi_N(\Lambda,\alpha)$ does not depend at all on the level structure $\alpha$ or even the ideal $N$; hence we may make use of the notation $\pi(\Lambda)$ instead, and in fact will also use $\pi(C)$ to denote $\pi(\Lambda)$ for a lattice in the irreducible component $C$.
\end{paragraph}

\begin{definition} \label{def:canonicalLevel}
  For a lattice $\Lambda$ with $\pi(\Lambda) = [A]$, to each choice of a generating vector $\psi \in \Psi^r$ satisfying $\Lambda = \psi A^r$ there is an associated canonical level $N$ structure $\alpha_\psi : \parens{N^{-1}/A}^r \ionto N^{-1}\Lambda/\Lambda$, given by $\alpha_\psi(l) = \psi l$\footnote{Here $\psi$ is considered as a row vector and $l$ as a column vector.} for $l_i \in \parens{N^{-1}/A}^r$.
\end{definition}

The choice of a canonical level $N$ structure $\alpha_\psi$ for a lattice $\Lambda = \psi A^r$ is equivalent to having $(\Lambda,\alpha_\psi) = \Theta([(\psi,\mathrm{Id})])$ where $\mathrm{Id} \in \GL[r]{\finadele}$ is the identity matrix.

\begin{paragraph} \label{para:canonicalLevelEquiv}
  Note that for any two choices $\psi_1, \psi_2 \in \Psi^r$ of generating vectors for a lattice $\Lambda$ with $\pi(\Lambda) = [A]$, since $\psi_1 A^r = \psi_2 A^r$ we have that $\psi_2 = \psi_1 \gamma$ for some $\gamma \in \GL[r]{A}$, and hence $\alpha_{\psi_2}(l) = \psi_2 l = \psi_1 \gamma l = \alpha_{\psi_1}(\gamma l)$ for all $l \in \parens{N^{-1}/A}^r$, \ie $\alpha_{\psi_2} = \alpha_{\psi_1} \circ \gamma$.
  Thus $\det{\alpha_{\psi_1}^{-1} \circ \alpha_{\psi_2}} \in A^\times = \FF_q^\times$.
\end{paragraph}

\begin{proposition} \label{prop:latticeN*Charac}
  Let $(\Lambda,\alpha) = \Theta([(\psi,g)])$.
  Then if $\pi(\Lambda) = [A]$, we have that
  \[ i_N([\det{\alpha_\psi^{-1} \circ \alpha}]) = [\det{g}]. \qedhere \]
\end{proposition}
\begin{proof}
  We may choose different representatives $\psi$ and $g$ for $(\Lambda,\alpha)$, since $[\det{g}]$ is invariant under such a change and
  \[ [\det{\alpha_{\psi_2}^{-1} \circ \alpha}] = [\det{\alpha_{\psi_2}^{-1} \circ \alpha_{\psi_1}^{-1}}] [\det{\alpha_{\psi_1}^{-1} \circ \alpha}] = [\det{\alpha_{\psi_1}^{-1} \circ \alpha}] \]
  for any two generating vectors $\psi_1, \psi_2$ for $\Lambda$.

  Now since $\pi(\Lambda) = [A]$, there is a generating vector $\psi \in \Psi^r$ for $\Lambda$.
  Choosing $g' = \mathrm{Id}$ to be the identity matrix, $\Theta([(\psi,g')]) = (\Lambda,\alpha_\psi)$.
  Letting $\gamma \in \GL[r]{\hat{A}}$ be a lift of $\alpha_\psi^{-1} \circ \alpha \in \GL[r]{A/N}$, we can define $g = g' \gamma = \gamma$ which by following \cref{diag:psig_lattice_level} we see induces our level structure $\alpha$.
  So $\alpha_\psi^{-1} \circ \alpha = \gamma \bmod{N}$, so that
  \[ [\det{g}] = [\det{\gamma}] = i_N([\det{\alpha_\psi^{-1} \circ \alpha}]). \qedhere \]
\end{proof}

\begin{theorem} \label{thm:IdentCompCharac}
  $(\Lambda,\alpha) \in 1_N^r$ if and only if there is a generating vector $\psi \in \Psi^r$ for $\Lambda$ such that $\alpha = \alpha_\psi$.
\end{theorem}
\begin{proof}
  For the forward direction, let $(\Lambda,\alpha)$ be in the identity component $1_N^r$.
  Then by \cref{coro:latticeClass1Charac} there is a $\psi' = (\psi'_1, \dotsc, \psi'_r) \in \Psi^r$ such that $\Lambda = \psi' A^r$.

  Now by \cref{prop:latticeN*Charac}, since $(\Lambda,\alpha)$ is in $1_N^r$ we have that $\det{\alpha_{\psi'}^{-1} \circ \alpha} \in \FF_q^\times$, and so there is a lift $\gamma \in \GL[r]{A}$ for $\alpha_{\psi'}^{-1} \circ \alpha \in \GL[r]{A/N}$.
  So define $\psi = \psi' \gamma$; then $\psi A^r = \psi' \gamma A^r = \psi' A^r = \Lambda$, and for $l \in \parens{N^{-1}/A}^r$ we have that $\alpha(l) = \alpha_{\psi'}(\gamma l) = \psi' \gamma l = \psi l$.

  For the reverse direction, note that $(\Lambda,\alpha) = \Theta([(\psi,\mathrm{Id})])$ where $\mathrm{Id}$ is the identity in $\GL[r]{\finadele}$, and hence $(\Lambda,\alpha)$ lies in $1_N^r$.
\end{proof}
In other words, $(\Lambda,\alpha)$ is in the identity component if and only if $\pi(\Lambda) = [A]$ and a canonical level structure is used.

%% file: text/2_3_lm_GLrAN.tex
\subsection[The action of GLr(A/N) on the moduli space]{The action of $\GL[r]{A/N}$ on $M_{A,K(N)}^r(\CCi)$ and $\LL_N^r$}

\begin{definition} \label{def:GLrAhatAction}
  We define a left action of $\GL[r]{\hat{A}}$ on $M_{A,K(N)}^r(\CCi)$ and $\LL_N^r$ by
  \[ \gamma [(\omega,g)] = [(\omega, g \circ \gamma^{-1})] \lrsptext{and} \gamma [(\psi,g)] = [(\psi, g \circ \gamma^{-1})] \lsptext{for} \gamma \in \GL[r]{\hat{A}}. \qedhere \]
\end{definition}

\begin{proposition} \label{prop:GLrAhatActionVerify}
  The above action is well defined.
\end{proposition}
\begin{proof}
  Left to the reader.
\end{proof}

\begin{paragraph}
  Note that since the above action leaves the components of $\Omega^r$ and $\Psi^r$ unchanged, and the topology on the quotient $\rquotient{\GL[r]{\finadele}}{K(N)}$ is discrete, the above actions are rigid analytic automorphisms of the relevant spaces.
\end{paragraph}
  
\begin{proposition} \label{prop:GLrAhatAN}
  The kernel of the above action on $\LL_N^r$ is the normal subgroup $K(N)$ of $\GL[r]{\hat{A}}$; it thus induces an action of $\GL[r]{A/N} \simeq \rquotient{\GL[r]{\hat{A}}}{K(N)}$.
\end{proposition}
\begin{proof}
  \newcommand*{\ohmega}{\overline{\omega}}
  Firstly, let $\Omega^r \times \CCi^\times \ni (\omega,\psi_r) = \kappa^{-1}(\psi)$ for $\psi \in \Psi^r$.
  Then for any $f \in \GL[r]{F}$ and a representative $\ohmega \in \Psi^r$ for $\omega \in \Omega^r$,
  \begin{align*}
    & \psi = f \psi \\
    \iff & (\omega,\psi_r) = f(\omega,\psi_r) = \parens{\omega f^{-1}, \frac{(\omega f^{-1})_r}{\omega_r} \psi_r} \\
    \iff & \omega = \omega f^{-1} \lrsptext{and} \psi_r = \frac{(\ohmega f^{-1})_r}{\ohmega_r} \psi_r \\
    \iff & \omega = \omega f^{-1} \lrsptext{and} (\ohmega f^{-1})_r = \ohmega_r \\
    \iff & \ohmega = \ohmega f^{-1} \\
    \iff & (f-1)\ohmega = 0 \\
    \iff & f = 1
  \end{align*}
  since the $\ohmega_i$ are $F_\infty$-linearly independent.
  Thus
  \begin{align*}
    &\gamma \in \GL[r]{\hat{A}} \sptext{is in the kernel of the action} \\
    \iff &(\forall [(\psi,g)] \in \LL_N^r)\ [(\psi,g)] = \gamma [(\psi,g)] = [(\psi, g \gamma^{-1})] \\
    \iff &(\forall [(\psi,g)] \in \LL_N^r)\ (\exists f \in \GL[r]{F}, k \in K(N))\ \psi = f \psi \mathrel{\wedge} g \gamma^{-1} = f g k \\
    \iff &(\forall [(\psi,g)] \in \LL_N^r)\ (\exists k \in K(N))\ g \gamma^{-1} = g k \\
    \iff &(\exists k \in K(N))\ \gamma = k^{-1} \\
    \iff & \gamma \in K(N).
  \end{align*}
  The proof for $M^r_{A,K(N)}(\CCi)$ is similar.
\end{proof}

We will thus view the above as actions of $\GL[r]{A/N}$ on $M_{A,K(N)}^r(\CCi)$ and $\LL_N^r$, writing elements of $\GL[r]{A/N}$ as $[\gamma]$ for $\gamma \in \GL[r]{\hat{A}}$ where necessary.

We now consider this action's effect on the irreducible components:
\begin{proposition} \label{prop:GLrANComponents}
  The above action of $\GL[r]{A/N}$ induces an action of $\parens{A/N}^\times$ on the irreducible components of $M_{A,K(N)}^r(\CCi)$ and $\LL_N^r$ via $[\gamma] \mapsto [\det{\gamma^{-1}}]$ with kernel $\FF_q^\times$ which leaves the ideal class $\pi_N(C)$ of each component $C$ unchanged.
\end{proposition}
\begin{proof}
  Under the determinant map of \cref{prop:detSeperatesComps}, $\GL[r]{A/N}$ acts on the components as follows:
  \[ \det\parens{[\gamma][(\psi,g)]} = \det{[(\psi,g\gamma^{-1})]} = [\det{g\gamma^{-1}}] = [\det{g}] \cdot [\det{\gamma^{-1}}]\,. \]
  and similarly for $M^r_{A,K(N)}(\CCi)$.
  From this one can see that if two points are in the same component, then they are still in the same component after acting with $\gamma$.
  Now for $\gamma \in \GL[r]{\hat{A}}$, $\det{\gamma^{-1}} \in \hat{A}^\times$ so that $v_\fp(\det{\gamma^{-1}}) = 0$ for all prime $\fp$ and hence for an irreducible component $C$ of $\LL_N^r$ we have that $\pi_N(\det{C}\det{\gamma^{-1}}) = \pi_N(\det{C})$.
  In particular, $\pi_N(\det{1_N^r} \det{\gamma^{-1}}) = 1$, so that $\det{1_N^r \gamma^{-1}} \in \Im{i_N}$.

  Finally, $\gamma$ is in the kernel of this action if and only if $[\det{\gamma^{-1}}] = [1]$, \ie if $\det{\gamma^{-1}} = f k$ for some $f \in F^\times$, $k \in \det{K(N)}$.
  Since $v_\fp(\det{\gamma^{-1}}) = v_\fp(k) = 0$ for all prime $\fp$, we have that $v_\fp(f) = 0$ for all prime $\fp$; hence $f \in \FF_q^\times$ and so $\det{\gamma^{-1}} \in \hat{A}^\times \cap \parens{\FF_q^\times +N\hat{A}}$; the converse can be shown easily.
  Thus
  \[ i_N^{-1}(\det{1_N^r\gamma^{-1}}) = \FF_q^\times \subseteq \GL[r]{\hat{A}/N\hat{A}} \simeq \GL[r]{A/N}. \qedhere \]
\end{proof}

So when $\gamma \in \GL[r]{A/N}$ acts on $M_{A,K(N)}^r(\CCi)$ and $\LL_N^r$, it permutes the irreducible components, leaving them fixed if and only if $\det{\gamma} \in \FF_q^\times$.

We now translate this action of $\GL[r]{A/N}$ to actions on $M_{A,K(N)}^r(\CCi)$ and $\LL_N^r$ considered as the set of (equivalence classes of) pairs $(\Lambda,\alpha)$ of a lattice of rank $r$ with level structure $\alpha$:

\begin{proposition} \label{prop:GLrANactLa}
  The action of $\GL[r]{A/N} \simeq \rquotient{\GL[r]{\hat{A}}}{K(N)}$ described as above acts on the spaces $M_{A,K(N)}^r(\CCi) = \set{[(\Lambda,\alpha)]}$ and $\LL_N^r = \set{(\Lambda,\alpha)}$ as follows:
  \[ \gamma[(\Lambda,\alpha)] = [(\Lambda, \alpha \circ \gamma^{-1})]; \qquad \gamma(\Lambda,\alpha) = (\Lambda, \alpha \circ \gamma^{-1}). \qedhere \]
\end{proposition}
\begin{proof}
  \newcommand*{\subbset}{\mathord{\subset}} 
  We will need to revisit the isomorphism $\Theta$ in \cref{thm:PsiDoubleQuotientIso}:
  \[ \Theta : \lquotient{\GL[r]{F}}{\parens{\Psi^r \times \GL[r]{\finadele}/K(N)}} \simeq \LL_N^r\,. \]
  
  Suppose that $\gamma(\Lambda,\alpha) = \Theta([(\psi,g)])$; here $\gamma[(\psi,g)] = [(\psi, g\hat{\gamma}^{-1})]$ for $\hat{\gamma} \in \GL[r]{\hat{A}}$ a lift of $\gamma \in \GL[r]{A/N}$.
  Then since $\hat{\gamma}^{-1} \hat{A}^r = \hat{A}^r$,
  \[ \Lambda' = \psi\parens{F^r \cap g \hat{\gamma}^{-1}\hat{A}^r} = \psi\parens{F^r \cap g \hat{A}^r} = \Lambda\,. \]
  To determine $\alpha'$, note that
  \[ \alpha \circ \subbset^{-1} \circ g^{-1} = \alpha' \circ \subbset^{-1} \circ (g\hat{\gamma}^{-1})^{-1} = \alpha' \circ \subbset^{-1} \circ \hat{\gamma} \circ g^{-1}, \]
  where $\subbset : \parens{N^{-1}/A}^r \ionto \rquotient{N^{-1}\hat{A}^r}{\hat{A}^r}$ is induced by the inclusion $N^{-1} \subset N^{-1}\hat{A}$.
  Now for $x = (x_1, \dotsc, x_r) \in (N^{-1}/A)^r$,
  \[ \subbset(\gamma(x)) = ((\gamma(x)_1)_\fp, \dotsc, (\gamma(x)_r)_\fp) = (\hat{\gamma}((x)_\fp)_1, \dotsc, \hat{\gamma}((x)_\fp)_r) = \hat{\gamma}(\subbset(x)); \]
  hence $\subbset \circ \gamma = \hat{\gamma} \circ \subbset$, so that
  \[ \alpha \circ \subbset^{-1} = \alpha' \circ \subbset^{-1} \circ \hat{\gamma} = \alpha' \circ \gamma \circ \subbset^{-1} \iff \alpha' = \alpha \circ \gamma^{-1}. \qedhere \]
\end{proof}

Note that the above action of $\GL[r]{A/N}$ on $\LL_N^r$ coincides with that defined in the proof of \cref{prop:LLR_rigidAnalytic}.

%% file: text/2_4_lm_invadele.tex
\subsection[The action of the adeles on the moduli space]{The action of $\invertadele$ on $M_{A,K(N)}^r(\CCi)$ and $\LL_N^r$}

\begin{definition} \label{def:invAdele_Action}
  We define a left action of $\invertadele$ on $M_{A,K(N)}^r(\CCi)$ and $\LL_N^r$ by
  \[ x[(\omega,g)] = [(\omega,g x^{-1})] \lrsptext{and} x[(\psi,g)] = [(\psi,g x^{-1})] \lsptext{for} x \in \invertadele. \qedhere \]
\end{definition}
Since the multiplicative group $\invertadele$ of invertible finite adeles is abelian, the distinction between a left and a right action is not very important here, but we call it a left action for harmony with the previous subsection.

\begin{proposition} \label{prop:invertadeleActionVerify}
  The above action is well defined.
\end{proposition}
\begin{proof}
  Left to the reader.
\end{proof}

\begin{paragraph}
  Note that since the above action leaves the components of $\Omega^r$ and $\Psi^r$ in the double quotients
  \begin{align*}
      \lquotient{\GL[r]{F}}{\parens{\Omega^r \times \rquotient{\GL[r]{\finadele}}{K(N)}}} &\longisoto M_{A,K(N)}^r\parens{\CCi} \lsptext{and} \\
      \lquotient{\GL[r]{F}}{\parens{\Psi^r \times \rquotient{\GL[r]{\finadele}}{K(N)}}} &\longisoto \LL_N^r
  \end{align*}
  respectively unchanged, and the topology on $\rquotient{\GL[r]{\finadele}}{K(N)}$ is discrete, the above actions are rigid analytic automorphisms of the relevant spaces.
\end{paragraph}
  
\begin{paragraph}
  Also note that since $\invertadele \subset \GL[r]{\finadele}$ consists of scalar matrices and thus is in the centre of $\GL[r]{\finadele}$, the action of $\invertadele$ commutes with the action of any other subgroup of $\GL[r]{\finadele}$, and in particular with the actions of $\GL[r]{\hat{A}}$ and $\GL[r]{A/N}$ defined in the previous subsection.
\end{paragraph}
  
\begin{proposition} \label{prop:invAdele_Kernel}
  The kernel of the above action on $\LL_N^r$ is $\hat{A}^\times \cap (1+N\hat{A})$.
\end{proposition}
\begin{proof}
  \newcommand*{\Idr}{\mathrm{Id}_r}
  As in the proof of \cref{prop:GLrAhatAN}, if $\psi = f \psi$ for $\psi \in \Psi^r$ and $f \in \GL[r]{F}$ then $f = 1$.
  Thus
  \begin{align*}
    & x \in \invertadele \sptext{is in the kernel of the action} \\
    \iff &(\forall [(\psi,g)] \in \LL_N^r)\ [(\psi,g)] = x [(\psi,g)] = [(\psi, g x^{-1})] \\
    \iff &(\forall [(\psi,g)] \in \LL_N^r)\ (\exists f \in \GL[r]{F}, k \in K(N))\ \psi = f \psi \mathrel{\wedge} g x^{-1} = f g k \\
    \iff &(\forall [(\psi,g)] \in \LL_N^r)\ (\exists k \in K(N))\ g x^{-1} = g k \\
    \iff &(\exists k \in K(N))\ x^{-1}\Idr = k \\
    \iff & x \Idr \in K(N).
  \end{align*}
  For $x \Idr$ to be in $K(N)$, we must have that $x^r = \det(x \Idr) \in \hat{A}^\times$, so that $x \in \hat{A}^\times$, and that $(x-1)\Idr = x\Idr-\Idr \in N M_{r\times r}(\hat{A})$, so that $x \equiv_N 1$.
\end{proof}
In other words, $x \in \invertadele$ is in the kernel of this action if and only if it is an invertible profinite integer with $x-1 \in N\hat{A}$.

Recall that we denote by $\cJ(A)$ the set of $A$-fractional ideals in $F$.
\begin{proposition} \label{prop:invAdele_FracANstar}
  There is an abelian group isomorphism
  \[ \rquotient{\invertadele}{\parens{\hat{A}^\times \cap (1+N\hat{A})}} \longisoto \cJ(A) \times (A/N)^\times. \qedhere \]
\end{proposition}
\begin{proof}
  $(A/N)^\times \simeq \prod_{\fp \mid N} \parens{A/\fp^{v_\fp(N)}}^{\!\times}$, so if we choose a uniformiser $u_\fp \in A_\fp$ for each prime $\fp \mid N$ we can define the forward map
  \[ \brackets{x, x \in \invertadele} \longmapsto \parens{\prod_\fp \fp^{v_\fp(x)}, \parens{\brackets{x u_\fp^{-v_\fp(x)}}}_{\fp \mid N}}\,. \]
  We show that this map is well defined: if $[x_1] = [x_2]$ for $x_1, x_2 \in \invertadele$, then $x_2/x_1 \in \hat{A}^\times \cap (1+N\hat{A})$; hence $v_\fp(x_2/x_1) = 0 \implies v_\fp(x_1) = v_\fp(x_2)$ for each prime $\fp$, so that $\prod_\fp \fp^{v_\fp(x_1)} = \prod_\fp \fp^{v_\fp(x_2)}$.
  Moreover, $x_2/x_1 \equiv 1 \pmod{N\hat{A}}$, so that for each prime $\fp \mid N$ we have that $x_1 u_\fp^{-v_\fp(x_1)} \equiv x_2 u_\fp^{-v_\fp(x_2)} \pmod{N\hat{A}}$.

  Now we define the inverse map, after choosing a uniformiser $u_\fp$ for \emph{every} prime $\fp$ (although the map defined actually only depends on the choice of uniformiser for $\fp \mid N$).
  The map is:
  \[ (J, [n, n \in A]) \longmapsto \brackets{\parens{u_\fp^{v_\fp(J)}}_{\fp \nmid N} \cup \parens{n u_\fp^{v_\fp(J)}}_{\fp \mid N}}\,. \]
  By composing this inverse map with the described forward map, we see that they are actually inverses, which establishes the isomorphism.
\end{proof}

Note that although the above isomorphism between the quotient of the group $\invertadele$ by the kernel of its action on $\LL_N^r$ and the product $\cJ(A) \times (A/N)^\times$ is explicit, it is not canonical since it depends on the choice of uniformisers $u_\fp$ for $\fp \mid N$.
The following, although a weaker result, is canonical:
\begin{proposition} \label{prop:invAdele_ExactSeq}
  There is a short exact sequence
  \[ 0 \into \parens{A/N}^\times \xinto{x_N} \rquotient{\invertadele}{\parens{\hat{A}^\times \cap (1+N\hat{A})}} \xonto{J} \cJ(A) \onto 0, \]
  the maps $x_N$ and $J$ given by
  \[
    \parens{A/N}^\times \ni [n, n \in A] \xmapsto{\mathmakebox[1em]{x_N}} \brackets{(1)_{\fp \nmid N} \cup (n)_{\fp \mid N}} \lrsptext{and}
    [x] \xmapsto{\mathmakebox[1em]{J}} \prod_\fp \fp^{v_\fp(x)} \in \cJ(A). \qedhere
  \]
\end{proposition}
\begin{proof}
  These maps are extracted from the proof of \cref{prop:invAdele_FracANstar}.
\end{proof}
Note that $J(x) = x\hat{A} \cap F$ is the unique fractional ideal such that $x \hat{A} = J(x) \hat{A}$.
There is a related short exact sequence; for its proof, keep in mind the identification
\[ \invertadele \supset F^\times (A/N)^{\!\times} \simeq \rquotient{F^\times \!\times\! (A/N)^{\!\times}}{(F^\times \!\cap\! (A/N)^{\!\times})} \simeq (\!\rquotient{F^\times\!}{\FF_q^\times}\!) \times (A/N)^{\!\times}\,. \]

\begin{proposition} \label{prop:invAdele_ExactSeq2}
  There is a short exact sequence
  \[ 0 \into F^\times \parens{A/N}^\times \xinto{x_{N,F}} \rquotient{\invertadele}{\parens{\hat{A}^\times \cap (1+N\hat{A})}} \xonto{[J]} \Cl(F) \onto 0, \]
  with the maps $x_{N,F}$ and $[J]$ given by
  \begin{align*}
    F^\times \parens{A/N}^\times \ni \parens{[f],[n,n \in A]} &\xmapsto{\mathmakebox[2em]{x_{N,F}}} \brackets{(f)_{\fp \nmid N} \cup (fn)_{\fp \mid N}}
    \shortintertext{and}
    [x] &\xmapsto{\mathmakebox[2em]{[J]}} \brackets{\prod_\fp \fp^{v_\fp(x)}}_{\Cl(F)}. \qedhere
  \end{align*}
\end{proposition}
\begin{proof}
  The proof is similar to that of \cref{prop:invAdele_ExactSeq}.
\end{proof}

We now investigate this action of $\invertadele$ on $\LL_N^r$ considered as the set of pairs $(\Lambda,\alpha)$ of a lattice $\Lambda$ with level $N$ structure $\alpha$.
But first, an observation on the action of $x \in \invertadele$ on ideal quotients:
\begin{lemma} \label{lem:invAdele_IdealAction}
  If $I \in \cJ(A)$ is a fractional ideal and $x \in \invertadele$ with corresponding $J = J(x)$, then there is a natural $A/N$-module isomorphism $x^{-1} : N^{-1}I/I \ionto N^{-1}J^{-1}I/J^{-1}I$ which makes the following diagram commute:
  \begin{cdiagram*} \label{diag:invAdele_IdealAction}
    N^{-1}I / I
      \arrow[d, hook, two heads, "\subset"']
      \arrow[r, hook, two heads, dotted, "x^{-1}"] &
    N^{-1}J^{-1}I / J^{-1}I = N^{-1} (x^{-1}I\hat{A} \cap F) / (x^{-1}I\hat{A} \cap F)
      \arrow[d, hook, two heads, "\subset"] \\
    N^{-1}I\hat{A} / I\hat{A}
      \arrow[r, hook, two heads, "x^{-1}"'] &
    N^{-1}x^{-1}I\hat{A}/ x^{-1}I\hat{A}
  \end{cdiagram*}
\end{lemma}

\begin{proposition} \label{prop:invAdele_L}
  If $x(\Lambda,\alpha) = (\Lambda',\alpha')$ for $x \in \invertadele$ with corresponding fractional ideal $J = J(x) \in \cJ(A)$, then $\Lambda' = J^{-1}\Lambda$ and $\alpha'$ is given as the composite
  \[ \parens{N^{-1}/A}^r \xionto{\ \alpha\ } N^{-1}\Lambda/\Lambda \xionto{x^{-1}} N^{-1}J^{-1}\Lambda/J^{-1}\Lambda. \qedhere \]
\end{proposition}
\begin{proof}
  Let $\Lambda = I_1 \psi_1 +\dotsb + I_r \psi_r$ for fractional ideals $I_i$ and $\psi \in \Psi^r$.
  Then choosing the rows $g'_i$ of $g' = (g'_1, \dotsc, g'_r)^T$ such that $F \cap g_i \hat{A}^r = I_i$ for each $i$, we have that $\psi\parens{F^r \cap g' \hat{A}^r} = \psi \cdot (I_1, \dotsc, I_r)^T = \Lambda$.
  Then since $\GL[r]{\hat{A}}$ acts transitively on the set of level structures for any given lattice, for a suitable $\gamma \in \GL[r]{\hat{A}}$ we will have that $(\Lambda,\alpha) = \Theta([(\psi,g)])$ for $g = g'\gamma$, since $g \hat{A}^r = g' \gamma \hat{A}^r = g' \hat{A}^r$.

  Now $x[(\psi,g)] = [(\psi, g x^{-1})]$, and so
  \begin{align*}
    \Lambda' &= \psi \parens{F^r \cap (g x^{-1}) \hat{A}^r} = \psi \parens{F^r \cap g (x^{-1}\hat{A})^r} \\
    &= \psi \parens{(J^{-1}F)^r \cap g (J^{-1}\hat{A})^r} = \psi \parens{J^{-1} \parens{F^r \cap g \hat{A}^r}} \\
    &= J^{-1} \psi \parens{F^r \cap g \hat{A}^r} = J(x)^{-1} \Lambda.
  \end{align*}

  Comparing the two corresponding versions of \cref{diag:psig_lattice_level} for $\alpha$ and $\alpha'$, note that $\psi$ is common in both.
  Hence for $\Lambda = I_1 \psi_1 +\dotsb + I_r \psi_r$ we have the following commutative diagram, making use of \cref{lem:invAdele_IdealAction}:
  \begin{cdiagram*}
    \parens{N^{-1}/A}^r
      \arrow[r, hook, two heads, "\alpha"]
      \arrow[rd, hook, two heads, dotted, "\alpha'"'] &
    N^{-1}\Lambda/\Lambda
      \arrow[r, hook, two heads, "\psi^{-1}"]
      \arrow[d, hook, two heads, dotted, "x^{-1}"'] &
    \bigoplus_{i = 1}^r N^{-1}I_i/I_i
      \arrow[d, hook, two heads, "x^{-1}"] \\
    & N^{-1}J^{-1}\Lambda/J^{-1}\Lambda &
    \bigoplus_{i = 1}^r N^{-1}J^{-1}I_i/J^{-1}I_i
      \arrow[l, hook, two heads, "\psi"]
  \end{cdiagram*}
  the dotted arrows defined so as to make the diagram commute.
\end{proof}

Hence the action of $\invertadele$ on $\LL_N^r$ induces a rigid analytic action of the set of fractional ideals on $\LL^r$ by $\Lambda \xmapsto{J} J^{-1} \Lambda$.


%% file: text/3__lattices_2.tex
\section{Lattices with metric structure} \label{sec:lattices_ii}

Some of the proofs in this section are lengthy and technical.
For these, sketches of proofs are provided; please see \cite{baker2020lattice} for full proofs.

\input{text/3_1_llattices.tex}

\input{text/3_2_ll_level.tex}

\input{text/3_3_ll_strata.tex}

\input{text/3_4_ll_GLrAN.tex}

\input{text/3_5_ll_invadele.tex}

%% file: text/3_1_llattices.tex
\subsection{A metric on the space of lattices}

\subsubsection{Metric definitions}

The space of all prelattices can be equipped with the structure of a metric space using the associated exponential functions:

\begin{definition} \label{def:dd_LL}
  The metric $\dd_V$ is defined on the space $V$ of all prelattices as follows:
  \[ \dd_V(\Lambda_1,\Lambda_2) = \sup_{\abs{z} \leq 1} \abs{e_{\Lambda_1}(z) -e_{\Lambda_2}(z)} \lsptext{for} \Lambda_1, \Lambda_2 \in V. \]
  If in addition $\Lambda_1, \Lambda_2$ are lattices, we may use the notation $\dd_\LL$ instead.
\end{definition}
It is not hard to see that the above is in fact a metric.



The restriction to $\abs{z} \leq 1$ in the above metric definition is largely superfluous in that the same local topology is generated if we replace it with $\abs{z} \leq R$, as shown by the following proposition, the proof of which will be postponed until \cpageref{proof:ddLLradius} in \cref{sec:modular_forms}:
\begin{restatable}{proposition}{ddLLradius} \label{prop:dd_LL_radius}
  Let $R > 0$ and let $\Lambda$ and $\Lambda'$ be lattices of rank $\leq r$, with $\Lambda'$ being variable. If $\Lambda' \to \Lambda$, then $\sup_{\abs{z} \leq R} \abs{e_{\Lambda'}(z)-e_{\Lambda}(z)} \to 0$.
\end{restatable}

With this metric, we can consider the `size' of a prelattice to be its distance to the zero lattice, which by an abuse of notation we denote as $0$.
As it turns out, prelattices all of whose nonzero elements are large are `small':
\begin{proposition} \label{prop:lattice_large}
  If a prelattice $\Lambda$ has $\minp_{\lambda \in \Lambda} \abs{\lambda} = R > 1$, then $\dd_V (\Lambda, 0) \leq R^{1-q}$.
\end{proposition}
\begin{proof}
  Apply \cref{prop:e_Lambda_leqR}.
\end{proof}
\begin{corollary}
  If $\minp_{\lambda \in \Lambda} \abs{\lambda} \to \infty$, then $\Lambda \to 0$.
\end{corollary}

More particularly, elements of a lattice which are `large' make a `small' difference in the topology induced by this metric as shown by \cref{prop:lattice_limit}.
But first, a technical lemma:
\begin{lemma} \label{lem:idealBoundedReps}
  For any nonzero fractional ideal $J \in \cJ(A)$, there is a bound $B_J > 0$ such that for every $f \in F$ there is a $n \in J$ with $\abs{f-n} < B_J$.
\end{lemma}
\begin{proof}
  Let $T \in A$ with $\abs{T} > 1$.
  Then $\FF_q[T]$ is a principal ideal domain and a subring of $A$.
  Thus $J$, as a torsion-free $\FF_q[T]$-module, is free:
  \[ J = j_1\FF_q[T] +\dotsb +j_m\FF_q[T] \lsptext{for} \FF_q(T)-\text{independent}\ j_i \in J. \]
  The field $F$ can be written similarly $F = j_1\FF_q(T) +\dotsb +j_m\FF_q(T)$.

  Now $\FF_q[T]$ is a Euclidean domain, and so for any $f' \in \FF_q(T)$ there is a $t \in \FF_q[T]$ with $\abs{f'-t} < 1$.
  Thus for any $f = j_1f'_1 +\dotsb +j_mf'_m \in F$ with $f'_i \in \FF_q(T)$, we can find corresponding $t_i \in \FF_q[T]$ with $\abs{f'_i-t_i} < 1$ for each $i$, so that with $n = j_1t_1 +\dotsb j_mt_m \in I$ we have
  \[ \abs{f-n} \leq \max_{i = 1}^m \ \parens{\abs{j_i}\cdot\abs{f'_i-t_i}} < \max_{i = 1}^m \abs{j_i} \eqdef B_J. \qedhere \]
\end{proof}

\begin{proposition} \label{prop:lattice_limit}
  Let $\Lambda$ be a fixed lattice of rank $r$ and $I \subseteq A$ a fixed ideal of $A$.
  Then for variable $\omega \in \CCi$ with $\dd(\omega,\Lambda) = \min_{\lambda\in\Lambda} \abs{\omega-\lambda} \to \infty$, we have that $\Lambda+I\omega \to \Lambda$ with respect to $\dd_V$.
\end{proposition}
\begin{proof}
  $\Lambda$ is a sublattice of $\Lambda+I\omega$, so from \cref{prop:e_Lambda_props}, we have that
  \[ e_{\Lambda+I\omega}(z) = e_{e_\Lambda(\Lambda+I\omega)}(e_\Lambda(z)) = e_{e_\Lambda(I\omega)}(e_\Lambda(z)). \]
  Now let $ R_1 = \dd(\omega,\Lambda) = \min_{\lambda\in\Lambda} \abs{\omega-\lambda}$ be large, so that $R_2 = \abs{e_\Lambda(\omega)}$ is large by \cref{prop:e_Lambda_infinity}.
  Then for any nonzero $a \in I$,
  \[ \abs{e_\Lambda(a\omega)} = \abs{a} \abs{e_\Lambda(\omega)} \prodp_{\lambda \in a^{-1}\Lambda/\Lambda} \frac{\abs{e_\Lambda(\omega)-e_\Lambda(\lambda)}}{\abs{e_\Lambda(\lambda)}}\,. \]
  Now by \cref{lem:idealBoundedReps}, there is a bounded set of representatives of $F/I$; hence since $e_\Lambda$ is entire the $e_\Lambda(\lambda)$ above are bounded independently of $a$, say by $L > 0$.
  Hence for large enough $R_2 > L$ independent of $a$,
  \[ \abs{e_\Lambda(a\omega)} = \abs{a} \frac{\abs{e_\Lambda(\omega)}^{\abs{a}^r}}{\prodp_{\lambda \in \Lambda} \abs{e_\Lambda(\lambda)}} \geq \abs{a} \brackets{\frac{\abs{e_\Lambda(\omega)}}{L}}^{\abs{a}^r} \geq R_2, \]
  and so $\dd(I\omega-\set{0},\Lambda)$ is large; thus $R_3 = \minp_{\lambda \in e_\Lambda(A\omega)} \abs{\lambda}$ is large.

  Now let $\sup_{\abs{z} \leq 1} \abs{e_\Lambda(z)} = m$ which is independent of $\omega$; since $R_3$ is large, we can also suppose that $R_3 > m$.
  Then
  \begin{align*}
    \dd_{\LL}(\Lambda+I\omega,\Lambda) &= \sup_{\abs{z} \leq 1} \abs{e_{e_\Lambda(I\omega)}(e_\Lambda(z))-e_\Lambda(z)} \\
    &\leq \sup_{\abs{z} \leq m} \abs{e_{e_\Lambda(I\omega)}(z)-z} \\
    &\leq \sup_{\abs{z} \leq m} \abs{z}^q R_3^{1-q} = m^q R_3^{1-q} \lsptext{by \cref{prop:e_Lambda_leqR}}\\
    &\to 0 \lsptext{as} R_3 \to \infty.
  \end{align*}
  Thus $\Lambda+I\omega \to \Lambda$ as desired.
\end{proof}

\subsubsection{Metric completeness}

We have just seen an example where a variable lattice of rank $r+1$ tended to a lattice of rank $r$.
In general, we have the result of \cref{coro:lattice_limit_rank}; but first, a proposition:
\begin{proposition} \label{prop:LLleqR_complete}
  $\LL^{\leq r}$ is complete.
\end{proposition}
\begin{proof}[Sketch of proof]
  Given a Cauchy sequence $\Lambda_n$ of lattices, we can show that the sequence of associated inverse functions $e_{\Lambda_n}^{-1}$ converge, and similarly the Drinfeld modules $\phi^n$ converge to another Drinfeld module $\phi$.
  This Drinfeld module corresponds to another lattice $\Lambda$ which we show is the limit of the $\Lambda_n$.
\end{proof}

\begin{corollary} \label{coro:lattice_limit_rank}
  If $\parens{\Lambda_n}_{n = 1}^\infty$ is a sequence of lattices of rank $\leq R$ and $\Lambda$ a lattice of rank $r$ such that $\Lambda_n \to \Lambda$, then $r \leq R$.
\end{corollary}
\begin{proof}
  The lattices $\Lambda_n$ are elements of $\LL^{\leq R}$, which is complete by \cref{prop:LLleqR_complete}.
  Since the $\Lambda_n$ form a Cauchy sequence, they hence converge to a lattice $\Lambda_R \in \LL^{\leq R}$.
  But then $\Lambda_R = \Lambda$, and so $\Lambda$ is of rank at most $R$.
\end{proof}

\begin{proposition} \label{prop:LLR_dense}
  $\LL^r$ is a dense subset of $\LL^{\leq r}$.
\end{proposition}
\begin{proof}[Sketch of proof]
  For a lattice $\Lambda$ of rank $s \leq r$, we consider the rank $r$ lattices $\Lambda_n = \Lambda +Af^n \omega_{s+1} +\dotsb +Af^n\omega_r$ where $f \in F$ has $\abs{f} > 1$ and the $\omega_i$ are $F_\infty$-linearly independent with $\Lambda$.
  Since the lattice points added to $\Lambda_n$ become arbitrarily far from $\Lambda$, it follows that $\Lambda_n \to \Lambda$.
\end{proof}

By \cref{prop:lattice_limit}, we see that the space $\LL^r$ of lattices of rank $r$ is not complete with respect to $\dd_\LL$.
However, combining the two previous propositions we have:
\begin{theorem} \label{prop:LLleqR_completion}
  The space $\LL^{\leq r}$ is the completion of $\LL^r$.
\end{theorem}
\begin{proof}
  By \cref{prop:LLleqR_complete,prop:LLR_dense}, $\LL^{\leq r}$ is complete with dense subset $\LL^r$.
\end{proof}

We also note the following trivial proposition without proof for explicitness:
\begin{proposition} \label{prop:LLleqR_union} ~ \vspace{-12pt}
  \[ \LL^{\leq r} = \bigsqcup_{0 \leq s \leq r} \LL^s. \qedhere \]
\end{proposition}

\begin{definition} \label{def:LLR_strata}
  In the above decomposition of the space $\LL^{\leq r}$ into a disjoint union of $\LL^s$ for $0 \leq s \leq r$, each of the uniands\footnote{Uniands are to unions as summands are to sums.} $\LL^s$ is called a \emph{stratum} of \emph{dimension} $s$ and \emph{codimension} $r-s$; if $s < r$, the $\LL^s$ are also called \emph{boundary strata}.
\end{definition}

\subsubsection{Topology of the irreducible components}
\label{subsubsec:irredCompo}

In this subsubsection we let $C$ denote an irreducible component of $\LL^r \simeq \LL_A^r$, these components defined by the values of $\pi(\Lambda) \in \Cl(F)$ for $\Lambda \in \LL^r$. (Recall that $\pi$ is defined in \cref{prop:Comps_ClF_AN*_Bijection}, with $N$ being irrelevant due to \cref{para:piN_noN}.)

So we see that the space $\LL^r$ has boundary $\bd{\LL^r} = \bd{\LL^{\leq r}} = \LL^{\leq r-1}$ in the complete metric space $\LL^{\leq r}$ consisting of the $\Lambda$ with $\Lambda$ of rank strictly less than $r$.
We will see that each of the irreducible components of $\LL^r$ shares the same boundary.

First, let us look closer at $\LL^{\leq r-1}$:
\begin{proposition} \label{prop:LLR_bdClosed}
  $\LL^{\leq r-1}$ is a closed subset of $\LL^{\leq r}$, so $\LL^r = \LL^{\leq r}-\LL^{\leq r-1}$ is open.
\end{proposition}
\begin{proof}
  Let $(\Lambda_n)_{n = 0}^\infty$ be a Cauchy sequence in $\LL^{\leq r-1}$, so that each $\Lambda_n$ has rank strictly less than $r$.
  Then since $\LL^{\leq r}$ is complete, this sequence has a limit $\Lambda \in \LL^{\leq r}$.
  Hence by \cref{coro:lattice_limit_rank} we have that $\Lambda$ has rank less than $r$, so that $\Lambda \in \LL^{\leq r-1}$.
\end{proof}

\begin{proposition} \label{prop:LLR_irredComponentDense}
  $C$ is a dense subset of $C \cup \LL^{\leq r-1}$.
\end{proposition}
\begin{proof}
  Let $\Lambda \in \LL^{\leq r-1}$ with $\Lambda = I_1\psi_1 +\dotsb +I_s\psi_s$ of rank $s < r$, let $f \in F$ with $\abs{f} > 1$, and let $\psi_{s+1}, \dotsc, \psi_r \in \CCi$ be $F_\infty$-linearly independent with $\Lambda$.
  Also let $I \in \cJ(A)$ such that $[I_1 \dotsm I_s I]_{\Cl(F)} = \pi(C)$, let
  \[ \Lambda' = A\psi_{s+1} +\dotsb +A\psi_{r-1} +I\psi_r \]
  be a lattice of rank $r-s$ and let $\Lambda_n = \Lambda+f^n\Lambda'$ be lattices of rank $r$ for $n \in \NN$.
  Then each $\Lambda_n$ has
  \[ \pi(\Lambda_n) = [I_1 \dotsm I_s A \dotsm A I]_{\Cl(F)} = \pi(C), \]
  so that each $\Lambda_n$ also lies in the component $C$, and similarly to the proof of \cref{prop:LLR_dense} we have that $\Lambda_n \to \Lambda$.
\end{proof}

\begin{proposition} \label{prop:LLR_irredComponentOpen}
  $C$ is open in $\LL^{\leq r}$.
\end{proposition}
This proposition's proof is lengthy and technical; for the full proof, see \cite{baker2020lattice}.

\begin{corollary} \label{coro:LLR_irredComponentBdComplete}
  $C \cup \LL^{\leq r-1}$ is closed in $\LL^{\leq r}$.
\end{corollary}
\begin{proof}
  \newcommand*{\md}{\mathcal{D}}
  If $\md$ denotes the set of irreducible components of $\LL^r$, then
  \[ \LL^{\leq r} -\parens{C \cup \LL^{\leq r-1}} = \LL^r -C = \bigcup_{C' \in \md-\set{C}} C' \lsptext{is open.} \qedhere \]
\end{proof}

\begin{theorem} \label{thm:LLR_irredComponent_completion}
  $C \cup \LL^{\leq r-1}$ is the completion of $C$, \ie the closure of $C$ in $\LL^{\leq r}$.
\end{theorem}
\begin{proof}
  By \cref{prop:LLR_irredComponentDense,coro:LLR_irredComponentBdComplete}, $C$ is dense in $C \cup \LL^{\leq r-1}$, which is closed in $\LL^{\leq r}$ and hence complete.
\end{proof}

%% file: text/3_2_ll_level.tex
\subsection{A metric on the space of lattices with level structure}

\subsubsection{Metric definition}

\begin{definition} \label{def:LLNR}
  We define the space $\LLNRi$ as the space of all pairs $(\Lambda,\iota)$ of a lattice $\Lambda$ of rank $\leq r$ and an $A/N$-module injection $\iota : N^{-1}\Lambda/\Lambda \into \parens{N^{-1}/A}^r$.
  Such an $\iota$ is called an \emph{$r$-inverse level $N$ structure}, or simply an inverse level $N$ structure if the value of $r$ is understood.

  The notation above with a backwards arrow is used to remind the reader that $\LLNRi$ is the collection of lattices with \emph{inverse} level $N$ structure.
\end{definition}

\begin{paragraph} \label{para:LLNRimainCompoIdentLLNR}
  If in particular $\Lambda$ in the above definition is of rank \emph{equal to} $r$, then considering the sizes of $\parens{N^{-1}/A}^s$ and $N^{-1}\Lambda/\Lambda$, the injection $\iota$ must in fact be an isomorphism, and so the subset of $\LLNRi$ with $\Lambda$ of rank $r$ is isomorphic to $\LL_N^r$, with $\iota^{-1}$ playing the role of the relevant level $N$ structure (hence calling $\iota$ an \emph{inverse} level $N$ structure).
\end{paragraph}

\begin{definition} \label{def:LLNR_recipRoot}
  For $(\Lambda,\iota) \in \LLNRi$, we define
  \[ \mu_{\Lambda,\iota} : \parens{N^{-1}/A}^r \longto \CCi, \quad l \longmapsto \begin{cases} e_\Lambda(\iota^{-1}(l))^{-1} & \sptext{if} l \in \Im{\iota}-\set{0} \\ 0 & \sptext{otherwise} \end{cases}. \qedhere \]
\end{definition}
Note that $\mu_{\Lambda,\iota}(l) \neq 0 \iff l \in \Im{\iota}-\set{0}$.

\begin{paragraph}
  Note that in \cref{def:LLNR_recipRoot}, $e_{\Lambda} \circ \iota^{-1}$ is an $A/N$-module bijection from $\Im{\iota} \subseteq \parens{N^{-1}/A}^r$ to the $N$-division points $\phi^\Lambda[N]$ of the Drinfeld module $\phi^\Lambda$ associated to $\Lambda$; in particular, if $\Lambda$ is of rank $r$, then $e_\Lambda \circ \iota^{-1}$ is a Drinfeld module level $N$ structure for $\phi^\Lambda$.
\end{paragraph}
So we have the following proposition:

\begin{proposition} \label{prop:LLNR_muModuleFactorisation}
  We can factorise the Drinfeld module-associated polynomial $\phi^\Lambda_N$ from \cref{def:IdealDModule} as
  \[ \phi^\Lambda_N(X) = X \prodp_{l \in \Im{\iota}} \parens{1-\frac{X}{e_\Lambda(\iota^{-1}(l))}} = X \prod_{l \in (N^{-1}/A)^r} \parens{1-\mu_{\Lambda,\iota}(l)X}. \qedhere \]
\end{proposition}
\begin{proof}
  $\Im{\iota^{-1}} = N^{-1}\Lambda/\Lambda$; this proves the first equality.
  The second follows from the definition of $\mu_{\Lambda,\iota}(l)$, which is $0$ for $l \in \parens{\parens{N^{-1}/A}^r -\Im{\iota}} \cup \set{0}$.
\end{proof}

As in the case of lattices without level structure, we can define a metric on the space of lattices with level structure and find its completion.
Here is our metric, related to $\dd_\LL$, and defined on the space $\overleftarrow{\LL_{N}^r}$:
\begin{definition} \label{def:dd_LLNR}
  The metric $\dd_{\LLNRi}$ on $\LLNRi$ is defined by
  \[ \dd_{\LLNRi} \bigl((\Lambda_1,\iota_1),(\Lambda_2,\iota_2)\bigr) = \dd_\LL(\Lambda_1,\Lambda_2) +\sum_{l \in (N^{-1}/A)^r} \abs{\mu_{\Lambda_1,\iota_1}(l)-\mu_{\Lambda_2,\iota_2}(l)}. \qedhere \]
\end{definition}

In \cref{def:dd_LLNR}, the sets $\Im{\iota_1}$ and $\Im{\iota_2}$ may have nonempty set difference, both for $\Im{\iota_1}-\Im{\iota_2}$ and $\Im{\iota_2}-\Im{\iota_1}$; so as special cases of the above definition we have the following:
\begin{corollary} ~ \begin{itemize}
  \item If $\Lambda_2 = 0$ is the zero lattice, then
  \[ \dd_{\LLNRi} \bigl((\Lambda,\iota),(0,0)\bigr) = \dd_\LL(\Lambda,0) +\sump_{\lambda \in N^{-1}\Lambda/\Lambda} \abs{\frac{1}{e_\Lambda(\lambda)}}\,. \]
  \item If $\Im{\iota_2} \subseteq \Im{\iota_1}$, then
  \begin{multline*}
      \dd_{\LLNRi} \bigl((\Lambda_1,\iota_1),(\Lambda_2,\iota_2)\bigr) = \dd_\LL(\Lambda_1,\Lambda_2) +\sum_{\substack{l \in \Im{\iota_1} \\ l \notin \Im{\iota_2}}} \abs{\frac{1}{e_{\Lambda_1}(\iota_1^{-1}(l))}} \\ +\sump_{l \in \Im{\iota_2}} \abs{\frac{1}{e_{\Lambda_1}(\iota_1^{-1}(l))}-\frac{1}{e_{\Lambda_2}(\iota_2^{-1}(l))}}\,. \qedhere
  \end{multline*}
\end{itemize} \end{corollary}

This metric space structure on $\LLNRi$ induces a topology in which lattices tending to zero behave similarly to the case without level structure:
\begin{proposition} \label{prop:LLNRi_zeroNeighbour}
  Let $\Lambda$ be a variable lattice of rank at most $r$ with a variable $r$-inverse level $N$ structure $\iota$. If $\dd_L(\Lambda,0) \to 0$ then $\dd_{\LLNRi}\bigl((\Lambda,\iota),0\bigr) \to 0$.
\end{proposition}
\begin{proof}
  Since $\dd_L(\Lambda,0) \to 0$, we have that $R = \minp_{\lambda \in \Lambda} \abs{\lambda}$ is large.
  Now let $l \in \Im{\iota}-\set{0}$ and $a \in N-\set{0}$ where $a$ is fixed.
  Then for a representative $\lambda' \in \iota^{-1}(l)$ we have that $e_\Lambda(\iota^{-1}(l))^{-1} = \sum_{\lambda \in \Lambda} (\lambda'+\lambda)^{-1}$, so that
  \[ \abs{\frac{1}{a \cdot e_\Lambda(\iota^{-1}(l))}} \leq \max_{\lambda \in \Lambda} \frac{1}{\abs{a\lambda'+a\lambda}} \leq \max_{\lambda \in \Lambda} \frac{1}{R} = \frac{1}{R}, \]
  since each $a\lambda'+a\lambda \in \Lambda-\set{0}$.
  Hence
  \[ \sump_{l \in \Im{\iota}} \abs{\frac{1}{e_\Lambda(\iota^{-1}(l))}} \leq \frac{\abs{a}\parens{\size{\Im{\iota}}-1}}{R} < \frac{\abs{a}\sizep{A/N}^r}{R} \longto 0. \qedhere \]
\end{proof}

For rank $s \leq r$, the space $\LLi{N}{s}$ can be considered in multiple ways as a subspace of $\LLi{N}{r}$; we now compare the metrics on this subspace both as $\LLi{N}{s}$ and as a sub-metric space of $\LLi{N}{r}$:
\begin{proposition} \label{prop:dd_LLNsrEquality}
  Any $A/N$-module injection $\delta : \parens{N^{-1}/A}^s \into \parens{N^{-1}/A}^r$ induces an injection $\LLi{N}{s} \into \LLi{N}{r},\ (\Lambda,\iota) \mapsto (\Lambda, \delta \circ \iota)$.
  The metrics $\dd_{\LLi{N}{s}}$ and $\dd_{\LLi{N}{r}}$ agree on $\LLi{N}{s} \subseteq \LLi{N}{r}$.
\end{proposition}
\begin{proof}
  Left to the reader.
\end{proof}
The metric $\dd_{\LLi{N}{r}}$ is thus independent of the rank $r$ in a sense.

\subsubsection{Metric Completeness}

As in the case of level-less lattices, $\LLNRi$ is the metric completion of $\LL_N^r$ as shown by the following results:

\begin{proposition} \label{prop:LLNRi_complete}
  $\LLNRi$ is complete.
\end{proposition}
\begin{proof}[Sketch of proof]
  For a Cauchy sequence $(\Lambda_n,\iota_n)$, we have that $\Lambda_n$ is a Cauchy sequence in $\LL^{\leq r}$ and hence converges to $\Lambda \in \LL^{leq r}$.
  For each $l \in \parens{N^{-1}/A}^r$ we have that $(\mu_{\Lambda_n,\iota_n}(l))_{n = 0}^\infty$ is a Cauchy sequence in $\CCi$ and hence converges; we call the limit $\mu(l)$, where $\mu : \parens{N^{-1}/A}^r \to \CCi$.
  We then show that each $\mu(l) \in \phi^\Lambda[N]$, after which we can define $\iota = e_\Lambda^{-1} \circ (1/\mu)$ as the needed inverse level structure.
\end{proof}

\begin{proposition} \label{prop:LLNR_dense}
  $\LL_N^r$ is a dense subset of $\LLNRi$.
\end{proposition}
\begin{proof}[Sketch of proof]
  As in the proof of \cref{prop:LLR_dense}, for $(\Lambda,\iota) \in \LLNRi$ of rank $s$ with $0 \leq s \leq r$ we define a sequence $(\Lambda_n,\iota_n)$ of lattices with level structure where $\Lambda \subset \Lambda_n$ and the added lattice points grow far from $\Lambda$.
  We then show that $\mu_{\Lambda_n,\iota_n}(l) \to \mu_{\Lambda,\iota}(l)$ separately for $l \in \Im{\iota}-\set{0}$ and for $l \in \Im{\iota_n}-\Im{\iota}$.
\end{proof}

\begin{theorem} \label{prop:LLNR_completion}
  The space $\LLNRi$ is the completion of $\LL_N^r$.
\end{theorem}
\begin{proof}
  By \cref{prop:LLNRi_complete,prop:LLNR_dense}, $\LL_N^r$ is dense in the complete $\LLNRi$.
\end{proof}

\subsubsection{Topology of the irreducible components}

So we see that the space $\LL_N^r$ has boundary $\bd{\LL_N^r} = \bd{\LLNRi}$ in the complete metric space $\LLNRi$ consisting of the $(\Lambda,\iota)$ with $\Lambda$ of rank strictly less than $r$.
Denoting this boundary by $\bd_N^r$, we will see that each of the irreducible components of $\LL_N^r$ possesses the same boundary.

First, let us look closer at $\bd_N^r$:
\begin{proposition} \label{prop:LLNR_bdClosed}
  $\bd_N^r$ is closed as a subset of $\LLNRi$, so that $\LL_N^r = \LLNRi-\bd_N^r$ is open.
\end{proposition}
\begin{proof}
  Let $(\Lambda_n,\iota_n)_{n = 0}^\infty$ be a Cauchy sequence in $\bd_N^r$, so that each $\Lambda_n$ has rank strictly less than $r$.
  Then since $\LLNRi$ is complete, this sequence has a limit $(\Lambda,\iota) \in \LLNRi$.
  Now by the definitions of $\dd_{\LLNRi}$ and $\dd_\LL$, we also have that $\Lambda_n \to \Lambda$; hence by \cref{coro:lattice_limit_rank} we have that $\Lambda$ has rank less than $r$, so that $(\Lambda,\iota) \in \bd_N^r$ as desired.
\end{proof}

For the rest of the results in this subsubsection, we let $C$ denote a fixed irreducible component of $\LL_N^r$.
\begin{proposition} \label{prop:LLNR_irredComponentDense}
  $C$ is a dense subset of $C \cup \bd_N^r$.
\end{proposition}
\begin{proof}[Sketch of proof]
  Similarly to the proof of \cref{prop:LLR_dense}, for a lattice $\Lambda$ of rank $s < r$ we can extend it to lattices $\Lambda_n = \Lambda +f^n\Lambda'$ where $\abs{f} > 1$, the lattice $\Lambda' = A\psi_{s+1} +\dotsb +A\psi_{r-1} +I\psi_r$, and the ideal $I$ is such that $\Lambda_n \in \pi_N(C)$.
  For the level structures $\iota_n$, we essentially define them separately on $\Lambda$ using $\iota$ and on $\Lambda'$ in such a way that $(\Lambda_n,\iota_n) \in C$ for each $n$.
  The rest of the proof proceeds as in the proof of \cref{prop:LLNR_dense} to show that $(\Lambda_n,\iota_n) \to (\Lambda,\iota)$.
\end{proof}

\begin{proposition} \label{prop:LLNR_irredComponentOpen}
  $C$ is open in $\LLNRi$.
\end{proposition}
As in the proof of \cref{prop:LLR_irredComponentOpen}, this proof is technical and omitted here; for a full proof, see \cite{baker2020lattice}.

\begin{corollary} \label{coro:LLNR_irredComponentBdClosed}
  $C \cup \bd_N^r$ is closed in $\LLNRi$.
\end{corollary}
\begin{proof}
  \newcommand*{\md}{\mathcal{D}}
  If $\md$ denotes the set of irreducible components of $\LL_N^r$, then
  \[ \LLNRi -\parens{C \cup \bd_N^r} = \LL_N^r -C = \bigcup_{C' \in \md-\set{C}} C' \lsptext{is open.} \qedhere \]
\end{proof}

\begin{theorem} \label{thm:LLNR_irredComponent_completion}
  $C \cup \bd_N^r$ is the completion of $C$, \ie the closure of $C$ in $\LLNRi$.
\end{theorem}
\begin{proof}
  By \cref{prop:LLNR_irredComponentDense}, $C$ is dense in $C \cup \bd_N^r$, which is closed in $\LLNRi$ by \cref{coro:LLNR_irredComponentBdClosed} and hence complete.
\end{proof}

%% file: text/3_3_ll_strata.tex
\subsection{Boundary strata for lattices with level structure}

We now decompose $\LLNRi$ into a disjoint union of $\LL_N^s$ for $0 \leq s \leq r$, but first we will need some technical ring- and module-theoretic results, the proofs of most of which are omitted:

\subsubsection{The structure of finitely generated free $A/N$-modules}


\begin{lemma} \label{lem:lIndepBasis}
  If $S$ is a linearly independent subset of the free $A/N$-module $\parens{A/N}^r$, then $S$ can be extended to a basis.
\end{lemma}

\begin{proposition} \label{prop:SurInjFree} ~ \begin{enumerate}[leftmargin=2\bigskipamount]
  \item If $f : \parens{A/N}^r \onto \parens{A/N}^s$ is an $A/N$-module surjection, then $\ker f \simeq \parens{A/N}^{r-s}$ is free.
  \item If $i : \parens{A/N}^s \into \parens{A/N}^r$ is an $A/N$-module injection, then $\rquotient{\parens{A/N}^r}{\Im{i}} \simeq \parens{A/N}^{r-s}$ is free. \qedhere
\end{enumerate} \end{proposition}
\begin{proof} ~ \begin{enumerate}[leftmargin=2\bigskipamount]
  \item Since $\parens{A/N}^s$ is free it is projective and hence $f$ has a right inverse $g : \parens{A/N}^s \into \parens{A/N}^r$.
  Then for a basis $S$ of $\parens{A/N}^s$, $g(S)$ is linearly independent in $\parens{A/N}^r$, and so by \cref{lem:lIndepBasis} can be extended to a basis $T = \set{t_1, \dotsc, t_r}$, where without loss of generality $S = \set{f(t_1), \dotsc, f(t_s)}$ and $f(t_i) = 0$ for $s < i \leq r$.
  Then $\ker{f} \simeq \sum_{s < i \leq r} \parens{A/N} t_i \simeq \parens{A/N}^{r-s}$.
  \item Let $B$ be a basis for $\parens{A/N}^s$.
  Then $i(B)$ is linearly independent in $\parens{A/N}^r$ and so by \cref{lem:lIndepBasis} can be extended to a basis $T = \set{t_1, \dotsc, t_r}$ where without loss of generality $i(B) = \set{t_1, \dotsc, t_s}$.
  Then if $h$ is the projection $h : \parens{A/N}^r \onto \rquotient{\parens{A/N}^r}{\Im{i}}$, the quotient has as a basis $\set{h(t_{s+1}), \dotsc, h(t_r)}$ and thus is free. \qedhere
\end{enumerate} \end{proof}

\begin{definition}
  For each pair of integers $r,s$ such that $0 \leq s \leq r$, we define the sets
  \begin{align*}
      \Sur_N^{r,s} &= \set{\parens{A/N}^r \onto \parens{A/N}^s} \quad\text{of $A/N$-module surjections (epimorphisms), and} \\
      \Inj_N^{s,r} &= \set{\parens{A/N}^s \into \parens{A/N}^r} \quad\text{of $A/N$-module injections (monomorphisms).} \qedhere
  \end{align*}
\end{definition}

\begin{paragraph} \label{para:SurInjSplit}
  By \cref{prop:SurInjFree}, the elements of $\Sur_N^{r,s}$ and $\Inj_N^{s,r}$ are actually all split epimorphisms and split monomorphisms, respectively.
\end{paragraph}

\begin{proposition} \label{prop:GLactFreeTrans} ~ \begin{enumerate}
  \item The natural left action of $\GL[s]{A/N}$ on $\Sur_N^{r,s}$ is free.
  \item The natural right action of $\GL[r]{A/N}$ on $\Sur_N^{r,s}$ is transitive.
  \item The natural right action of $\GL[s]{A/N}$ on $\Inj_N^{s,r}$ is free.
  \item The natural left action of $\GL[r]{A/N}$ on $\Inj_N^{s,r}$ is transitive. \qedhere
\end{enumerate} \end{proposition}

We now turn to counting the cardinalities of the sets $\Sur_N^{r,s}$ and $\Inj_N^{s,r}$:
\begin{proposition} \label{prop:SurInjCount}
  \[ \size{\Inj_N^{s,r}} = \size{\Sur_N^{r,s}} = \rquotient{\size{\GL[r]{A/N}}}{\size{\GL[r-s]{A/N}} \cdot \abs{N}^{s(r-s)}} \qedhere \]
\end{proposition}
\begin{proof}[Sketch of proof]
  Considering the transitive right action of $G = \GL[r]{A/N}$ on $\Sur_N^{r,s}$, we only need to calculate the cardinality of the stabiliser group $G_f$ for a fixed  surjection $f$.
  We then consider a $\gamma \in G_f$ as a matrix with entries in $A/N$; translating $f \circ \gamma = f$ to conditions on the entries of $\gamma$ then yields the desired count.
\end{proof}

\begin{corollary} \label{coro:LLNRi_numInvStructs}
  If $\Lambda$ is a lattice of rank $s \leq r$, then the number of $r$-inverse level $N$ structures for $\Lambda$ is equal to
  \[ \rquotient{\size{\GL[r]{A/N}}}{\size{\GL[r-s]{A/N}} \cdot \abs{N}^{s(r-s)}}. \qedhere \]
\end{corollary}
\begin{proof}
  We want to count the $A/N$-module injections $N^{-1}\Lambda/\Lambda \into \parens{N^{-1}/A}^r$.
  By \cref{para:lattice_quot}, this is equal to the number of $A/N$-module injections $\parens{A/N}^s \into \parens{A/N}^r$; \ie the cardinality of $\Inj_N^{s,r}$.
\end{proof}

\begin{definition} \label{def:FreeSubsAndMaps}
  For an ideal $N$ of $A$ and integers $r,s$ with $0 \leq s \leq r$, we let $\Free_N^{s,r}$ be the set of free $A/N$-submodules of $\parens{A/N}^r$ of rank $s$, and define $\Free_N^r = \sqcup_{s = 0}^r \Free_N^{s,r}$.

  Also, we say that
  \[ \delta : \Free_N^r \longto \sqcup_{s = 0}^r \Inj_N^{s,r}, \ U \longmapsto \delta_U \]
  is an \emph{injective selection of $\Free_N^r$} if for each $U \in \Free_N^{s,r}$, $\delta_U \in \Inj_N^{s,r}$ has $\Im{\delta_U} = U$.
  Similarly,
  \[ \epsilon : \Free_N^r \longto \sqcup_{s = 0}^r \Sur_N^{r,s}, \ U \longmapsto \epsilon_U \]
  is a \emph{surjective selection of $\Free_N^r$} if for each $U \in \Free_N^{s,r}$, $\epsilon_U \in \Sur_N^{r,s}$ has $\ker{\epsilon_U} = U$.
\end{definition}

\begin{proposition} \label{prop:FreeNsrBijInjQuot}
  There are bijections
  \begin{align*}
    \rquotient{\Inj_N^{s,r}}{\GL[s]{A/N}} &\longionto \Free_N^{s,r}, & [i] &\longmapsto \Im{i} \lsptext{and} \\
    \lquotient{\GL[s]{A/N}}{\Sur_N^{r,s}} &\longionto \Free_N^{s,r}, & [f] &\longmapsto \ker{f}. \qedhere
  \end{align*}
\end{proposition}
\begin{proof}
  We will prove the first bijection; the second is proven analogously.
  By \cref{prop:SurInjFree} $\Im{i} \in \Free_N^{s,r}$ for each $i \in \Inj_N^{s,r}$, and if two $i_1, i_2 \in \Inj_N^{s,r}$ satisfy $ I = \Im{i_1} = \Im{i_2}$, then they induce bijections $i_1', i_2' : \parens{A/N}^s \ionto I$; thence $\gamma = i_1'^{-1} \circ i_2' \in \GL[s]{A/N}$ satisfies $i_1 \circ \gamma = i_2$ and so $i_1, i_2$ are equivalent under the action of $\GL[s]{A/N}$.
  Moreover, $\Im{i} = \Im(i \circ \gamma)$ for $\gamma \in \GL[s]{A/N}$ and $i \in \Inj_N^{s,r}$.
  Also, for any $U \in \Free_N^{s,r}$, $U$ has a basis by \cref{lem:lIndepBasis}; thus $\parens{A/N}^s \simeq U \subseteq \parens{A/N}^r$, and the composition $i \in \Inj_N^{s,r}$ of these relations has $\Im{i} = U$, proving the bijection.
\end{proof}

\begin{paragraph} \label{para:GLactInjSurSel}
  By \cref{prop:FreeNsrBijInjQuot}, injective and surjective selections of $\Free_N^r$ exist, by choosing an element in each class of $\rquotient{\Inj_N^{s,r}}{\GL[s]{A/N}}$ and $\lquotient{\GL[s]{A/N}}{\Sur_N^{r,s}}$ for $s$ between $0$ and $r$, respectively.
\end{paragraph}

\subsubsection{Decomposition of $\LLNRi$ into a union of $\LL_N^s$}

\begin{paragraph} \label{para:AN.NinvASwap}
  For the following theorem, we fix isomorphisms $\parens{N^{-1}/A}^s \ionto \parens{A/N}^s$ for $s \in \NNO$, and will pass through these isomorphisms often without mention.
  More generally, we will henceforth view $\Inj_N^{s,r}$ and $\Sur_N^{r,s}$ as maps between $\parens{N^{-1}/A}^r$ and $\parens{N^{-1}/A}^s$ instead of between $\parens{A/N}^r$ and $\parens{A/N}^s$, and similarly for $\Free_N^r$, $\Free_N^{s,r}$ and their injective and surjective selections.
\end{paragraph}

Analogously to \cref{prop:LLleqR_union}, we can view $\LLNRi$ as a disjoint union of $\LL_N^s$ for $0 \leq s \leq r$, but on the contrary, we usually have multiple copies of $\LL_N^s$. 
\begin{theorem} \label{prop:LLNRi_union}
  If $\delta$ is an injective selection of $\Free_N^r$, then we have a bijection
  \begin{align*}
    \bigsqcup_{\substack{0 \leq s \leq r \\ U \in \Free_N^{s,r}}} \LL_N^s &\longisoto \LLNRi \\
    \parens{\Lambda, \alpha}_U &\longmapsto (\Lambda, \delta_U \circ \alpha^{-1}). \qedhere
  \end{align*}
\end{theorem}
\begin{proof}
  \newcommand*{\apt}{\operatorname{apt}}
  \newcommand*{\tog}{\operatorname{tog}}

  We show that the above map from the union over the $\Free_N^{s,r}$ to $\LLNRi$ is a bijection by describing its inverse: let $(\Lambda,\iota) \in \LLNRi$ with $\Lambda$ of rank $s$.
  Then $\Im{\iota}$ is free of rank $s$ via similar arguments as before, and $\iota$ and $\delta_{\Im{\iota}}$ induce isomorphisms $\iota' : N^{-1}\Lambda/\Lambda \ionto \Im{\iota}$ and $\delta_{\Im{\iota}}' : \parens{N^{-1}/A}^s \ionto \Im{\iota}$; hence we obtain an $s$-level $N$ structure $\alpha_\iota = \iota'^{-1} \circ \delta_{\Im{\iota}}'$ for $\Lambda$ which satisfies $\delta_{\Im{\iota}} \circ \alpha_\iota^{-1} = \iota$.
  Call the map given in the theorem statement by the name `$\tog$' and this proposed inverse by the name `$\apt$'\footnote{Short for `together' and `apart', respectively.}, so that $\tog{(\Lambda,\alpha)_U} = (\Lambda, \delta_U \circ \alpha^{-1})$ and $\apt{(\Lambda,\iota)} = (\Lambda,\alpha_\iota)_{\Im{\iota}}$.
  Then
  \begin{align*}
    \tog{\apt{(\Lambda,\iota)}} &= \tog{(\Lambda, \alpha_\iota)_{\Im{\iota}}} = (\Lambda, \delta_{\Im{\iota}} \circ \alpha_\iota^{-1}) = (\Lambda, \iota) \lsptext{and} \\
    \apt{\tog{(\Lambda,\alpha)_U}} &= \apt{(\Lambda, \delta_U \circ \alpha^{-1})} = (\Lambda, \alpha_{\delta_U \circ \alpha^{-1}})_{\Im(\delta_U \circ \alpha^{-1})} \\
    &= (\Lambda, (\delta_U \circ \alpha^{-1})'^{-1} \circ \delta'_{\Im{\delta_U}})_{\Im{\delta_U}} = (\Lambda, (\delta_U \circ \alpha^{-1})'^{-1} \circ \delta'_U)_U \\
    &= (\Lambda, \alpha)_U,
  \end{align*}
  which completes the proof.
\end{proof}

\begin{definition} \label{def:LLNRi_bdStrata}
  In the above decomposition of $\LLNRi$ into a disjoint union of $\LL_N^s$ for $0 \leq s \leq r$, each of the uniands\footnote{Uniands are to unions as summands are to sums.} $\LL_N^s$ is called a \emph{stratum of dimension $s$} and \emph{codimension} $r-s$; if $s < r$, they are also called \emph{boundary strata}, and in the case where $s = r$ the stratum $\LL_N^r$ is called the \emph{main stratum}.
\end{definition}


\subsubsection{The number of strata}

\begin{corollary} \label{coro:LLNRi_NumStrata}
  The number of strata of dimension $s$ in $\LLNRi$ is
  \[ \frac{\size{\GL[r]{A/N}}}{\size{\GL[s]{A/N}} \size{\GL[r-s]{A/N}} \sizep{A/N}^{s(r-s)}}. \qedhere \]
\end{corollary}
\begin{proof}
  The number in question is the number of free submodules of $\parens{A/N}^r$ of rank $s$, or equivalently the cardinality of $\rquotient{\Inj_N^{s,r}}{\GL[s]{A/N}}$.
  Now by \cref{prop:GLactFreeTrans} the right action of $\GL[s]{A/N}$ on $\Inj_N^{s,r}$ is free; hence the cardinality of $\rquotient{\Inj_N^{s,r}}{\GL[s]{A/N}}$ is the ratio of the cardinalities of $\Inj_N^{s,r}$ and $\GL[s]{A/N}$, and so \cref{coro:LLNRi_numInvStructs} completes the proof.
\end{proof}

Interestingly, the above number is invariant under the involution $s \mapsto r-s$.

\begin{definition} \label{def:EulerPhi}
  For nonnegative integer $r$ and an ideal $N$ of $A$ which factorises into a product of prime ideals as $N = \prod_i \fp_i^{d_i}$, we define the function 
  \[ \phi^r(N) = \abs{N}^{r} \cdot \prod_i \parens{1-\abs{\fp_i}^{-r}} \]
  analogously to the definition of the classical Euler $\phi$ function.
\end{definition}

There should be no confusion between this use of $\phi^r$ as a function on ideals and $\phi^\Lambda$ as a polynomial, since $r$ is an integer while $\Lambda$ is a lattice.
Note that $\phi^r$ is a multiplicative function, \ie if $M$ and $N$ are coprime ideals (\ie $M+N = A$) then $\phi^r(MN) = \phi^r(M) \phi^r(N)$.

\begin{proposition} \label{prop:GLFormula}
  \[ \size{\GL[r]{A/N}} = \abs{N}^{r(r-1)/2} \cdot \phi^r(N) \phi^{r-1}(N) \dotsm \phi^1(N). \qedhere \]
\end{proposition}
This has been proven in \cite[Lemma 2.3]{breuer2010torsion}.

\begin{corollary} \label{coro:LLNRi_NumStrataPhi}
  The number of strata of $\LLNRi$ of dimension $s$ is equal to
  \[ \frac{\phi^r(N) \dotsm \phi^1(N)}{\phi^s(N) \dotsm \phi^1(N) \cdot \phi^{r-s}(N) \dotsm \phi^1(N)}. \qedhere \]
\end{corollary}
\begin{proof}
  Apply \cref{prop:GLFormula} to \cref{coro:LLNRi_NumStrata}.
\end{proof}

Interestingly, this number is multiplicative, \ie for coprime ideals $N,M$ we have that the number of strata of dimension $s$ of $\overleftarrow{\LL_{MN}^r}$ is equal to the product of those of $\overleftarrow{\LL_M^r}$ and $\overleftarrow{\LL_N^r}$. Also, this formula is very reminiscent of the formula for binomial coefficients in terms of factorials.

\begin{paragraph}
  Note that if we set $N = A$ all through this and the previous subsection, each lattice $\Lambda$ has exactly one ($r$-inverse) level $N$ structure, being the zero map.
  Hence we have the following isomorphisms between spaces: $\LL_A^r \simeq \LL^r$ and $\overleftarrow{\LL_A^r} \simeq \LL^{\leq r}$; and so considerations of inclusion of level structure do not exclude the spaces without level structure, as long as $N \neq A$ is not assumed.
\end{paragraph}

%% file: text/3_4_ll_GLrAN.tex
\subsection[The action of GLr(A/N)]{The action of $\GL[r]{A/N}$ on $\LLNRi$}

Recall the action of $\gamma \in \GL[r]{A/N}$ on $\LL_N^r$ defined by $\gamma (\Lambda,\alpha) = (\Lambda, \alpha \circ \gamma^{-1})$.
Given the identification of the rank-$r$ subset of $\LLNRi$ with $\LL_N^r$ in \cref{para:LLNRimainCompoIdentLLNR} via $(\Lambda,\iota) \leftrightarrow (\Lambda,\alpha_{\iota}) = (\Lambda,\iota^{-1})$, we see how to extend this action to $\LLNRi$:

\begin{definition} \label{def:GLrANActLLNRi}
  $\GL[r]{A/N}$ acts on $\LLNRi$ on the left via $\gamma(\Lambda,\iota) = (\Lambda, \gamma \circ \iota)$.
\end{definition}

\begin{proposition} \label{GLrANActIsometry}
  The action of $\GL[r]{A/N}$ on $\LLNRi$ is an isometry.
\end{proposition}
\begin{proof}
  Left as an exercise to the reader.
\end{proof}
  

We now see how this action interacts with the decomposition of $\LLNRi$ into strata of the form $\LL_N^s$:
\begin{proposition} \label{prop:GLactBdComps}
  For $\delta$ an injective selection of $\Free_N^r$, $\gamma \in \GL[r]{A/N}$ acts on the decomposition $\bigsqcup_{\substack{0 \leq s \leq r \\ U \in \Free_N^{s,r}}} \LL_N^s$ of $\LLNRi$ by
  \[ \gamma(\Lambda, \alpha)_U = (\Lambda, \alpha \circ \delta_U^{-1} \circ \gamma^{-1} \circ \delta_{\gamma U})_{\gamma U}. \qedhere \]
\end{proposition}
\begin{proof}
  Let $\gamma(\Lambda,\alpha)_U = (\Lambda',\alpha')_{U'} \mapsto (\Lambda', \delta_{U'} \circ \alpha'^{-1})$.
  Firstly, it is easy to see that $\Lambda' = \Lambda$.
  Secondly, we have from \cref{def:GLrANActLLNRi} that
  \[ (\Lambda',\alpha')_{U'} = \gamma(\Lambda,\alpha)_U \mapsto \gamma (\Lambda, \delta_U \circ \alpha^{-1}) = (\Lambda, \gamma \circ \delta_U \circ \alpha^{-1}). \]
  Hence $U' = \Im(\gamma \circ \delta_U \circ \alpha^{-1}) = \gamma \Im(\delta_U) = \gamma U$.
  Finally,
  \begin{align*}
    \delta_{U'} \circ \alpha'^{-1} = \gamma \circ \delta_U \circ \alpha^{-1} &\iff \alpha'^{-1} = \delta_{\gamma U}^{-1} \circ \gamma \circ \delta_U \circ \alpha^{-1} \\
    &\iff \alpha' = \alpha \circ \delta_U^{-1} \circ \gamma^{-1} \circ \delta_{\gamma U}. \qedhere
  \end{align*}
\end{proof}

Note that in the above action, although $\delta_U$ is not always a bijection and so $\delta_U^{-1}$ does not exist as a function from $\parens{N^{-1}/A}^r$ to $U$, since
\[ \Im(\gamma^{-1} \circ \delta_{\gamma U}) = \gamma^{-1} \Im{\delta_{\gamma U}} = \gamma^{-1} \gamma U = U \]
the composite $\delta_U^{-1} \circ \gamma^{-1} \circ \delta_{\gamma U} : \gamma U \ionto U$ is in fact always well defined.

Also note that if for the main stratum $\LL_N^r$ of $\LLNRi$ (which corresponds to $U = \parens{N^{-1}/A}^r$) we have that $\delta_{U}$ is the identity map, then the action of $\gamma$ on the main stratum simplifies to $\gamma (\Lambda,\alpha)_U = (\Lambda, \alpha \circ \gamma^{-1})_U$, \ie the way it is defined on $\LL_N^r$.


%% file: text/3_5_ll_invadele.tex
\subsection[The action of the invertible adeles]{The action of $\invertadele$}

Recall the action of $x \in \invertadele$ on $\LL_N^r$ given by $x (\Lambda,\alpha) = (J(x)^{-1}\Lambda, x^{-1} \circ \alpha)$ in \cref{prop:invAdele_L}.
The action of $\invertadele$ on $\LL_N^r$ induces a rigid analytic action of the set of fractional ideals on $\LL^r$ by $\Lambda \xmapsto{J} J^{-1} \Lambda$.
In fact, we can extend this action to $\LL^{\leq r}$ by the following result:
\begin{proposition} \label{prop:fracIdealActLeqR}
  The action of $\cJ(A)$ on $\LL^{\leq r}$ given by $J(\Lambda) = J^{-1} \Lambda$ is a homeomorphism.
\end{proposition}
\begin{proof}[Sketch of proof]
  It is enough to show that each action by $J \in \cJ(A)$ is continuous, since then the action by $J^{-1}$ is continuous too.
  We then split the proof into cases of $J$ being an ideal of $A$, where we use that the coefficients of the polynomial $\phi_J^\Lambda(X)$ are continuous on $\LL^{\leq r}$, and of $J = (a)^{-1}$ for $a \in A$ which is straghtforward.
\end{proof}

Unfortunately we have not yet, as of the time of writing, been able to extend the action of $\invertadele$ to $\LLNRi$ in a sufficiently `nice' way. In fact, we suspect this extension to not be possible, although we have not yet proven this either.
We hope that the interested reader would carry this research forward.

%% file: text/4__modular_forms.tex
\section{Modular Forms} \label{sec:modular_forms}

\input{text/4_1_mf_def.tex}

\input{text/4_2_mf_eg.tex}

\input{text/4_3_mf_cusp_exp.tex}

\input{text/4_4_mf_BBP_relation.tex}

\input{text/4_5_mf_gek_relation.tex}

%% file: text/4_1_mf_def.tex
\subsection{Definitions}

\subsubsection{Modular forms on $\LL_N^r$}

To understand the structure and shape of the space of lattices with level structure or isomorphism classes of Drinfeld modules with level structure, one strategy is to investigate the collections of functions on those spaces.

\begin{definition} \label{def:weakMForm}
  For $k \in \ZZ$, a \emph{weak modular form} $f$ of \emph{weight} $k$ and \emph{rank} $r$ for the congruence subgroup $K(N)$ is a function $\LL_N^r \to \CCi$ which is
  \begin{itemize}
    \item holomorphic\footnote{Here holomorphic on $\LL_N^r$ means holomorphic on the isomorphic rigid analytic space $\lquotient{\GL[r]{F}}{\parens{\Psi^r \times \rquotient{\GL[r]{\finadele}}{K(N)}}}$, \ie holomorphic on the space $\Psi^r \subset \CCi^r$ and invariant under the two group actions of $\GL[r]{F}$ and $K(N)$.}, and
    \item homogeneous of degree $-k$, \ie
    \[ f(t \cdot \Lambda, t \cdot \alpha) = t^{-k} f(\Lambda,\alpha) \lsptext{for all} (\Lambda,\alpha) \in \LL_N^r \sptext{and} t \in \CCi^\times. \]
  \end{itemize}
  We denote the $\CCi$-vector space of weak modular forms for $K(N)$ of weight $k$ and rank $r$ by $\WeakMF_N^{k,r}$.
\end{definition}

\begin{paragraph} \label{para:weakMFormGradedRing}
  It is easy to check that $\WeakMF_N^{k,r}$ is a $\CCi$-vector space.
  Moreover, the product of two modular forms of weight $k_1$ and $k_2$ is a modular form of weight $k_1+k_2$, so that $\WeakMF_N^r \defeq \oplus_{k = 0}^\infty \WeakMF_N^{k,r}$ is a graded $\CCi$-algebra, graded by the weight $k$.
\end{paragraph}

\begin{definition} \label{def:strongMForm}
  A \emph{(strong) modular form} of \emph{weight} $k$ and \emph{rank} $r$ for $K(N)$ is a function $\LLNRi \to \CCi$ which is:
  \begin{itemize}
    \item continuous on $\LLNRi$,
    \item homogeneous of degree $-k$, and
    \item holomorphic on the main stratum $\LL_N^r$ of $\LLNRi$.
  \end{itemize}
  In general we will omit the adjective \emph{strong}, unless we are comparing strong with weak modular forms, and will omit the reference to $K(N)$.

  We denote the $\CCi$-vector space of strong modular forms for $K(N)$ of weight $k$ and rank $r$ by $\StrongMF_N^{k,r}$.
\end{definition}

\begin{paragraph} \label{para:strongMFormGradedRing}
  As in \cref{para:weakMFormGradedRing}, the strong modular forms of rank $r$ form a graded $\CCi$-algebra
  \[ \StrongMF_N^r \defeq \bigoplus_{k = 0}^\infty \StrongMF_N^{k,r}, \]
  graded by the weight $k$.
\end{paragraph}

\begin{paragraph}
  It is apparent that the restriction of a strong modular form to the main stratum $\LL_N^r$ of $\LLNRi$ is a weak modular form, and given the denseness of $\LL_N^r$ in $\LLNRi$, the values of a strong modular form (which is continuous) on the boundary strata can be recovered from the values on the main stratum.
  However, not every weak modular form can necessarily be extended to a strong modular form, if for instance it does not have a limit as one tends to the boundary of $\LLNRi$.
  Thus there is an injection $\StrongMF_N^{k,r}\into \WeakMF_N^{k,r}$ of $\CCi$-vector spaces for each weight $k$ and an injection $\StrongMF_N^r \into \WeakMF_N^r$ of graded $\CCi$-algebras.
\end{paragraph}

\begin{proposition} \label{prop:mFormNonpositiveWeight}
  The only modular forms of weight $0$ are the constant maps, and the only modular forms of negative weight are the zero maps.
\end{proposition}
\begin{proof}
  Let $f$ be a modular form of weight $0$, $(\Lambda,\iota) \in \LLNRi$, and $t \in \CCi^\times$ be large.
  Then since $f$ is homogeneous of weight $0$, $f(\Lambda,\iota) = f(t\Lambda,\iota t)$.
  But since $f$ is continuous, and $(t\Lambda,\iota t) \to 0$ as $\abs{t} \to \infty$, we get that $f(\Lambda,\iota) = f(0)$; thus $f$ is constant.
  Conversely, any constant map is a modular form of weight $0$.

  Now let $f$ be a modular form of weight $k < 0$, let $(\Lambda,\iota) \in \LLNRi$, and let $t \in \CCi^\times$ be large.
  Then $f(\Lambda,\iota) = t^k f(t\Lambda,t\iota)$; but $(t\Lambda,t\iota) \to 0$ and $t^k \to 0$ as $\abs{t} \to \infty$, so since $f$ is continuous we get that $t^k f(t\Lambda,t\iota) \to 0 \times f(0) = 0$; thus $f$ is the zero map, which is a modular form of weight $k$.
\end{proof}

\begin{definition} \label{def:cuspForm}
  A modular form $f \in \StrongMF_N^{k,r}$ is called a \emph{cusp form} if $f(\Lambda,\iota) = 0$ for $(\Lambda,\iota)$ in the boundary strata, \ie when $\Lambda$ is of rank strictly less than $r$.
  
  The $\CCi$-vector space of cusp forms of weight $k$ and rank $r$ for $K(N)$ is denoted $\CuspMF_N^{k,r}$, and the $\StrongMF_N^r$-algebra of all cusp forms for $K(N)$, which is also a graded $\CCi$-algebra graded by weight, is denoted $\CuspMF_N^r$.
  
  $\CuspMF_N^r$ is also an ideal of $\StrongMF_N^r$.
\end{definition}

The left action of $\GL[r]{A/N}$ on $\LLNRi$ carries over to a right action on the space of modular forms:
\begin{definition} \label{def:mFormSlash}
  For a modular form $f$ and $\gamma \in \GL[r]{A/N}$, we define the function $f|\gamma$ on $\LLNRi$ by $(f|\gamma)(\Lambda,\iota) = f(\gamma(\Lambda,\iota)) = f(\Lambda, \gamma \circ \iota)$.
\end{definition}

\begin{proposition} \label{prop:mFormSlashPreserve}
  The map $f \mapsto f|\gamma$ is a right action of $\gamma \in \GL[r]{A/N}$ on $\StrongMF_N^r$, which preserves the weight $k$ and maps cusp forms to cusp forms.
\end{proposition}
\begin{proof}
  Let $f$ be a modular form of weight $k$; then since the action of $\gamma \in \GL[r]{A/N}$ is an isometry, $f|\gamma$ is also continuous; since the action is a rigid analytic automorphism of $\LL_N^r$, $f|\gamma$ is also holomorphic on $\LL_N^r$; and it is easy to see that $f|\gamma$ is also homogeneous of degree $-k$.
  Thus $\gamma$ maps $\StrongMF_N^{k,r}$ to $\StrongMF_N^{k,r}$, and it is easy to see that $f \mapsto f|\gamma$ satisfies the conditions of a right action.
  
  Finally, since the action of $\gamma$ leaves $\Lambda$ unchanged, if $f(\Lambda,\iota) = 0$ for $\Lambda$ of rank less than $r$, then the same is true for $f|\gamma$.
\end{proof}

\subsubsection{Modular forms for $\LL^{\leq r}$}

\begin{definition} \label{def:mForm_noLevel}
  For $k \in \ZZ$, a (strong) modular form of weight $k$ and rank $r$ for $\LL^{\leq r}$ is a function $\LL^{\leq r} \to \CCi$ which is:
  \begin{itemize}
    \item continuous on $\LL^{\leq r}$,
    \item homogeneous of degree $-k$, and
    \item holomorphic on the main stratum $\LL^r$ of $\LL^{\leq r}$.
  \end{itemize}

  For $k \in \ZZ$, a weak modular form of weight $k$ and rank $r$ for $\LL^r$ is a function $\LL^r \to \CCi$ which is:
  \begin{itemize}
    \item homogeneous of degree $-k$, and
    \item holomorphic.
  \end{itemize}

  In addition, a strong modular form for $\LL^{\leq r}$ is a function $\LL^{\leq r} \to \CCi$ which is continuous on $\LL^{\leq r}$ and for which the restriction to $\LL^r$ is a weak modular form.

  We denote the spaces of weak and strong modular forms of weight $k$ and rank $r$ for $\LL^{\leq r}$ by $\WeakMF^{k,r}$ and $\StrongMF^{k,r}$, respectively.
  As before, a modular form on $\LL^{\leq r}$ will be assumed to mean a strong form unless otherwise indicated.
\end{definition}

\begin{paragraph}
  As in \cref{para:weakMFormGradedRing}, the weak and strong modular forms on $\LL^{\leq r}$ of rank $r$ form graded $\CCi$-algebras $\WeakMF^r \defeq \bigoplus_{k = 0}^\infty \WeakMF^{k,r}$ and $\StrongMF^r \defeq \bigoplus_{k = 0}^\infty \StrongMF^{k,r}$, graded by $k$.
\end{paragraph}

\begin{paragraph}
  Note that if a modular form $f$ for $\LLNRi$ is independent of the level structure $\iota$ (\ie if $f(\Lambda,\iota_1) = f(\Lambda,\iota_2)$ for any two inverse level $N$ structures $\iota_1, \iota_2$ for a lattice $\Lambda$, or equivalently it is invariant under the action of $\GL[r]{A/N}$) then it induces a unique modular form on $\LL^{\leq r} \simeq \lquotient{\GL[r]{A/N}}{\LLNRi}$.
\end{paragraph}

\begin{definition} \label{def:cuspForm_noLevel}
  A modular form $f \in \StrongMF^{k,r}$ is called a \emph{cusp form} if $f(\Lambda) = 0$ for $\Lambda$ in the boundary strata $\LL^{\leq r-1}$, \ie when $\Lambda$ has rank less than $r$.
  
  The $\CCi$-vector space of cusp forms of weight $k$ and rank $r$ is denoted $\CuspMF^{k,r}$, and the $\StrongMF^r$-algebra of all cusp forms, which is also a graded $\CCi$-algebra graded by weight, is denoted $\CuspMF^r$.
\end{definition}

The action of $\invertadele$ on $\LL^{\leq r}$, or equivalently of $\cJ(A)$ on $\LL^{\leq r}$, also carries over to an action on the space of modular forms on $\LL^{\leq r}$:
\begin{definition} \label{def:mFormSlashJ}
  For a modular form $f$ on $\LL^{\leq r}$ and a fractional ideal $J$, we define the function $f|J$ on $\LL^{\leq r}$ by $(f|J)(\Lambda) \defeq f(J(\Lambda)) = f(J^{-1}\Lambda)$.
\end{definition}

\begin{proposition} \label{prop:mFormSlashJPreserve}
  The map $f \mapsto f|J$ is an action of $\cJ(A)$ on $\StrongMF^{k,r}$ for each $k \in \ZZ$, which maps cusp forms to cusp forms.
\end{proposition}
\begin{proof}
  Let $f$ be a modular form of weight $k$; then since the action of $J$ is a homeomorphism, $f|J$ is also continuous; since the action is a rigid analytic automorphism of $\LL^r$, $f|J$ is also holomorphic on $\LL^r$; and since the action of $J$ commutes with scaling of a lattice, $f|J$ is also homogeneous of degree $-k$.
  Thus $J$ maps $\StrongMF^{k,r}$ to itself.

  Finally, if $\Lambda$ has rank $r' < r$, then so does $J^{-1}\Lambda$, and so if $f$ is a cusp form, then $f|J$ is too.
\end{proof}

\begin{paragraph}
  As a special case of this action by fractional ideals, we have the instances where $J = (t)$ is a principal ideal for $t \in F^\times$.
  Here, $J$ is simply scaling the lattices by $t^{-1}$, and thus is scaling the $\StrongMF_N^{k,r}$ by $t^k$.
  So the more interesting case is of non-principal ideals, with the ideal class group $\Cl(F) = \rquotient{\cJ(A)}{\mathrm{Prin}(A)}$ being of interest.
\end{paragraph}

%% file: text/4_2_mf_eg.tex
\subsection{Examples of modular forms}

In this subsection we list some classes of examples of modular forms, where we omit the more technical proofs that these are in fact holomorphic and continuous; for these proofs see \cite{baker2020lattice}.

There is large overlap between these modular forms and the examples given in \cite{BBPIII}, the correspondence between which may be made clearer by the connective results in the following subsection.

\subsubsection{Eisenstein Series and $e_\Lambda$ coefficient forms}

\begin{definition} \label{def:EisensteinSeries}
  For $k \in \NN$ we define the \emph{Eisenstein series}
  \[ E^k(\Lambda) = \sump_{\lambda \in \Lambda} \lambda^{-k}. \qedhere \]
\end{definition}

These converge since the lattice elements go to infinity.
Moreover, $E^k$ is homogeneous of degree $-k$; \ie $E^k(c \Lambda) = c^{-k} E^k(\Lambda)$ for $c \in \CCi^\times$.
Finally, $E^k$ is a holomorphic function on $\LL^r$.

There is a connection between these Eisenstein series and the Taylor expansion coefficients of $e_\Lambda$, or more specifically of $1/e_\Lambda$:
\begin{proposition} \label{prop:eLambdaEiskCoeff}
  For $\abs{z} < \minp_{\lambda \in \Lambda} \abs{\lambda}$,
  \[ \frac{1}{e_\Lambda(z)} -\frac{1}{z} = -\sum_{k = 1}^\infty E^k(\Lambda) z^{k-1}. \qedhere \]
\end{proposition}
\begin{proof}
  By \cref{prop:e_Lambda_props},
  \begin{align*}
    \frac{1}{e_\Lambda(z)} -\frac{1}{z} &= \sump_{\lambda \in \Lambda} \frac{1}{z-\lambda} = -\sump_{\lambda \in \Lambda} \sum_{k = 1}^\infty \frac{z^{k-1}}{\lambda^k} \\
    &= -\sum_{k = 1}^\infty z^{k-1} \sump_{\lambda \in \Lambda} \lambda^{-k} = -\sum_{k = 1}^\infty E^k(\Lambda) z^{k-1}. \qedhere
  \end{align*}
\end{proof}

\begin{definition} \label{def:eLambdaCoeffs}
  For $i \in \NNO$ we define the \emph{coefficient forms}
  \[ e_{\Lambda,i} = \cin{z}{q^i} e_\Lambda(z)\,.\footnote{Here and elsewhere, the notation $\cin{z}{k} f(z)$ denotes the coefficient of $z^k$ in $f(z)$.} \]
  As functions of $\Lambda$, these are continuous on $\LL^{\leq r}$, holomorphic on $\LL^r$, and homogeneous of degree $1-q^i$.
\end{definition}

\begin{proposition} \label{prop:EiskELambdaGF}
  For $i \in \NNO$ and $k \in \NN$ there are polynomials $G_i$, $G^k$ with coefficients in $\FF_p$ such that $e_{\Lambda,i} = G_i(E^1(\Lambda), E^2(\Lambda), \dotsc)$ and $E^k(\Lambda) = G^k(e_{\Lambda,1}, e_{\Lambda,2}, \dotsc)$.
\end{proposition}
\begin{proof} For $i \in \NNO$,
  \begin{align*}
    e_{\Lambda, i} &= \cin{z}{q^i} e_\Lambda(z) \\
    &= \cin{z}{q^i} z \brackets{1 -\parens{1-\frac{z}{e_\Lambda(z)}}}^{-1} = \cin{z}{q^i-1} \sum_{n = 0}^\infty \parens{1-\frac{z}{e_\Lambda(z)}}^n \\
    &= \cin{z}{q^i-1} \sum_{n = 0}^\infty \parens{\sum_{k = 1}^\infty E^k(\Lambda)z^k}^n \lsptext{by \cref{prop:eLambdaEiskCoeff}} \\
    &= \sum_{n = 0}^{q^i-1} \cin{z}{q^i-1} \parens{\sum_{k = 1}^\infty E^k(\Lambda) z^k}^n \\
    &= \sum_{n = 0}^{q^i-1} \sum_{\substack{k_1 +\dotsb +k_n = q^i-1 \\ k_1, k_2, \dotsc, k_n \geq 1}} E^{k_1}(\Lambda) E^{k_2}(\Lambda) \dotsm E^{k_n}(\Lambda)
  \end{align*}
  which is a polynomial in the $E^k(\Lambda)$ as desired.
  The second series of polynomials is found similarly, using a geometric series expansion on the relation
  \[ E^k(\Lambda) = \cin{z}{k} \parens{1-\frac{z}{e_\Lambda(z)}} = -\cin{z}{k} \frac{1}{1 +(e_\Lambda(z)/z-1)}. \qedhere \]
\end{proof}



Thus the coefficients $e_{\Lambda,i}$ and the Eisenstein series $E^k(\Lambda)$ form our first two classes of modular forms for $\LLNRi$, and are in fact modular forms on $\LL^{\leq r}$.
As such, we can extend these classes to more modular forms on $\LL^{\leq r}$:
\begin{proposition} \label{prop:EiskJLeJLCoeffsModular}
  For each $J \in \cJ(A)$ and $i,k \in \NNO$ we have that $e_{J^{-1}\Lambda,i}$ and $E^k(J^{-1}\Lambda)$ are modular forms on $\LL^{\leq r}$ of weight $q^i-1$ and $k$ respectively.
\end{proposition}
\begin{proof}
  $e_{J^{-1}\Lambda,i} = e_{\Lambda,i}|J$ and $E^k(J^{-1}\Lambda) = E^k(\Lambda)|J$, so \cref{prop:mFormSlashJPreserve} applies.
\end{proof}

\subsubsection{Partial Eisenstein Series}

\begin{definition} \label{def:Eisk}
  For each $k \in \NN$ and $l \in \parens{N^{-1}/A}^r -\set{0}$, we define the \emph{(partial) Eisenstein series} $E_l^k$ on $\LLNRi$ by
  \[ E_l^k(\Lambda,\iota) = \begin{dcases} \sum_{v \in \iota^{-1}(l)} v^{-k} & \sptext{if} l \in \Im{\iota} \\ 0 & \sptext{otherwise} \end{dcases}. \]
  We will omit the adjective \emph{partial}, unless contrasting with the \emph{complete} Eisenstein series of \cref{def:EisensteinSeries}.
\end{definition}

Restricting to the main stratum $\LL_N^r$, we have the corresponding weak modular forms:
\begin{definition} \label{def:weakEisk}
  For $k \in \NN$ and $l \in \parens{N^{-1}/A}^r -\set{0}$, we define the \emph{weak (partial) Eisenstein series} $E_l^k$ on $\LL_N^r$ by
  \[ E_l^k(\Lambda,\alpha) = \sum_{v \in \alpha(l)} v^{-k}. \qedhere \]
\end{definition}
It will be possible to see from each context whether the weak or the strong modular form $E_l^k$ is intended.

Each $E_l^k$ is homogeneous of degree $-k$, holomorphic on $\LL_N^r$ and continuous on $\LLNRi$ in the case of the strong/non-weak partial Eisenstein series.

In the case of weight $k = 1$, we encounter something familiar:

\begin{proposition} \label{prop:Eis1mu}
  $E_l^1(\Lambda,\iota) = \mu_{\Lambda,\iota}(l)$.
\end{proposition}
\begin{proof}
  If $l \in \Im{\iota}$, let $\lambda'$ be an element of $\iota^{-1}(l)$.
  Then by \cref{prop:e_Lambda_props},
  \[ E_l^1(\Lambda,\iota) = \sum_{v \in \iota^{-1}(l)} \frac{1}{v} = \sum_{\lambda \in \Lambda} \frac{1}{\lambda'+\lambda} = \frac{1}{e_\Lambda(\lambda')} = \frac{1}{e_\Lambda(\iota^{-1}(l))} = \mu_{\Lambda,\iota}(l)\,. \]
  For $l \notin \Im{\iota}$, both sides are $0$, and so the proof is complete.
\end{proof}


There is in fact some dependence between the $E_l^k$ and the modular forms $e_{\Lambda,i}$:

\begin{proposition} \label{prop:EislkELambdaGF}
  For each $k \in \NN$ there is a polynomial $g_k(x,y_1,y_2,\dotsc)$ with coefficients in $\FF_p$ such that $E_l^k = g_k(E_l^1, e_{\Lambda,1},e_{\Lambda,2},\dotsc)$ for all $l \in \parens{N^{-1}/A}^r-\set{0}$.
\end{proposition}
\begin{proof}
  Consider the generating function of the $E_l^k$.
  For $l \in \Im{\iota}$:
  \begin{align*}
    &\mathrel{\phantom{=}} \sum_{k > 0} E_l^k(\Lambda,\iota) z^k \\
    &= \sum_{v \in \iota^{-1}(l)} \sum_{k > 0} v^{-k}z^k = \sum_{v \in \iota^{-1}(l)} \frac{z/v}{1-z/v} = z \sum_{v \in \iota^{-1}(l)} \frac{1}{v-z} \\
    &= \frac{z}{e_\Lambda(\iota^{-1}(l)-z)} = \frac{z}{e_\Lambda(\iota^{-1}(l))-e_\Lambda(z)} = \frac{z}{e_\Lambda(\iota^{-1}(l))} \brackets{1-\frac{e_\Lambda(z)}{e_\Lambda(\iota^{-1}(l))}}^{-1} \\
    &= \sum_{n > 0} \frac{z e_\Lambda(z)^{n-1}}{e_\Lambda(\iota^{-1}(l))^n} = \sum_{n > 0} z e_\Lambda(z)^{n-1} (E_l^1)^n.
  \end{align*}
  For $l \notin \Im{\iota}$, both sides are zero, and so the equality still holds.
  So in general,
  \[ E_l^k = \cin{z}{k} \sum_{n > 0} E_l^n z^n = \cin{z}{k} \sum_{n > 0} z e_\Lambda(z)^{n-1} (E_l^1)^n = \sum_{n > 0} (E_l^1)^n \cin{z}{k-1} e_\Lambda(z)^{n-1}. \]
  Now each $\cin{z}{k-1} e_\Lambda(z)^{n-1}$ for $n > 0$ is a polynomial in the $e_{\Lambda,i}$ with coefficients in $\FF_p$.
  Moreover, since $e_\Lambda(z)$ has zero constant term, $\cin{z}{k-1} e_\Lambda(z)^{n-1} = 0$ for $n > k$, and so the last sum above is finite.
  This proves the claim.
\end{proof}

The $E_l^k$ form our third class of examples of modular forms, and our first which are not modular forms for $\LL^{\leq r}$.
Looking at their behaviour under the action of $\GL[r]{A/N}$, we see a

\begin{proposition} \label{prop:EislkSlashAN}
  For $l \in \parens{N^{-1}/A}^r-\set{0}$ and $\gamma \in \GL[r]{A/N}$,
  \[ E_l^k |\gamma = E_{\gamma^{-1}l}^k. \qedhere \]
\end{proposition}
\begin{proof} ~ \vspace{-18pt}
  \begin{align*}
    \parens{E_l^k|\gamma}(\Lambda,\iota) &= E_l^k(\Lambda,\gamma\iota) \\
    &= \begin{dcases}
      \sum_{v \in (\gamma\iota)^{-1}(l)} v^{-k} & \sptext{if} l \in \Im(\gamma\iota) \\
      0 & \sptext{otherwise}
    \end{dcases} \\
    &= \begin{dcases}
      \sum_{v \in \iota^{-1}(\gamma^{-1}(l))} v^{-k} & \sptext{if} \gamma^{-1}(l) \in \Im{\iota} \\
      0 & \sptext{otherwise}
    \end{dcases} \\
    &= E_{\gamma^{-1}l}^k(\Lambda,\iota) \qedhere
  \end{align*}
\end{proof}

Since these Eisenstein series are our first series of examples which actually depend on the $r$-inverse level structure $\iota$, it is interesting to see how a partial Eisenstein series looks when restricted to a boundary stratum $\LL_N^s$ of $\LLNRi$:
\begin{proposition} \label{prop:Eislk_LLNRiUnion}
  Let $l \in \parens{N^{-1}/A}^r-\set{0}$, $k \in \NNO$, and $\delta$ be an injective selection of $\Free_N^r$.
  Then for each $U \in \Free_N^r$ of rank $s$, the restriction $E_l^k|_U$ of the partial Eisenstein series $E_l^k$ on $\LLNRi$ to the boundary stratum $\LL_N^s$ corresponding to $U$ as in \cref{prop:LLNRi_union} is a partial Eisenstein series of rank $s$, as follows:
  \[ E_l^k|_U(\Lambda,\alpha) = \begin{cases} E_{\delta_u^{-1}(l)}^k (\Lambda,\alpha) & l \in \Im{\delta_U} \\ 0 & \text{otherwise} \end{cases}. \qedhere \]
\end{proposition}
\begin{proof}
  For $(\Lambda,\alpha) \in \LL_N^s$, by \cref{prop:LLNRi_union} we have
  \begin{align*}
    E_l^k|_U(\Lambda,\alpha) &= E_l^k(\Lambda, \delta_U \circ \alpha^{-1}) \\
    &= \begin{dcases} \sum_{v \in (\delta_U\circ\alpha)^{-1}(l)} v^{-k} & \sptext{if} l \in \Im(\delta_U\circ\alpha) \\ 0 & \sptext{otherwise} \end{dcases} \\
    &= \begin{dcases} \sum_{v \in \alpha^{-1}(\delta_U^{-1}(l))} v^{-k} & \sptext{if} l \in \Im{\delta_U} \\ 0 & \sptext{otherwise} \end{dcases} \\
    &= \begin{dcases} E_{\delta_u^{-1}(l)}^k & \sptext{if} l \in \Im{\delta_U} \\ 0 & \sptext{otherwise} \end{dcases}. \qedhere
  \end{align*}
\end{proof}

\subsubsection{Drinfeld module coefficient forms}

\begin{proposition} \label{prop:DmoduleCoeffsModular}
  For each $a \in A$ and for each ideal $I$ of $A$, the coefficients $\phi^\Lambda_{a,i} = \cin{\tau}{i} \phi^\Lambda_a = \cin{X}{q^i} \phi^\Lambda_a(X)$ and $\phi^\Lambda_{I,i} = \cin{\tau}{i} \phi^\Lambda_I = \cin{X}{q^i} \phi^\Lambda_I(X)$ for $i \in \NNO$ considered as functions of $\Lambda$ are modular forms of weight $q^i-1$ on $\LL^{\leq r}$.
\end{proposition}
\begin{proof}
  By \cref{prop:Dmod_from_lattice}, for each $i \in \NNO$ we have that
  \begin{align*}
    \phi^\Lambda_{a,i} &= \cin{X}{q^i} \phi^\Lambda_a(X) = a \cin{X}{q^i} X \prodp_{\lambda \in a^{-1}\Lambda/\Lambda} \parens{1-\frac{X}{e_\Lambda(\lambda)}} \\
    &= a \cin{X}{q^i-1} \prodp_{l \in ((a)^{-1}/A)^r} \parens{1-E_l^1(\Lambda)X}
  \end{align*}
  is a homogeneous symmetric polynomial of degree $q^i-1$ in the $E_l^1(\Lambda)$ for $l \in ((a)^{-1}/A)^r -\set{0}$, and hence is a modular form of degree $q^i-1$ on $\LLi{(a)}{r}$.
  Moreover, since the aforementioned polynomial is symmetric, $\phi^\Lambda_{a,i}$ is invariant under the action of $\GL[r]{A/(a)}$ on $\LLi{(a)}{r}$ and hence is a modular form on $\LL^{\leq r}$.
  
  The proof for $\phi^\Lambda_{I,i}$ is similar, based instead on \cref{def:IdealDModule}.
\end{proof}

\begin{corollary} \label{coro:DmoduleCoeffsJModular}
  For each $J \in \cJ(A)$, $i \in \NNO$, $a \in A$, and ideal $I$ of $A$, the coefficients $\phi^{J^{-1}\Lambda}_{a,i} = \cin{\tau}{i} \phi^{J^{-1}\Lambda}_a$ and $\phi^{J^{-1}\Lambda}_{I,i} = \cin{\tau}{i} \phi^{J^{-1}\Lambda}_I$ are modular forms on $\LL^{\leq r}$ of weight $q^i-1$.
\end{corollary}
\begin{proof}
  $\phi^{J^{-1}\Lambda}_{a,i} = \phi^{\Lambda}_{a,i} |J$ and $\phi^{J^{-1}\Lambda}_{I,i} = \phi^{\Lambda}_{I,i} |J$, and so \cref{prop:mFormSlashJPreserve} applies.
\end{proof}

We can now prove a proposition which was used in the previous section:
\ddLLradius*
\begin{proof} \label{proof:ddLLradius}
  Let $a \in A$ such that $\abs{a} \geq R$; since then
  \[ \sup_{\abs{z} \leq R} \abs{e_{\Lambda'}(z)-e_{\Lambda}(z)} \leq \sup_{\abs{z} \leq \abs{a}} \abs{e_{\Lambda'}(z)-e_{\Lambda}(z)}, \]
  it is enough to prove the above statement for the case $R = \abs{a}$ for some $a \in A$.

  Consider the polynomial $\phi^\Lambda_a$; since each of its finitely many nonzero coefficients is continuous on $\LL^{\leq r}$, this polynomial is itself continuous on $\LL^{\leq r}$. Thus
  \begin{align*}
    \sup_{\abs{z} \leq \abs{a}} \abs{e_{\Lambda'}(z)-e_\Lambda(z)} 
    &= \sup_{\abs{az} \leq \abs{a}} \abs{e_{\Lambda'}(az)-e_\Lambda(az)} = \sup_{\abs{z} \leq 1} \abs{\phi^{\Lambda'}_a(e_{\Lambda'}(z))-\phi^\Lambda_a(e_\Lambda(z))} \\
    &\leq \sup_{\abs{z} \leq 1} \abs{\phi^{\Lambda'}_a(e_{\Lambda'}(z)-e_\Lambda(z))} +\sup_{\abs{z} \leq 1} \abs{\phi^{\Lambda'}_a(e_\Lambda(z))-\phi^\Lambda_a(e_\Lambda(z))}.
  \end{align*}
  Now since $\Lambda' \to \Lambda$, we have that $\phi^{\Lambda'}_a \to \phi^\Lambda_a$ and so the second term goes to zero since $e_\Lambda(z)$ is bounded on $\abs{z} \leq 1$.
  For the first term, note that since $\phi^{\Lambda'}_a \to \phi^\Lambda_a$, the coefficients of the polynomial $\phi^{\Lambda'}_a$ are bounded and so since $e_{\Lambda'}(z)-e_\Lambda(z) \to 0$ uniformly on $\abs{z} \leq 1$, the first term also goes to zero.
  Thus as desired,
  \[ \sup_{\abs{z} \leq \abs{a}} \abs{e_{\Lambda'}(z)-e_\Lambda(z)} \longto 0. \qedhere \]
\end{proof}

\begin{paragraph}
  There is some dependence between these Drinfeld module coefficient forms for different values of $a \in A$ and ideals $I$.
  Indeed, from $\phi^\Lambda_{ab} = \phi^\Lambda_a \circ \phi^\Lambda_b$ we see that each coefficient $\phi^\Lambda_{ab,i}$ can be written as a polynomial in the coefficients $\phi^\Lambda_{a,j}$ and $\phi^\Lambda_{b,j}$ and similarly $\phi^\Lambda_{IJ,i}$ can be written as a polynomial in the $\phi^\Lambda_{I,j}$ and the $\phi^\Lambda_{J,j}$ for ideals $I$ and $J$.
  Of course, since each $\phi^\Lambda_a$ is a polynomial of degree at most $r \cdot \deg{a}$ in $\tau$, and exactly that much if $\Lambda$ has rank equal to $r$, $\phi^\Lambda_{a,i}$ is zero for $i > r\deg{a}$.
\end{paragraph}

\begin{paragraph}
  Let us also consider the modular forms $\Delta_a(\Lambda) = \phi^\Lambda_{a,r\deg{a}}$ on $\LL^{\leq r}$ for $a \in A$ and $\Delta_I(\Lambda) = \phi^\Lambda_{I,r\deg{I}}$, which are nonzero on the main stratum $\LL^r$ since $\phi^\Lambda_a$ and $\phi^\Lambda_I$ have degree exactly $r\deg{a}$ and $r\deg{I}$ respectively. For the same reason, these modular forms are cusp forms, being zero on the boundary $\LL^{\leq r-1} = \LL^{\leq r} -\LL^r$.
\end{paragraph}

\begin{paragraph}
  From the relation $\phi^\Lambda_{ab} = \phi^\Lambda_a \circ \phi^\Lambda_b$ we have that
  \begin{align*}
    \Delta_{ab}(\Lambda) &= \cin{\tau}{r\deg{ab}} \phi^\Lambda_{ab} = \cin{\tau}{r\deg{ab}} \phi^\Lambda_a \phi^\Lambda_b \\
    &= \cin{\tau}{r\deg{ab}} \parens{\Delta_a(\Lambda) \tau^{r\deg{a}}+\littleo{\tau^{r\deg{a}}}} \parens{\Delta_b(\Lambda) \tau^{r\deg{b}} +\littleo{\tau^{r\deg{b}}}} \\
    &= \cin{\tau}{r\deg{a}+r\deg{b}} \Delta_a(\Lambda) \tau^{r\deg{a}} \Delta_b(\Lambda) \tau^{r\deg{b}} \\
    &= \Delta_a(\Lambda) \Delta_b(\Lambda)^{q^{r\deg{a}}} = \Delta_a(\Lambda) \Delta_b(\Lambda)^{\abs{a}^r}.
  \end{align*}
  Since $ab = ba$, this implies that $\Delta_a(\Lambda) \Delta_b(\Lambda)^{\abs{a}^r} = \Delta_b(\Lambda) \Delta_a(\Lambda)^{\abs{b}^r}$, and so we have that $\Delta_a(\Lambda)^{\abs{b}^r-1} = \Delta_b(\Lambda)^{\abs{a}^r-1}$.
  Thus if $\sqrt[\abs{a}^r-1]{\Delta_a(\Lambda)}$ exists in some sense, it would be largely independent of $a$; the same applies for the ideal-based Drinfeld module coefficients.
\end{paragraph}

%% file: text/4_3_mf_cusp_exp.tex
\subsection{Expansions of modular forms at a cusp}

\subsubsection{Cusp expansions on $\LL^{r}$}

Similarly to the case of modular forms of rank $2$, higher rank modular forms have expansions at a cusp, i.e. as a lattice tends to a lattice of rank one less.
Before we begin, we need a technical lemma:

\begin{lemma} \label{lem:funcCoeffsSmall}
  Let $f(z) = \sum_{k = -\infty}^\infty f_k z^k$ be a series which converges on $0 < \abs{z} \leq 1$.
  Then if $\abs{f(z)} < \epsilon$ for all $z$ with $\abs{z} = 1$, each coefficient $f_k$ has $\abs{f_k} < \epsilon$.
\end{lemma}
\begin{proof}
  Let $\ell \in \ZZ$.
  Since $f(z)$ converges, we have that $f_k \to 0$ as $k \to \pm\infty$; so let $\abs{f_k} < \epsilon$ for $\abs{k} \geq K_\epsilon$.
  Since $\CCi \supset \overline{\FF_q}$, we can let $\zeta \in \FF_{q^n}-\FF_{q^{n-1}} \subset \CCi$ be a primitive $N$-th root of unity, for some $n$ with $N = q^n-1 > 2K_\epsilon$.

  Now consider the sum
  \[ F = \sum_{i = 0}^{N-1} f(\zeta^i)\zeta^{-i\ell} = \sum_{k = -\infty}^\infty f_k \sum_{i = 0}^{N-1} \zeta^{i(k-\ell)} = \sum_{\substack{k \in \ZZ \\ N \mid k-\ell}} N f_k = -\sum_{k = -\infty}^\infty f_{\ell+Nk}\,. \]
  On the one hand, $\abs{F} \leq \max_{i = 0}^{N-1} \abs{f(\zeta^i)\zeta^{-i\ell}} = \max_{i = 0}^{N-1} \abs{f(\zeta^i)} < \epsilon$.
  On the other hand, $\abs{F+f_\ell} \leq \max_{k \in \ZZ, k \neq 0} \abs{f_{\ell+Nk}} < \epsilon$; thus $\abs{f_\ell} < \epsilon$ as desired.
\end{proof}

Now recall from \cref{prop:lattice_limit} that if $\dd(\omega,\Lambda) \to \infty$ then $\Lambda +I\omega \to \Lambda$.
\begin{proposition} \label{prop:mFormSeries}
  Let $f$ be a weak modular form of weight $k$ on $\LL^r$, $\Lambda$ a lattice of rank $r-1$, $I$ a fractional ideal of $A$, and $\omega \in \CCi$ a variable such that $\dd(\omega,\Lambda) \to \infty$.
  Then $f(\Lambda+I\omega)$ has a `Fourier' expansion
  \[ f(\Lambda+I\omega) = \sum_{n=-\infty}^\infty f_{I,n}(\Lambda) e_{I^{-1}\Lambda}(\omega)^{-n} \]
  where each $f_{I,n}$ is a weak modular form of weight $k-n$ on $\LL^{r-1}$.

  Additionally, if $f$ is the restriction to $\LL^{r}$ of a \emph{strong} modular form, then $f_{I,n} = 0$ for $n < 0$, and $f_{I,0} = 0$ if $f$ is also the restriction of a cusp form.\footnote{Note that for large enough $n$ the $f_{I,n}$ above have negative weight but are not necessarily the zero form, even if $f$ is a strong modular form.}
\end{proposition}

\begin{proof}
  Consider the functions $a(\omega) = f(\Lambda+I\omega)$ and $b(\omega) = e_{I^{-1}\Lambda}(\omega)^{-1}$:
  \begin{itemize}
    \item These are holomorphic functions,
    \item $\begin{multlined}[t]
      b(\omega_1) = b(\omega_2) \iff e_{I^{-1}\Lambda}(\omega_1-\omega_2) = 0 \iff \omega_1-\omega_2 \in I^{-1}\Lambda \\
      \iff I(\omega_1-\omega_2) \subseteq \Lambda \iff \Lambda+I\omega_1 = \Lambda+I\omega_2 \implies a(\omega_1) = a(\omega_2),\ \text{and}
    \end{multlined}$
    \item $b'(\omega) = 1 \neq 0$ for all $\omega$.
  \end{itemize}
  Thus the composite $c = a \circ b^{-1}$ is well defined and holomorphic.
  The domain of $c$ is the range of $e_{I^{-1}\Lambda}^{-1}$, which is $\CCi\setminus\set{0}$ since $e_{I^{-1}\Lambda}$ is entire.
  Hence, being holomorphic, it has a Laurent series expansion $c(z) = \sum_{n=-\infty}^\infty f_n z^n$ convergent in a neighbourhood of $0$.
  Thus \[ f(\Lambda+I\omega) = a(\omega) = c(b(\omega)) = \sum_{n=-\infty}^\infty f_{I,n}(\Lambda) e_{I^{-1}\Lambda}(\omega)^{-n} \]
  for small enough $b(\omega) = e_{I^{-1}\Lambda}(\omega)^{-1}$, i.e. for large enough $\dd(\omega,\Lambda)$.

  Now since $f$ is homogeneous of degree $-k$, we have that for $r \in \CCi^\times$,
  \begin{align*}
    & \sum_{n=-\infty}^\infty r^{-k} f_{I,n}(\Lambda) e_{I^{-1}\Lambda}(\omega)^{-n} = r^{-k} f(\Lambda+I\omega) = f(r\Lambda+rI\omega) \\
    =& \sum_{n=-\infty}^\infty f_{I,n}(r\Lambda) e_{I^{-1}r\Lambda}(r\omega)^{-n} = \sum_{n=-\infty}^\infty f_{I,n}(r\Lambda) \brackets{r e_{I^{-1}\Lambda}(\omega)}^{-n} \\
    \implies& r^{-k} f_{I,n}(\Lambda) = r^{-n} f_{I,n}(r\Lambda) \lsptext{for each} n \in \ZZ,
  \end{align*}
  so that each $f_{I,n}$ is homogeneous of degree $n-k$.
  Finally, $a$ and $b$ above are holomorphic functions of $\Lambda$ (as long as $\Lambda$ stays away from $\FF_\infty \omega$), and so the series coefficients $f_{I,n}$ are too.
  
  Additionally, if $f$ is the restriction of a \emph{strong} modular form, then $f$ is continuous as $\dd(\omega,\Lambda) \to \infty$, and so $c(z)$ is continuous and thus bounded as $z \to 0$.
  So applying \cref{lem:funcCoeffsSmall} to $c(rz)$ where $r \to 0$, we get that the coefficients $f_n r^n$ are uniformly bounded for $n \in \ZZ$ and small enough $r$; thus $f_n = 0$ for $n < 0$.
  Finally, if $f$ is the restriction of a \emph{cusp} form, then the aforementioned uniform bound goes to $0$ as $r \to 0$ which gives $f_0 = 0$.
\end{proof}

\subsubsection{Cusp expansions on $\LL_N^r$}

Similarly to the case of $\LL^r$, we have cusp expansions for modular forms with level, but we need a bit more setup to handle the level structure:

\begin{lemma} \label{lem:lmLimitLevels}
  Let $\Lambda$ be a lattice of rank $r-1$, and $\iota : N^{-1}\Lambda/\Lambda \into \parens{N^{-1}/A}^r$ an $r$-inverse level $N$ structure.
  Let $I$ be a fractional ideal, let $\omega \in \CCi$ be a variable such that $\dd(\omega,\Lambda) \to \infty$, and let $j : N^{-1}I/I \to \parens{N^{-1}/A}^r$ be an $A/N$-module injection such that $\Im{\iota} \cap \Im{j} = \set{0}$.
  
  Then there is a unique $r$-inverse level $N$ structure $(\iota,j)_\omega$ for $\Lambda+I\omega$ defined by $(\iota,j)_\omega(\lambda+i\omega) = \iota(\lambda)+j(i)$ for $\lambda \in N^{-1}\Lambda/\Lambda$ and $i \in N^{-1}I/I$.
  Moreover, $(\Lambda+I\omega,(\iota,j)_\omega) \to (\Lambda,\iota)$ with respect to $\dd_{\LLNRi}$ as $\dd(\omega,\Lambda) \to \infty$.
\end{lemma}

The proof is technical but not hard, and is omitted here; see \cite{baker2020lattice} for details.

Note that if we have a full $r$-inverse level structure $p$ for $\Lambda+I\omega$, we can recover an $r$-inverse level structure $\iota$ for $\Lambda$ and the injection $j : N^{-1}I/I \into \parens{N^{-1}/A}^r$ by $\iota(\lambda) = p(\lambda)$ and $j(i) = p(i\omega)$ and we will have $\Im{\iota} \cap \Im{j} = \set{0}$ as in the lemma.

  

\begin{proposition} \label{prop:mFormNSeries}
  Let $\Lambda$, $\iota$, $I$, $\omega$, $j$, and $(\iota,j)_{\omega}$ be defined as in \cref{lem:lmLimitLevels}.
  Additionally, let $f$ be a weak modular form of weight $k$ and rank $r$ for $K(N)$.
  Then $f(\Lambda+I\omega,(\iota,j)_\omega)$ has a `Fourier' expansion
  \[ f(\Lambda+I\omega,(\iota,j)_\omega) = \sum_{n=-\infty}^\infty f_{I,j,n}(\Lambda,\iota) e_{I^{-1}\Lambda}(\omega)^{-n}, \]
  where each $f_{I,j,n}$ is a weak modular form of weight $k-n$ and rank $r-1$.

  Additionally, if $f$ is the restriction to $\LL_N^r$ of a \emph{strong} modular form, then $f_{I,j,n} = 0$ for $n < 0$, and if $f$ is (the restriction of) a cusp form, then $f_{I,j,0} = 0$ as well.
\end{proposition}

The proof is omitted since it is quite similar to the proof of \cref{prop:mFormSeries}.

%% file: text/4_4_mf_BBP_relation.tex
\subsection{Relation with BBP definitions}

We will investigate the relation between our modular forms and that defined by Basson, Breuer, and Pink\footnote{For the remainder of this section, we will largely refer to the authors Basson, Breuer, and Pink collectively as \emph{BBP}.} in \cite{BBP}, which was previously on the ArXiv as the trio of papers \cite{BBPI,BBPII,BBPIII}.
There are some differences in notation between our work and that of \cite{BBP}; for instance, there the elements of $\Omega^r$ and $\Psi^r$ are considered as column vectors as opposed to row vectors, which results in differences of definition of various actions.
For this reason, we will translate relevant definitions and results from \cite{BBP} to our context before referring to them.

Let $\xi \in \CCi^\times$ be fixed for the remainder of this paper.
Note that each $\omega \in \Omega^r \simeq \rquotient{\Psi^r}{\CCi^\times}$ has a representative $\overline{\omega} \in \Psi^r$ with last component $\overline{\omega}_r = \xi$; we denote this representative by $\psi(\omega)$.
The map $ \Omega^r \to \Psi^r$, $\omega \mapsto \psi(\omega)$ is rigid analytic, since the rigid analytic structure on $\Psi^r$ has been defined to be that of the product $\Omega^r \times \CCi$ via the isomorphism \cref{prop:PsiIsoRigid}.

\begin{definition} \label{def:jGammaOmega}
  For $\gamma \in \GL[r]{F}$, $\psi \in \Psi^r$, and $\omega \in \Omega^r$, we define
  \[ j(\gamma,\psi) \defeq \xi^{-1} \cdot \parens{\psi\gamma^{-1}}_r, \]
  where $(\psi\gamma^{-1})_r$ denotes the last entry of the vector $\psi\gamma^{-1}$, and similarly
  \[ j(\gamma,\omega) \defeq \xi^{-1} \cdot \parens{\psi(\omega)\gamma^{-1}}_r = j(\gamma,\psi(\omega)). \qedhere \]
\end{definition}

This $j(\gamma,\psi)$ serves as a normalisation factor, preserving the last component being equal to $\xi$ under the action of $\GL[r]{F}$:
\begin{proposition} \label{prop:jGammaOmega}
  For $\omega \in \Omega^r$ and $\gamma \in \GL[r]{F}$,
  \[ \psi(\omega\gamma^{-1}) = j(\gamma,\psi(\omega))^{-1} \cdot \psi(\omega)\gamma^{-1} = j(\gamma,\omega)^{-1} \cdot \psi(\omega)\gamma^{-1}. \qedhere \]
\end{proposition}
\begin{proof}
  With $\simeq$
  denoting similarity up to a multiple of $\CCi^\times$, we have that $\psi(\omega\gamma^{-1}) \simeq \omega\gamma^{-1} \simeq \psi(\omega)\gamma^{-1}$.
  Inspecting the last components of the first and last terms yields the relevant scaling factor $j(\gamma,\psi(\omega))$.
\end{proof}

\begin{definition} \label{def:OmegaRSlash}
  For $f : \Omega^r \to \CCi$, $k \in \ZZ$, and $\gamma \in \GL[r]{F}$, we define
  \[ f|_k \gamma : \Omega^r \longto \CCi, \qquad \omega \longmapsto j(\gamma,\omega)^{-k} f(\omega \gamma^{-1}). \qedhere \]
\end{definition}

It is easily shown that this `slash operator' $|_k$ induces a right action of $\GL[r]{F}$ on the set of functions $\Omega^r \to \CCi$.

For future reference, we include here the definition of a weak modular form in \cite{BBPI}, which we will refer to as a \emph{weak BBP modular form} to contrast with the weak modular forms defined earlier.\footnote{We do not include the `type' parameter $m$ investigated in BBP's work; all the modular forms considered in this paper are of type 0.}
\begin{definition} \label{def:BBP_weakForm}
  Consider an integer $k$ and an arithmetic subgroup $\Gamma < \GL[r]{F}$.
  A \emph{weak BBP modular form $f$ of weight $k$ for $\Gamma$} is a holomorphic function $f : \Omega^r \to \CCi$ such that for all $\gamma \in \Gamma$, $f|_k\gamma = f$.

  The $\CCi$-vector space of weak BBP modular forms of weight $k$ for $\Gamma$ will be denoted by $\cw_k(\Gamma)$, and the graded $\CCi$-algebra of weak BBP modular forms for $\Gamma$ will be denoted by $\cw_*(\Gamma)$.
\end{definition}

With the decomposition of $\LL_N^r$ into $\bigsqcup_{g \in H}\ \lquotient{\Gamma_g}{\Psi^r}$ in \cref{prop:LLNR_components} in mind, with $H$ being a set of representatives in $\GL[r]{\finadele}$ for {\small $\lrquotient{\GL[r]{F}}{\GL[r]{\finadele}}{K(N)}$} and $\Gamma_g$ being defined as $g K(N) g^{-1} \cap \GL[r]{F}$ for each $g \in H$, we define maps from modular forms on $\LL_N^r$ to BBP modular forms on the quotients $\lquotient{\Gamma_g}{\Omega^r}$.
These maps essentially separate the function $f$ into its values on each of the irreducible components of $\LL_N^r$.

\begin{definition} \label{def:wFormToComps}
  For a function $f : \LL_N^r \to \CCi$ and $g \in \GL[r]{\finadele}$, we define the function $f_g : \Omega^r \to \CCi$ by
  \[ f_g(\omega) = f\parens{\Theta([\psi(\omega),g])}; \]
  here $\Theta$ is the bijection defined in \cref{thm:PsiDoubleQuotientIso}.

  For $f : \LLNRi \to \CCi$, we let $f_g$ denote the same construction using the restriction of $f$ to the main stratum $\LL_N^r$ of $\LLNRi$.
\end{definition}

\begin{proposition} \label{prop:wFormToComps_slashF}
  For $f \in \WeakMF_N^{k,r}$ and $\gamma \in \GL[r]{F}$, we have that $f_g|_k\gamma = f_{\gamma^{-1}g}$.
\end{proposition}
\begin{proof}
  Let $\omega \in \Omega^r$.
  Then since $f$ is homogeneous of degree $-k$,
  \begin{align*}
    (f_g|_k\gamma)(\omega) &= j(\gamma,\omega)^{-k} f_g(\omega\gamma^{-1}) \\
    &= j(\gamma,\omega)^{-k} f\parens{\Theta([\psi(\omega\gamma^{-1}),g])} \\
    &= j(\gamma,\omega)^{-k} f\parens{\Theta([j(\gamma,\psi(\omega))^{-1} \cdot \psi(\omega)\gamma^{-1},g])} \\
    &= j(\gamma,\omega)^{-k} f\parens{j(\gamma,\omega)^{-1} \Theta([\psi(\omega)\gamma^{-1},g])} \\
    &= j(\gamma,\omega)^{-k} j(\gamma,\omega)^k f\parens{\Theta([\psi(\omega),\gamma^{-1}g])} \\
    &= f_{\gamma^{-1}g}(\omega)
  \end{align*}
  since $[\psi\gamma^{-1},g] = [\psi,\gamma^{-1}g]$ because $\gamma \in \GL[r]{F}$.
\end{proof}

\begin{theorem} \label{thm:wFormToComps_weak}
  If $f \in \WeakMF_N^{k,r}$ is a weak modular form of weight $k$, then each $f_g$ is a weak BBP modular form for $\Gamma_g$.
\end{theorem}
\begin{proof}
  Since $f$, $\Theta$, and the map $\omega \mapsto \psi(\omega)$ are rigid analytic, $f_g$ is also rigid analytic.
  So it remains to show that $f_g$ satisfies the relevant transformation relation in \cref{def:BBP_weakForm}.
  Now $f_g|_k\gamma = f_{\gamma^{-1}g}$ and $g^{-1} \gamma g \in K(N)$, so that in the quotient $\rquotient{\GL[r]{\finadele}}{K(N)}$ we have $[\gamma^{-1}g] = [\gamma^{-1} g g^{-1} \gamma g] = [g]$.
  Thus
  \[ f_{\gamma^{-1}g}(\omega) = f\parens{\Theta([\psi(\omega),\gamma^{-1}g])} = f\parens{\Theta([\psi(\omega),g])} = f_g(\omega). \qedhere \]
\end{proof}

\begin{proposition} \label{prop:wFormToComps_algebra}
  For each $g \in \GL[r]{\finadele}$ the map
  \[ \WeakMF_N^r \longto \cw_*(\Gamma_g), \qquad f \longmapsto f_g \]
  is a homomorphism of graded $\CCi$-algebras.
\end{proposition}
\begin{proof}
  \cref{thm:wFormToComps_weak} shows that the above map sends a weak modular form of weight $k$ to a weak BBP modular form of weight $k$, and from \cref{def:wFormToComps} it preserves scaling by $\CCi$ and multiplication of weak modular forms.
\end{proof}

\begin{proposition} \label{prop:wFormToComps_Injective}
  If $H$ is a set of representatives in $\GL[r]{\finadele}$ for the double quotient
  \[ \lrquotient{\GL[r]{F}}{\GL[r]{\finadele}}{K(N)}, \]
  then the following homomorphism of graded $\CCi$-algebras is injective:
  \[ \WeakMF_N^r \longto \prod_{g \in H} \cw_*(\Gamma_g), \qquad f \longmapsto (f_g)_{g \in H}. \qedhere \]
\end{proposition}
\begin{proof}
  By \cref{prop:wFormToComps_algebra}, it is enough to show that if if $f \in \WeakMF_N^{k,r}$ is mapped to $0 = (0)_{g \in H}$ then $f = 0$.
  So assume $f$ to be such that $f_g = 0$ for all $g \in H$.

  So let $g \in H$, so that $f_g = f(\Theta([\psi(\omega),g])) = 0$ for all $\omega \in \Omega^r$.
  Then scaling by $\CCi^\times$ we see that $f(\Lambda,\alpha) = 0$ for all $(\Lambda,\alpha)$ in the irreducible component corresponding to $g$.
  Since $H$ is a complete set of representatives for $\lrquotient{\GL[r]{F}}{\GL[r]{\finadele}}{K(N)}$, $f = 0$ over all the irreducible components of $\LL_N^r$ and thus is identically zero.
\end{proof}

Now after our definition of weak modular forms two subsections ago we defined strong modular forms, which can be seen as weak modular forms which satisfy continuity conditions at the boundary $\bd_N^r$ of $\LLNRi$.
Similarly, BBP modular forms are defined as weak BBP modular forms which satisfy an additional condition `at infinity':

\emph{Holomorphicity at infinity} and \emph{going to zero at infinity} are conditions regarding the behaviour of a holomorphic function on $\Omega^r$ as the first component of $\omega \in \Omega^r$ goes to infinity, which we will not repeat here.
We will, however, include the following characterisation of holomorphicity at infinity and going to zero at infinity which was communicated to us by BBP; the proof is included in the appendix of \cite{baker2020lattice}:

\begin{proposition} \label{prop:BBP_verticalStrips}
  Let $\Gamma < \GL[r]{F}$ be an arithmetic subgroup and $f : \Omega^r \to \CCi$ be a holomorphic function such that $f|_k\gamma = f$ for all $\gamma \in \Gamma$.
  
  Then $f$ is holomorphic at infinity if and only if it is bounded on every vertical line, \ie for every vector $\psi' = (\psi_2, \dotsc, \psi_r) \in \Psi^{r-1}$ there are real numbers $N > 0$ and $R > 0$ such that for all $\psi_1 \in \CCi$ satisfying $\dd(\psi_1,\psi'F_\infty^{r-1}) > R$ we have $\abs{f\pcoord{\psi_1 : \dotsc : \psi_r}} < N$.
  
  Moreover, $f$ goes to zero at infinity if and only if it goes to zero on each vertical line, \ie for every $\psi' = (\psi_2, \dotsc, \psi_r) \in \Psi^{r-1}$ and $\epsilon > 0$ there is an $R > 0$ such that $\abs{f\pcoord{\psi_1 : \dotsc : \psi_r}} < \epsilon$ for $\dd\parens{\psi_1,\psi'F_\infty^{r-1}} > R$.
\end{proposition}

\begin{definition} \label{def:BBP_strongMForm}
  For an integer $k$ and an arithmetic subgroup $\Gamma < \GL[r]{F}$, a \emph{strong BBP modular form $f$ of weight $k$ for $\Gamma$} is a weak BBP modular form such that $f|_k\gamma$ is holomorphic at infinity for all $\gamma \in \GL[r]{F}$.

  The $\CCi$-vector space of strong BBP modular forms of weight $k$ for $\Gamma$ is denoted by $\cm_k(\Gamma)$, with the graded $\CCi$-algebra of all strong BBP modular forms denoted by $\cm_*(\Gamma)$.
  Also, the adjective \emph{strong} will be omitted unless contrasting with BBP \emph{weak} modular forms.
\end{definition}
\begin{definition} \label{def:BBP_cuspForm}
  A \emph{BBP cusp form of weight $k$ for $\Gamma$} is a strong BBP modular form $f$ such that $f|_k\gamma$ goes to zero at infinity for all $\gamma \in \GL[r]{F}$.

  The $\CCi$-vector space of BBP cusp forms of weight $k$ for $\Gamma$ will be denoted by $\cs_k(\Gamma)$, with the graded $\CCi$-algebra of all BBP cusp forms denoted by $\cs_*(\Gamma)$.
\end{definition}

Using \cref{prop:BBP_verticalStrips}, we can prove the following:
\begin{proposition} \label{prop:mFormToComps}
  If $f \in \StrongMF_N^{k,r}$ is a strong modular form and $g \in \GL[r]{\finadele}$, then $f_g$ is a strong BBP modular form of weight $k$ for $\Gamma_g$.
\end{proposition}
\begin{proof}
  Let $\gamma \in \GL[r]{F}$ and $\psi' \in \Psi^{r-1}$, where without loss of generality $\psi_r = \xi$, and let $\psi = (\psi_1, \dotsc, \psi_r)$ where $\psi_1 \in \CCi$ is variable.
  Also let $F^r \cap \gamma^{-1}g\hat{A}^r = (I_1, I_2, \dotsc, I_r)^T$ where the $I_i$ are fractional ideals, so that for $(\Lambda,\alpha) = \Theta([\psi,\gamma^{-1}g])$ we have $\Lambda = \psi(F^r \cap \gamma^{-1}g\hat{A}^r) = I_1\psi_1 +I_2\psi_2 +\dotsb I_r\psi_r$.
  Then as in \cref{prop:lattice_limit} we have that as $\dd(\psi_1,\psi'F_\infty^{r-1}) \to \infty$ the lattice $\Lambda$ tends to $\Lambda' \defeq I_2\psi_2 +\dotsb +I_r\psi_r$.
  Also, as in the proof of \cref{prop:LLNR_dense} the $r$-inverse level $N$ structure $\iota = \alpha^{-1}$ tends to the $r$-inverse level $N$ structure
  \[ \iota' : N^{-1}\Lambda'/\Lambda' \longinto \parens{N^{-1}/A}^r,\quad [\lambda_2\psi_2 +\dotsb +\lambda_r\psi_r]_{\Lambda'} \longmapsto \iota([\lambda_2\psi_2 +\dotsb +\lambda_r\psi_r]_{\Lambda})\,. \]
  Thus as $\dd(\psi_1,\psi'F_\infty^{r-1}) \to \infty$ we have that $(\Lambda,\iota) \to (\Lambda',\iota')$, with $(\Lambda',\iota')$ not dependent on $\psi_1$.
  Thus since $f$ is continuous on $\LLNRi$ we have that
  \[ (f_g|_k\gamma)(\omega) = f_{\gamma^{-1}g}(\omega) = f(\Lambda,\iota) \longto f(\Lambda',\iota'); \]
  since the limit exists, $f_g|_k\gamma$ is bounded on the vertical line defined by $\psi'$.
\end{proof}

\begin{proposition} \label{prop:cFormToComps}
  If $f \in \CuspMF_N^{k,r}$ is a cusp form and $g \in \GL[r]{\finadele}$, then $f_g$ is a BBP cusp form of weight $k$ for $\Gamma_g$.
\end{proposition}
\begin{proof}
  Let $\gamma \in \GL[r]{F}$.
  Then similarly to \cref{prop:mFormToComps},
  \[ (f_g|_k\gamma)(\omega) = f_{\gamma^{-1}g}(\omega) = f(\Lambda,\iota) \longto f(\Lambda',\iota') = 0, \]
  the zero due to $f$ being a cusp form.
  Since the limit is zero, $f_g|_k\gamma$ goes to zero at infinity on the vertical line defined by $\psi'$.
\end{proof}

\begin{theorem} \label{prop:mFormToComps_Injective}
  If $H$ is a set of representatives in $\GL[r]{\finadele}$ for the double quotient
  \[ \lrquotient{\GL[r]{F}}{\GL[r]{\finadele}}{K(N)}, \]
  then the following homomorphisms of graded $\CCi$-algebras are injective:
  \[ \StrongMF_N^r \longto \prod_{g \in H} \cm_*(\Gamma_g), \qquad \CuspMF_N^r \longto \prod_{g \in H} \cs_*(\Gamma_g), \qquad f \longmapsto (f_g)_{g \in H}. \qedhere \]
\end{theorem}
\begin{proof}
  These are shown by \cref{prop:wFormToComps_Injective,prop:mFormToComps,prop:cFormToComps}.
\end{proof}

\begin{corollary} \label{coro:mForms_finiteDimension}
  For each $k \in \ZZ$, $\StrongMF_N^{k,r}$ is a finite-dimensional $\CCi$-vector space.
\end{corollary}
\begin{proof}
  By \cref{prop:mFormToComps} and \cref{prop:mFormToComps_Injective}, $\StrongMF_N^{k,r}$ can be seen as a subspace of $\prod_{g \in H} \cm_k(\Gamma_g)$; by \cite[Theorem 11.1]{BBPII} this space is finite dimensional.
\end{proof}

\begin{paragraph}
  The above injection of strong modular forms is not a bijection in general, as by \cref{thm:LLNR_irredComponent_completion} the irreducible components of $\LL_N^r$ share a common boundary and hence the condition of continuity on $\LLNRi$ imposes relations between the modular form's value on different irreducible components, whereas in the product $\prod_{g \in H} \cm_*(\Gamma_g)$ the modular forms on each component are completely independent.
  As a more explicit example, consider the modular form of weight $0$ on $\prod_{g \in H} \cm_*(\Gamma_g)$ which is defined to have constant value $1$ on $\cm_*(\Gamma_{g_0})$ for some particular $g_0 \in \GL[r]{\finadele}$ and constant value $0$ on $\cm_*(\Gamma_g)$ for $g \in H - \set{g_0}$.
\end{paragraph}
However, the above injection does become a bijection when restricting to cusp forms, which are zero on the boundary:

\cuspFormsToCompsBijective
\begin{proof}
  Let $(f_g)_{g \in H} \in \prod_{g \in H} \cs_k(\Gamma_g)$ be a tuple of BBP cusp forms of weight $k$, and define the weak modular form $f \in \CuspMF_N^{k,r}$ by $f(t \Theta([\psi(\omega),g])) = t^{-k} f_g(\omega)$ for $\omega \in \Omega^r$, $t \in \CCi$ and $g \in H$; the holomorphicity and homogeneity of $f$ follows from the corresponding properties of the $f_g$.
  All that remains is to show that $f$ is continuous when considered as a function on $\LLNRi$, defined to be zero on the boundary $\bd_N^r = \LLNRi-\LL_N^r$.
  However, since by \cref{coro:LLNR_irredComponentBdClosed} the unions $C_g \cup \bd_N^r$, where $C_g$ denotes the irreducible component of $\LL_N^r$ corresponding to $g \in H$, are each closed in $\LLNRi$ and together cover the whole space $\LLNRi$, and there are finitely many such, it is enough to show that $f$ is continuous on each $C_g \cup \bd_N^r$; as in the proofs of \cref{prop:mFormToComps,prop:cFormToComps}, this follows from the fact that each $f_g$ is a cusp form.
\end{proof}

%% file: text/4_5_mf_gek_relation.tex
\subsection{Relation with Gekeler work}

In this subsection, we will investigate the relation between our modular forms and those defined by Gekeler in \cite{gekelerHigherVII}, being part of his series \cite{gekelerHigherI,gekelerHigherII,gekelerHigherIII,gekelerHigherIV,gekelerHigherV,gekelerHigherVI,gekelerHigherVII,gekeler2026expansions} of papers on a theory of Drinfeld modular forms of higher rank.
In addition, Gekeler defines modular forms of an integer type $m$ modulo $q-1$; as in the previous section, we note that the modular forms defined in this paper are only those of type $0$.


We begin with a bijection between the spaces on which modular forms are defined in Gekeler's work and in ours, which we will prove have the same topologies.
In this section, as per Gekeler's notation in \cite{gekelerHigherVII}, we let $r \geq 1$ be fixed, and for $F$-subspaces $U$ of $F^r$ we denote by $\Psi_U$ the set of discrete embeddings $i : U \to \CCi$, with the strong topology given there.
In particular, $\Psi^r = \Psi_{F^r}$ and $\overline{\Psi}^r = \cup_{U \leq F^r} \Psi_U$.
Note that as in \cite[\S 1.2]{gekelerHigherVII}, $\omega \in \Psi^r$ can be seen both as an embedding $F^r \to \CCi$ and as $\omega = (\omega_1, \dotsc, \omega_r) \in \CCi^r$ via the standard basis of $F^r$; in this subsection, we will primarily take the former view.

Now for a fixed $A$-lattice $Y$ (i.e.~projective $A$-module of rank $r$) in $F^r$, it can be written as $Y = I_1b_1+\dotsb+I_rb_r$ for some fractional ideals $I_i$ of $A$ and linearly independent $b_i \in F^r$; thus there is a corresponding $\pi(Y) = [I_1\dotsm I_r] \in \Cl(F)$ and thus a corresponding irreducible component $C$ of $\LL^r$ with $\pi(C) = \pi(Y)$.
\begin{definition} \label{def:gek_space_bijection}
For a fixed $A$-lattice $Y \subset F^r$ of rank $r$, we define the mapping
\begin{align*}
    \Xi_Y : \lquotient{\GL{Y}}{\overline{\Psi}^r} &\longto \overline{C} = C \cup \LL^{\leq r-1} \\
    i : U \to \CCi &\longmapsto i(Y\cap U) \qedhere
\end{align*}
\end{definition}
$\Xi_Y$ maps $\Psi^r$ to $C$ while it maps $\Psi_U$ (for a subspace $U < F^r$ of dimension $s < r$) to $\LL^s$.
\begin{proposition} \label{prop:Gek:isotopology} ~
\begin{enumerate}[leftmargin=2\bigskipamount]
    \item $\Xi_Y$ is a well-defined bijection.
    \item If, in $\overline{\Psi}^r$, we have $(U',i') \to (U,i)$ for fixed $U \subseteq U'$ and $i : U \to \CCi$ and variable $i' : U' \to \CCi$, then $i'(Y\cap U') \to i(Y \cap U)$.
    \item If a variable $\Lambda' \to \Lambda$ in $\overline{C}$, then there is a variable embedding $i' : U' \to \CCi$ and a fixed embedding $i : U \to \CCi$ such that $i' \to i$, $\Xi_Y(i') = \Lambda'$, and $\Xi_Y(i) = \Lambda$.
    \qedhere
\end{enumerate}
\end{proposition}

\begin{proof}~
\begin{enumerate}[leftmargin=2\bigskipamount,itemsep=\medskipamount]
    \item
    \begin{description}[leftmargin=0pt,itemsep=\smallskipamount]
        \item[Well-defined] Let $\gamma \in \GL{Y}$ (i.e.~$\gamma \in \GL{F^r}$ fixes $Y$), let $r_\gamma : F^r \to F^r$, $x \mapsto x\gamma$, and let $i : U \to \CCi$ be an embedding.
        Then $\gamma i = i \circ r_\gamma : U\gamma^{-1} \to \CCi$,
        \[ \rsptext{so that} \gamma i(Y \cap U\gamma^{-1}) = i((Y \cap U\gamma^{-1})\gamma) = i(Y\gamma \cap U\gamma^{-1}\gamma) = i(Y\cap U); \]
        thus $\Xi_Y$ is well defined.
        \item[Injective] Suppose that we have embeddings $i : U \to \CCi$ and $i' : U' \to \CCi$ such that $i(Y \cap U) = i'(Y \cap U')$ has rank $s \leq r$.
        Then by \cite[\S 22.14]{curtisreiner2006repr} and since $Y \cap F(Y \cap U) = Y \cap U$, we can decompose $Y = (Y \cap U) + Z$ where $Z \subset F^r$ is a projective $A$ module of rank $r-s$; similarly, $Y = (Y \cap U') + Z'$ where $Z'$ also has rank $r-s$.
        Now
        \begin{gather*}
            \pi(i(Y \cap U)) = \pi(i'(Y \cap U')) \implies \pi(Y \cap U) = \pi(Y \cap U'), \lsptext{and} \\
            \pi(Y \cap U) \pi(Z) = \pi((Y \cap U) + Z) = \pi(Y) =
            \pi((Y \cap U') + Z') =
            \pi(Y \cap U') \pi(Z'),
        \end{gather*}
        and so $\pi(Z) = \pi(Z')$; and so since $Z$ and $Z'$ have the same rank there is an $A$-module bijection $t : Z \to Z'$.
        Thus there is an $A$-module bijection $\gamma : Y = (Y \cap U) + Z \to Y = (Y \cap U') + Z'$, given by $(i')^{-1} \circ i$ on $Y \cap U$ and by $t$ on $Z$, and so $i$ and $i'$ are equivalent under the action of $\GL{Y}$ on $\overline{\Psi}^r$.
        \item[Surjective] Let $\Lambda \in \overline{C} = C \cup \LL^{\leq r-1}$, and let $I$ be an $A$-fractional ideal such that $\pi(Y) = [I]$.
        If $\Lambda \in C$, or equivalently if $\Lambda$ has rank $r$, then we can write $\Lambda = A \omega_1 +\dotsb +A \omega_{r-1} + I \omega_r$ and $Y = A b_1 +\dotsb A b_{r-1} +I b_r$ for $F_\infty$-linearly independent $\omega_j \in \CCi$ and $F$-linearly independent $b_j \in F^r$.
        Then with the $F$-linear $i : F^r \to \CCi$ given by $i(b_j) = \omega_j$ for each $j$, we get that $i(Y\cap F^r) = i(Y) = \Lambda$.
        If on the other hand $\Lambda \in \LL^{\leq r-1}$, let $\Lambda$ have rank $s < r$, so that it can be written as $\Lambda = A \omega_1 +\dotsb +A \omega_{s-1} +J \omega_s$ for $F_\infty$-linearly independent $\omega_j \in \CCi$ and a fractional ideal $J$.
        Now for the $A$-module $Q = A \oplus \dotsb \oplus A \oplus I/J$ of rank $r-s$, since $\pi(\Lambda \oplus Q) = [J] [I/J] = [I] = \pi(Y)$ and they have the same rank, we have that $Y \simeq \Lambda \oplus Q$; in particular, we can decompose $Y = Y' + Q'$ for $\Lambda \simeq Y' \subset Y \subset F^r$ and $Q \simeq Q' \subseteq Y \subset F^r$.
        The bijection $Y' \simeq \Lambda$ can be extended $F$-linearly to a map $i : U \defeq F Y' \to \CCi$, so that $i(Y \cap U) = i(Y') = \Lambda$.
    \end{description}
    \item We show that $e_{i'(Y \cap U')}(z) \to e_{i(Y \cap U)}(z)$ uniformly on $\abs{z} \leq 1$ using the characterisation of the topology in \cite[\S 1.3]{gekelerHigherVII}.
    Now similarly to the proof of \cref{prop:dd_LL_radius}, it suffices to prove this uniform convergence on $\abs{z} \leq \epsilon$ for some small $\epsilon > 0$; we thus choose such an $\epsilon$ that $\epsilon < \abs{i'(\ell)}$ for nonzero $\ell \in Y \cap U'$ and $\epsilon < \abs{i(\ell)}$ for nonzero $\ell \in Y \cap U$.
    Now
    \begin{align*}
        &\mathrel{\phantom{=}} e_{i'(Y\cap U')}(z) - e_{i(Y\cap U)}(z) \\
        &= e_{i'(Y\cap U')}(z) e_{i(Y\cap U)}(z) \brackets{e_{i(Y\cap U)}^{-1}(z) - e_{i'(Y\cap U')}^{-1}(z)} \\
        &= e_{i'(Y\cap U')}(z) e_{i(Y\cap U)}(z) \brackets{\sum_{\lambda \in i(Y \cap U)} \frac{1}{z-\lambda} -\sum_{\lambda \in i'(Y\cap U')} \frac{1}{z-\lambda}} \\
        &= e_{i'(Y\cap U')}(z) e_{i(Y\cap U)}(z) \brackets{\sum_{\ell \in Y\cap U} \parens{\frac{1}{z-i(\ell)}-\frac{1}{z-i'(\ell)}} -\sum_{\ell \in (Y \cap U')\setminus U} \frac{1}{z-i'(\ell)}} \\
        &= e_{i'(Y\cap U')}(z) e_{i(Y\cap U)}(z) \brackets{\sum_{\ell \in Y\cap U} \frac{i(\ell)-i'(\ell)}{(z-i(\ell))(z-i'(\ell))} -\sum_{\ell \in (Y \cap U')\setminus U} \frac{1}{z-i'(\ell)}}. \tag{A}
    \end{align*}
    Now for $0 \neq \ell \in Y \cap U$, we have that $i'(\ell) \to i(\ell)$, and for $0 \neq \ell \in (Y \cap U') \setminus U$, we have that $\abs{i'(\ell)} \to \infty$ uniformly in $\ell$; thus $i'(Y \cap U')$ is uniformly bounded away from $0$, and so $e_{i(Y\cap U)}(z) = z \prodp_{\lambda \in i(Y\cap U)} (1-\frac{z}{\lambda})$ is bounded above for $\abs{z} \leq \epsilon$, and similarly for $e_{i'(Y\cap U')}(z)$.
    Moreover, since $\abs{i'(\ell)} \to \infty$ uniformly in $\ell$, the second sum in (A) tends to $0$ uniformly over $\abs{z} \leq \epsilon$.
    Finally, for the first sum in (A), let $R>0$ be large.
    Note that since $i$ and $i'$ are discrete embeddings, $L_R \defeq \setst{\ell \in Y\cap U}{\abs{i(\ell)} \leq R}$ is finite, and so
    \[ \abs{\sum_{\ell \in L_R} \frac{i(\ell)-i'(\ell)}{(z-i(\ell))(z-i'(\ell))}} \leq \max_{\substack{\ell \in L_R \\ \ell \neq 0}} \abs{\frac{i(\ell)-i'(\ell)}{i(\ell)i'(\ell)}} \longto 0 \;\;\;\text{since each}\ i'(\ell) \longto i(\ell) \neq 0. \]
    On the other hand,
    \begin{align*}
        &\mathrel{\phantom{\leq}} \abs{\sum_{\ell \in (Y\cap U) \setminus L_R} \frac{i(\ell)-i'(\ell)}{(z-i(\ell))(z-i'(\ell))}} \\
        &\leq \max_{\substack{\ell \in Y \cap U \\ \abs{i(\ell)} > R}} \abs{\frac{i(\ell)-i'(\ell)}{i(\ell)i'(\ell)}} = \max_{\substack{\ell \in Y \cap U \\ \abs{i(\ell)} > R}} \abs{\frac{1}{i'(\ell)}-\frac{1}{i(\ell)}} < \frac{1}{R}\,.
    \end{align*}
    Hence, as $R \to +\infty$, and as $(U',i') \to (U,i)$ in $\overline{\Psi}^r$, we have that $e_{i'(Y \cap U')}(z) \to e_{i(Y \cap U)}(z)$ uniformly on $\abs{z} \leq \epsilon$ and thus on $\abs{z} \leq 1$.
    \item Since $\Xi_Y$ is surjective, we can let $i : U \to \CCi$ have $i(Y \cap U) = \Lambda$.
    Let $D \defeq \min_{0 \neq \lambda \in \Lambda} \abs{\lambda}$, let $R > D$ be large, let $a \in A$ with $\abs{a} > 1$, and let $\epsilon > 0$ be small; in particular, let $\epsilon < \abs{a}^{-1} D$.
    Then by \cref{prop:dd_LL_radius}, we have that for $\Lambda'$ sufficiently close to $\Lambda$, $\abs{e_{\Lambda'}(z)-e_\Lambda(z)} < \epsilon$ for all $\abs{z} < R$.
    Thus for each $\lambda' \in \Lambda'$ with $\abs{\lambda'} < R$, we have that $\abs{e_\Lambda(\lambda')} < \epsilon$; thus by \cref{prop:e_Lambda_props} item (3) we must have $\abs{\lambda'-\lambda} < \epsilon$ for some $\lambda \in \Lambda$; since $\epsilon < D$, this $\lambda$ is unique.
    Now assume for contradiction that there are two distinct $\lambda_1', \lambda_2' \in \Lambda'$ with $\abs{\lambda_1'-\lambda} < \epsilon > \abs{\lambda_2'-\lambda}$; then $0 \neq \lambda_3' \defeq \lambda_1'-\lambda_2'$ satisfies $\abs{\lambda_3'} < \epsilon$.
    Letting $n \in \mathbb{N}$ be the smallest such that $\abs{a^n\lambda_3'} \geq \epsilon$, we have that $a^n \lambda_3' \in \Lambda'$ and $\abs{a^n\lambda_3'} < \abs{a}\epsilon < D < R$.
    Similarly to before, since $\abs{e_\Lambda(a^n\lambda_3')} < \epsilon$, there must be a $\lambda \in \Lambda$ such that $\abs{\lambda-a^n\lambda_3'} < \epsilon$.
    Since $\abs{a^n\lambda_3'} \geq \epsilon$, we cannot have $\lambda = 0$, but then $\abs{\lambda} \geq D$ and $\abs{a^n\lambda_3'} < D$ give $\abs{\lambda-a^n\lambda_3'} \geq D > \epsilon$, a contradiction.\footnote{We leave the proof that the assignment $\lambda' \to \lambda$ is $A$-linear to the reader.}
    As a corollary, $\min_{0 \neq \lambda' \in \Lambda'} \abs{\lambda'}$ is also equal to $D$.
    With this corollary, we can argue similarly that for each $\lambda \in \Lambda$ there is a unique $\Lambda' \in \Lambda'$ with $\abs{\lambda-\lambda'} < \epsilon$, and so we have a bijection between $\setst{\lambda \in \Lambda}{\abs{\lambda} < R}$ and $\setst{\lambda' \in \Lambda'}{\abs{\lambda'} < R}$.
    
    We then partially define $i'$ on $Y \cap U$ (and thus on $U$ by $F$-linearity) for $\Lambda'$ sufficiently close to $\Lambda$ by $i'(\ell) = \lambda'$ for $i(\ell) = \lambda$ with $\abs{i(\ell)} < R$, and we can extend the definition of $i'$ arbitrarily to a subspace $U' \subseteq U$ to account for $\lambda' \in \Lambda' \setminus F\setst{\lambda' \in \Lambda'}{\abs{\lambda'} < R}$.
    Since $\abs{i(\ell)-i'(\ell)} < \epsilon$, sending $\epsilon \to 0$ we get that $i'(\ell) \to i(\ell)$ for $\ell \in Y \cap U$.
    Moreover, since $\abs{i'(\ell)} > R$ for $\ell \in Y \setminus U$, we have that $\abs{i'(\ell)} \to \infty$ uniformly over such $\ell$.
    Thus $i' \to i$ as per \cite[\S 1.3.3.]{gekelerHigherVII}.
    \qedhere
\end{enumerate}
\end{proof}

\paragraph{}
Since Gekeler in \cite[\S 2.2, \S 1.12]{gekelerHigherVII} defines modular forms of weight $k$ (of type $0$) for $\GL{Y}$ as holomorphic functions on $\overline{\Psi^r}$ which are homogeneous of degree $-k$ and continuous on the boundary, as do we, it follows that
\begin{proposition}
    There is a bijection between our modular forms on $\overline{C} = C \cup \LL^{\leq r-1}$ and the space $\mathbf{Mod}_{k,0}(\GL{Y})$ of Gekeler's modular forms for $\GL{Y}$ of type $0$ on $\overline{\Psi^r}$.
\end{proposition}
Since the $\size\Cl(F)$ irreducible components $C$ of $\LL^r$ share the same boundary in our definitions, we also have the following
\begin{theorem}
    Let $(Y_c)_{c \in \Cl(F)}$ be a set of $A$-lattices in $F^r$ with $\pi(Y_c) = c$ for each $c \in \Cl(F)$, and let $k$ be a nonnegative integer.
    Then there is an injection
    \begin{align*}
        \StrongMF^{k,r} &\longinto \prod_{c \in \Cl(F)} \mathbf{Mod}_{k,0}(\GL{Y_c}) \\
        f : \LL^{\leq r} \to \CCi &\longmapsto (f\circ\Xi_{Y_c})_{c \in \Cl(F)}
    \end{align*}
    from our modular forms of weight $k$ for $\LL^{\leq r}$ to a tuple of Gekeler modular forms of weight $k$ and type $0$.
    Moreover, when restricted to cusp forms, this is a bijection.

    These induce an injection and a bijection of the algebras
    \[ \StrongMF^r \longinto \prod_{c \in \Cl(F)} \mathbf{Mod}^0(\GL{Y_c}) \;\;\:\text{and}\;\;\: \CuspMF^r \longionto \prod_{c \in \Cl(F)} \mathbf{Mod}^{0,cusp}(\GL{Y_c}) \]
    of our modular forms to Gekeler modular forms respectively.
\end{theorem}

Considering the level $N$ introduces additional technical considerations.
We trust that the reader will be convinced of the following
\toGekelerLevelN

\begin{proof}[Proof sketch]
That $\Xi_{Y_c}(i)$ is an appropriate lattice and that the topologies are compatible follows from \cref{prop:Gek:isotopology} and the fact that in $\StrongMF_N^{k,r}$ the different components share a boundary.

To show that $t \circ i^{-1}$ is an $r$-inverse level structure for an injection $i : U \into \CCi$ and an bijection $t : N^{-1} Y_c/Y_c \ionto (N^{-1}/A)^r$, note first that for $\Lambda = \Xi_{Y_c}(i)$, we have that $N^{-1} \Lambda = N^{-1} i(Y_c \cap U) = i(N^{-1} Y_c \cap U)$ since $i$ is $A$-linear, so that $N^{-1}\Lambda / \Lambda = i((N^{-1} Y_c/Y_c) \cap U)$.
Thus $i^{-1}$ is an $A$-linear bijection from $N^{-1} \Lambda / \Lambda$ to $(N^{-1}Y_c/Y_c) \cap U \subseteq N^{-1}Y_c/Y_c$, so that $t \circ i^{-1}$ is an $A$-linear injection from $N^{-1} \Lambda/\Lambda$ to $(N^{-1}/A)^r$.

The injectivity of the above maps follows from the fact that the map
\[ X_N^r : (c,t,i) \longmapsto (\Xi_{Y_c}(i), t \circ i^{-1}) \in \LLNRi \] for $c \in \Cl(F)$, $t \in \Bij(Y_c)$, and $i : U \to \CCi$ is surjective, and the injectivity when restricted to cusp forms follows since it is injective on the main rank $r$ stratum (when $U = F^r$).
\end{proof}

%% file: text/6__conclusion.tex
\section{Conclusion} \label{sec:conclusion}

In this paper we have presented a theory of modular forms of arbitrary finite rank $r$, similarly to the work of Gekeler and of Basson, Breuer, and Pink.
In contrast to these other works, ours does not make much use of rigid analysis and uses continuity in a metric space to define when a weak modular form is in fact a strong modular form.
Hence it may be more accessible to those unfamiliar with rigid analysis.
We have also introduced actions of $\GL[r]{A/N}$ and $\cJ(A)$ on these modular forms, which as far as we can tell is novel and will hopefully have effects on the general theory of modular forms.

There are a number of directions in which this work may be extended, which we hope will be investigated in future:
\begin{itemize}
  \item Dimension formulae for the vector spaces of modular forms of weight $k$.
  \item A characterisation of when the action of $\cJ(A)$ on $\LL_N^r$ extends to a homeomorphism on $\LLNRi$ and thus to an action on modular forms for $K(N)$.
  \item Generalising the principal congruence subgroup $K(N)$ to a general compact open subgroup $\mathcal{K}$ of $\GL[r]{\finadele}$.
  \item Including the type $m \in \ZZ$ of a modular form as in \cite{BBP} and in \cite{gekelerHigherVII}.
\end{itemize}

%% file: phdthesis.bib
@phdthesis{baker2020lattice,
  author     = {Baker, Liam},
  school     = {Stellenbosch University},
  title      = {{D}rinfeld modular forms of higher rank from a lattice-oriented point of view},
  year       = {2020},
  eprinttype = {hdl},
  eprint     = {10019.1/108242},
  note       = {\url{http://hdl.handle.net/10019.1/108242}},
  url        = {http://hdl.handle.net/10019.1/108242}
}

@phdthesis{basson2014coefficients,
  title      = {On the coefficients of {D}rinfeld modular forms of higher rank},
  author     = {Basson, Dirk},
  year       = {2014},
  school     = {Stellenbosch University},
  eprinttype = {hdl},
  eprint     = {10019.1/86387},
  note       = {\url{http://hdl.handle.net/10019.1/86387}}
}

@article{basson2017product,
  title     = {A product formula for the higher rank {D}rinfeld discriminant function},
  author    = {Basson, Dirk},
  journal   = {Journal of Number Theory},
  volume    = {178},
  pages     = {190--200},
  year      = {2017},
  publisher = {Elsevier},
  doi       = {10.1016/j.jnt.2017.02.010},
  note      = {\url{https://doi.org/10.1016/j.jnt.2017.02.010}}
}

@article{basson2017certain,
  title      = {On certain {D}rinfeld modular forms of higher rank},
  author     = {Basson, Dirk and Breuer, Florian},
  journal    = {Journal de Th{\'e}orie des Nombres de Bordeaux},
  volume     = {29},
  number     = {3},
  pages      = {827--843},
  year       = {2017},
  publisher  = {JSTOR},
  eprint     = {26274100},
  eprinttype = {jstor},
  note       = {\url{https://www.jstor.org/stable/26274100}}
}

@article{BBPI,
  shorthand   = {BBP1},
  title       = {Drinfeld modular forms of arbitrary rank, {P}art {I}: {A}nalytic {T}heory},
  author      = {Basson, Dirk and Breuer, Florian and Pink, Richard},
  journal     = {arXiv preprint},
  year        = {2018},
  pages       = {1--24},
  eprint      = {1805.12335},
  eprinttype  = {arXiv},
  note        = {\url{https://arxiv.org/abs/1805.12335}},
  eprintclass = {math.NT}
}

@article{BBPII,
  shorthand     = {BBP2},
  title         = {Drinfeld modular forms of arbitrary rank, {P}art {II}: {C}omparison with {A}lgebraic {T}heory},
  author        = {Basson, Dirk and Breuer, Florian and Pink, Richard},
  journal       = {arXiv preprint},
  year          = {2018},
  pages         = {1--29},
  eprint        = {1805.12337},
  note          = {\url{https://arxiv.org/abs/1805.12337}},
  archiveprefix = {arXiv},
  primaryclass  = {math.NT}
}

@article{BBPIII,
  shorthand     = {BBP3},
  title         = {Drinfeld modular forms of arbitrary rank, {P}art {III}: {E}xamples},
  author        = {Basson, Dirk and Breuer, Florian and Pink, Richard},
  journal       = {arXiv preprint},
  year          = {2018},
  pages         = {1--30},
  eprint        = {1805.12339},
  note          = {\url{https://arxiv.org/abs/1805.12339}},
  archiveprefix = {arXiv},
  primaryclass  = {math.NT}
}

@article{BBP,
  author  = {Basson, Dirk and Breuer, Florian and Pink, Richard},
  journal = {Memoirs of the American Mathematical Society},
  number  = {1531},
  title   = {{D}rinfeld {M}odular {F}orms of {A}rbitrary {R}ank},
  volume  = {304},
  year    = {2024},
  doi     = {10.1090/memo/1531},
  note    = {\url{https://doi.org/10.1090/memo/1531}},
  pages   = {1-77}
}

@article{breuer2009drinfeld,
  title     = {Drinfeld modular polynomials in higher rank},
  author    = {Breuer, Florian and Rück, Hans-Georg},
  journal   = {Journal of Number Theory},
  volume    = {129},
  number    = {1},
  pages     = {59--83},
  year      = {2009},
  publisher = {Elsevier},
  doi       = {10.1016/j.jnt.2008.07.010},
  note      = {\url{https://doi.org/10.1016/j.jnt.2008.07.010}}
}

@article{breuer2010torsion,
  title     = {Torsion bounds for elliptic curves and {D}rinfeld modules},
  author    = {Breuer, Florian},
  journal   = {Journal of Number Theory},
  volume    = {130},
  number    = {5},
  pages     = {1241--1250},
  year      = {2010},
  publisher = {Elsevier},
  doi       = {10.1016/j.jnt.2009.11.009},
  note      = {\url{https://doi.org/10.1016/j.jnt.2009.11.009}}
}

@inproceedings{cornelissen1997survey,
  title        = {A survey of {D}rinfeld modular forms},
  author       = {Cornelissen, Gunther},
  booktitle    = {Proceedings of the workshop on {D}rinfeld modules, modular schemes and applications},
  editor       = {Gekeler, Ernst-Ulrich and van der Put, M and Reversat, M and Van Geel, J},
  pages        = {167--187},
  year         = {1997},
  organization = {World Scientific},
  note         = {\url{https://doi.org/10.1142/9789814529990}},
  doi          = {10.1142/9789814529990}
}

@article{drinfeld1974english,
  shorthand    = {Dri|en},
  title        = {Elliptic modules},
  author       = {Drinfel'd, V.G.},
  year         = {1974},
  month        = {4},
  journal      = {Mathematics of the USSR-Sbornik},
  shortjournal = {Math. USSR-Sb.},
  volume       = {23},
  number       = {4},
  pages        = {561--592},
  language     = {English},
  publisher    = {American Mathematical Society},
  doi          = {10.1070/sm1974v023n04abeh001731},
  note         = {\url{https://doi.org/10.1070/sm1974v023n04abeh001731}}
}

@article{gekelerHigherI,
  shorthand = {GHR|1},
  author    = {Gekeler, Ernst-Ulrich},
  title     = {On {D}rinfeld modular forms of higher rank},
  journal   = {Journal de Théorie des Nombres de Bordeaux},
  publisher = {Société Arithmétique de Bordeaux},
  volume    = {29},
  number    = {3},
  year      = {2017},
  pages     = {875--902},
  doi       = {10.5802/jtnb.1005},
  note      = {\url{https://doi.org/10.5802/jtnb.1005}}
}

@article{gekelerHigherII,
  shorthand = {GHR|2},
  title     = {On {D}rinfeld modular forms of higher rank {II}},
  journal   = {Journal of Number Theory},
  year      = {2019},
  pages     = {4--32},
  issn      = {0022-314X},
  doi       = {10.1016/j.jnt.2018.11.011},
  note      = {\url{https://doi.org/10.1016/j.jnt.2018.11.011}},
  author    = {Gekeler, Ernst-Ulrich}
}

@article{gekelerHigherIII,
  shorthand = {GHR|3},
  title     = {On {D}rinfeld modular forms of higher rank {III}: {T}he analogue of the $k/12$-formula},
  journal   = {Journal of Number Theory},
  volume    = {192},
  pages     = {293--306},
  year      = {2018},
  issn      = {0022-314X},
  doi       = {10.1016/j.jnt.2018.04.018},
  note      = {\url{https://doi.org/10.1016/j.jnt.2018.04.018}},
  author    = {Gekeler, Ernst-Ulrich}
}

@article{gekelerHigherIV,
  shorthand = {GHR|4},
  title     = {On {D}rinfeld modular forms of higher rank {IV}: {M}odular forms with level},
  journal   = {Journal of Number Theory},
  year      = {2019},
  pages     = {33-74},
  issn      = {0022-314X},
  doi       = {10.1016/j.jnt.2019.04.019},
  note      = {\url{https://doi.org/10.1016/j.jnt.2019.04.019}},
  author    = {Gekeler, Ernst-Ulrich}
}

@article{gekelerHigherV,
  shorthand = {GHR|5},
  title     = {On {D}rinfeld modular forms of higher rank {V}: The behavior of distinguished forms on the fundamental domain},
  journal   = {Journal of Number Theory},
  volume    = {222},
  pages     = {75-114},
  year      = {2021},
  issn      = {0022-314X},
  doi       = {\url{https://doi.org/10.1016/j.jnt.2020.10.007}},
  note      = {\url{https://doi.org/10.1016/j.jnt.2020.10.007}},
  url       = {\url{https://www.sciencedirect.com/science/article/pii/S0022314X20303322}},
  author    = {Gekeler, Ernst-Ulrich}
}

@article{gekelerHigherVI,
  shorthand = {GHR|6},
  title     = {On {D}rinfeld modular forms of higher rank {VI}: The simplicial complex associated with a coefficient form},
  journal   = {Journal of Number Theory},
  volume    = {252},
  pages     = {326-378},
  year      = {2023},
  issn      = {0022-314X},
  doi       = {\url{https://doi.org/10.1016/j.jnt.2023.05.004}},
  note      = {\url{https://doi.org/10.1016/j.jnt.2023.05.004}},
  url       = {\url{https://www.sciencedirect.com/science/article/pii/S0022314X23001191}},
  author    = {Gekeler, Ernst-Ulrich}
}

@article{gekelerHigherVII,
  shorthand = {GHR|7},
  title     = {On {D}rinfeld modular forms of higher rank {VII}: {E}xpansions at the boundary},
  journal   = {Journal of Number Theory},
  volume    = {269},
  pages     = {260-340},
  year      = {2025},
  issn      = {0022-314X},
  doi       = {10.1016/j.jnt.2024.09.015},
  note      = {\url{https://doi.org/10.1016/j.jnt.2024.09.015}},
  url       = {\url{https://www.sciencedirect.com/science/article/pii/S0022314X24002269}},
  author    = {Gekeler, Ernst-Ulrich}
}

@article{gekeler1999survey,
  title   = {A survey on {D}rinfeld modular forms},
  author  = {Gekeler, Ernst-Ulrich},
  journal = {Turkish J. Math},
  volume  = {23},
  number  = {4},
  pages   = {485--518},
  note    = {\url{https://journals.tubitak.gov.tr/cgi/viewcontent.cgi?article=2905&context=math}},
  year    = {1999}
}

@article{goss1980eisenstein,
  title     = {$\pi$-adic {E}isenstein series for function fields},
  author    = {Goss, David},
  journal   = {Compositio Mathematica},
  volume    = {41},
  number    = {1},
  pages     = {3--38},
  year      = {1980},
  publisher = {Sijthoff \& Noordhoff International Publishers},
  url       = {http://www.numdam.org/item/CM_1980__41_1_3_0},
  note      = {\url{http://www.numdam.org/item/CM_1980__41_1_3_0}}
}

@inproceedings{goss1992integrals,
  title     = {Some integrals attached to modular forms in the theory of function fields},
  author    = {Goss, David},
  booktitle = {The Arithmetic of Function Fields},
  editor    = {Goss, David and Rosen, Michael I and Hayes, David R},
  series    = {Ohio State University Mathematical Research Institute publications},
  pages     = {227--251},
  year      = {1992},
  publisher = {Walter de Gruyter},
  note      = {\url{https://doi.org/10.1515/9783110886153}},
  url       = {https://doi.org/10.1515/9783110886153},
  doi       = {10.1515/9783110886153}
}

@article{gekeler2026expansions,
  title={Modular forms for $\mathrm{GL}(r,\mathbb{F}_q[T])$: $t$-expansions of the basic forms},
  author={Gekeler, Ernst-Ulrich},
  journal={The Ramanujan Journal},
  volume={69},
  number={4},
  pages={97},
  year={2026},
  publisher={Springer},
  note={\url{https://doi.org/10.1007/s11139-026-01359-9}},
}

@article{hubschmid2013andre,
  title     = {The {A}ndré--{O}ort conjecture for {D}rinfeld modular varieties},
  author    = {Hubschmid, Patrik},
  journal   = {Compositio Mathematica},
  volume    = {149},
  number    = {4},
  pages     = {507--567},
  year      = {2013},
  publisher = {London Mathematical Society},
  doi       = {10.1112/S0010437X12000681},
  note      = {\url{https://doi.org/10.1112/S0010437X12000681}}
}

@article{kapranov1988english,
  shorthand    = {Kap|en},
  title        = {On cuspidal divisors on the modular varieties of elliptic modules},
  author       = {Kapranov, Mikhail Mikhailovich},
  journal      = {Mathematics of the USSR-Izvestiya},
  shortjournal = {Math. USSR-Izv.},
  volume       = {30},
  number       = {3},
  pages        = {533--547},
  year         = {1988},
  month        = {6},
  language     = {English},
  doi          = {10.1070/im1988v030n03abeh001029},
  note         = {\url{https://doi.org/10.1070/im1988v030n03abeh001029}}
}

@incollection{katz1973p,
  title     = {P-adic properties of modular schemes and modular forms},
  author    = {Katz, Nicholas M},
  booktitle = {Modular functions of one variable III},
  editor    = {Kuijk, Willem and Serre, Jean-Pierre},
  pages     = {69--190},
  year      = {1973},
  doi       = {10.1007/978-3-540-37802-0_3},
  note      = {\url{https://doi.org/10.1007/978-3-540-37802-0_3}},
  publisher = {Springer},
  isbn      = {978-3-540-06483-1},
  series    = {Lecture Notes in Mathematics},
  volume    = {350}
}

@article{Pink2013compactification,
  title     = {Compactification of {D}rinfeld modular varieties and {D}rinfeld modular forms of arbitrary rank},
  author    = {Pink, Richard},
  journal   = {Manuscripta mathematica},
  volume    = {140},
  number    = {3-4},
  pages     = {333--361},
  year      = {2013},
  publisher = {Springer},
  doi       = {10.1007/s00229-012-0544-3},
  note      = {\url{https://doi.org/10.1007/s00229-012-0544-3}}
}

@article{Pink2019Areciprocal,
  title   = {Compactification of {D}rinfeld {M}oduli {S}paces as {M}oduli {S}paces of {$A$}-{R}eciprocal {M}aps and {C}onsequences for {D}rinfeld {M}odular {F}orms},
  author  = {Pink, Richard},
  note    = {\url{https://doi.org/10.1090/jag/772}},
  journal = {Journal of Algebraic Geometry},
  year    = {2021},
  pages   = {477--527},
  volume  = {30},
  number  = {3}
}

@article{SchneiderStuhler1991cohomology,
  title        = {The cohomology of $p$-adic symmetric spaces},
  author       = {Schneider, Peter and Stuhler, Ulrich},
  journal      = {Inventiones Mathematicae},
  shortjournal = {U. Invent Math},
  volume       = {105},
  number       = {1},
  pages        = {47--122},
  year         = {1991},
  publisher    = {Springer},
  doi          = {10.1007/BF01232257},
  note         = {\url{https://doi.org/10.1007/BF01232257}}
}

@book{gekeler2006dmc,
  shorthand = {Ge|DMC},
  title     = {{D}rinfeld {M}odular {C}urves},
  author    = {Gekeler, Ernst-Ulrich},
  series    = {Lecture Notes in Mathematics},
  volume    = {1231},
  year      = {2006},
  publisher = {Springer},
  pagetotal = {108},
  doi       = {10.1007/BFb0072692},
  note      = {\url{https://doi.org/10.1007/BFb0072692}}
}

@book{goss2012basic,
  shorthand = {Go|Bas},
  title     = {Basic {S}tructures of {F}unction {F}ield {A}rithmetic},
  author    = {Goss, David},
  year      = {2012},
  publisher = {Springer Science \& Business Media},
  doi       = {10.1007/978-3-642-61480-4},
  note      = {\url{https://doi.org/10.1007/978-3-642-61480-4}},
  isbn      = {978-3-540-63541-3},
  pagetotal = {424}
}

@book{curtisreiner2006repr,
 shorthand = {CR|Rep},
    AUTHOR = {Curtis, Charles W. and Reiner, Irving},
     TITLE = {Representation theory of finite groups and associative algebras},
      NOTE = {Reprint of the 1962 original. \url{https://doi.org/10.1090/chel/356}},
 PUBLISHER = {AMS Chelsea Publishing, Providence, RI},
      YEAR = {2006},
     PAGES = {xiv+689},
      ISBN = {0-8218-4066-5},
       DOI = {10.1090/chel/356},
       URL = {https://doi.org/10.1090/chel/356},
}
